\documentclass{article}

\usepackage{amsmath}
\usepackage{amssymb}
\usepackage{epsfig}
\usepackage{graphicx}
\usepackage{tikz,pgf}
\usepackage{amsfonts}
\usepackage{fancyhdr}
\usepackage{pseudocode}

\usepackage{multicol}

\usepackage[ruled,lined,algonl,boxed]{algorithm2e}

\begin{document}
%
% Theorem like environments
%
\newtheorem{construction}{Construction}%[section]
\newtheorem{theorem}[construction]{Theorem}
\newtheorem{maintheorem}[construction]{Main Theorem}
\newtheorem{KeyTheorem}[construction]{Key Theorem}
\newtheorem{corollary}[construction]{Corollary}
\newtheorem{lemma}[construction]{Lemma}
\newtheorem{prop}[construction]{Proposition}
\newtheorem{definition}[construction]{Definition}
%
% Proof environment
%
\newenvironment{proof}%
{\noindent{\bf Proof.}\rm}%
{\hfill $\square$\medskip}%
%
%%%%%%%%
%%%
\newenvironment{comment}
	{\color{blue!70!black}}
    {}
%
% Special commands for this manuscript.
%
\newcommand{\G}[1]{\ensuremath{\mathcal{G}_{#1}}}
\renewcommand{\H}[1]{\ensuremath{\mathcal{H}_{#1}}}
\newcommand{\field}[1]{\ensuremath{\mathbb{F}_{#1}}}
\newcommand{\QR}[1]{\ensuremath{\mathbb{QR}_{#1}}}
\newcommand{\I}{\ensuremath{\mbox{\sc i}}}
\newcommand{\aut}[1]{\ensuremath{\mbox{\sc Aut}\left(#1\right)}}
\font\smallmath=cmmib7
\font\tinymath=cmmib5

\title{\bf Embedding $K_{3,3}$ and $K_5$ on\\ the Double Torus}  
\author{
 William L. Kocay\\
  \  University of Manitoba, Computer Science Department\\
  Winnipeg, Manitoba, R3T 2N2 Canada\\ \\
  Andrei Gagarin\\
 \  School of Mathematics, Cardiff University\\
  Cardiff, CF24 4AG, UK\\  
  }

%\date{}  

\maketitle

\begin{abstract}
The Kuratowski graphs $K_{3,3}$ and $K_5$ characterize planarity.
Counting distinct 2-cell embeddings of these two graphs on orientable surfaces was
previously done by using Burnside's Lemma and their automorphism groups, 
without actually constructing the embeddings.
We obtain all 2-cell embeddings of these graphs on the double torus, 
using a constructive approach.
This shows that there is a unique non-orientable 2-cell embedding of $K_{3,3}$,
$14$ orientable and $17$
non-orientable 2-cell embeddings of $K_5$ on the double torus,
which explicitly confirms the enumerative results.
As a consequence, several new polygonal representations of the double torus are presented.
%%%%%%%

\end{abstract}

%%%%%
\section{Introduction} %and $\Theta_5$}
\label{intro}
The distinct (non-isomorphic) embeddings of the non-planar Kuratowski graphs $K_{3,3}$ and $K_5$ on the torus
can be found, e.g., in~\cite{GagarinKocayNeilsen, KocayKreher}.   
We want to find explicitly all their distinct 2-cell embeddings on the double torus,
which can serve as a first step in the study of graphs embeddable on the double torus.
The number of 2-cell embeddings of these graphs on orientable surfaces was previously determined
by Mull, Rieper, and White \cite{MullRieperWhite}, and Mull \cite{Mull},
using Burnside's Lemma and automorphism groups of the graphs.
In this paper, we use a constructive approach to find the embeddings and 
to determine their orientability.
By Euler's formula for the double torus, we have $n+ f - \varepsilon = -2$, implying
that a 2-cell embedding of $K_{3,3}$ has $f=1$ face and a 2-cell embedding of $K_5$
has $f=3$ faces.
Possibly some classic results to efficiently solve
\emph{NP}-hard optimization problems for planar graphs
(e.g., see \cite{Hadlock1975})
can be extended to non-planar graphs by studying their embeddings.

Graphs we consider in this paper can contain parallel edges (i.e., be multi-graphs), but not loops. 
A $2$-cell embedding of a graph $G$ on an oriented surface is characterized by its {\em rotation system}.
Given a labelling of the vertices and edges, a rotation system consists of a 
cyclic list of the incident edges for each vertex $v$, called the {\em rotation} of $v$.   
The rotation system uniquely determines the facial boundaries, and therefore the embedding of $G$ on the surface.
If $\tau$ is a rotation system for an embedding of $G$, we denote the embedding by $G^\tau$.
Two embeddings $G^{\tau_1}$ and $G^{\tau_2}$ are isomorphic if there is a permutation
of the vertices $V(G)$ and of the edges $E(G)$ that transforms $\tau_1$ into $\tau_2$.  See~\cite{KocayKreher} for more information on graph embeddings and rotation systems.

We denote by $\Theta_m$ a multi-graph consisting of two vertices $\{u,v\}$ and a set of $m$ parallel edges between them ($m\ge 1$). This graph can be considered as a generalized theta graph.
In what follows, we first construct all $2$-cell embeddings of an auxiliary graph $\Theta_5$ on the double torus, and then derive all possible embeddings of $K_{3,3}$ and $K_5$ on this surface by expanding $\Theta_5$ and some other minors of $K_{3,3}$ and $K_5$ back to the original graphs. This is done in Sections~\ref{section-theta5}, \ref{section-K33}, and \ref{section-K5}. Section~\ref{section-reps} provides different polygonal representations of the double torus, based on the results of Sections~~\ref{section-theta5} and \ref{section-K33}. 
%%%%%%%
Section~\ref{section-conclusion} and Appendices conclude the paper, explicitly describing all rotation systems for $2$-cell embeddings of $K_5$ on the double and triple tori, and providing corrections to some numerical results from \cite{GagarinKocayNeilsen} for embeddings on the torus.

%%%%%
\section{$\Theta_5$ on the double torus}
\label{section-theta5}

The multi-graph $\Theta_5$ satisfies Euler's formula, with $f=1$ for the double torus.
Three 2-cell embeddings of $\Theta_5$ on the double torus are
shown in Figure~\ref{theta5}.
Here  the double torus is  represented by a standard octagon $a^+b^+a^-b^-c^+d^+c^-d^-$,
traversed clockwise,  
with paired sides $\{a,b,c,d\}$.  See~\cite{FrechetFan,Hilbert, Kinsey, Stillwell1, Stillwell2, WhiteBeineke} for
more information on representations of the double torus.

\begin{figure}[ht]
\begin{center}
\begin{tikzpicture}[scale=0.75, line width = 0.5]
%octagon boundary
\draw (0, -0.6)--(0, 0.6)--(1, 1.6)--(2.5, 1.6)--(3.5, 0.6)--(3.5, -0.6)--(2.5, -1.6)--(1, -1.6)--(0, -0.6);
\node  at (-0.2, 0) {$a$};
\node  at (0.3, 1.3) {$b$};
\node  at (1.8, 1.8) {$a$};
\node  at (3.3, 1.3) {$b$};
\node  at (3.7, 0) {$c$};
\node  at (3.3, -1.3) {$d$};
\node  at (1.8, -1.8) {$c$};
\node  at (0.3, -1.3) {$d$};
\node  at (1.2, 0.2) {$u$};
\node  at (2.2, -0.2) {$v$};
\node  at (0.4, 0.2) {\smallmath 1};
\node  at (2.2, 1.0) {\smallmath 1};
\node  at (0.9, 0.7) {\smallmath 2};
\node  at (3.0, 0.7) {\smallmath 2};
\node  at (1.8, 0.2) {\smallmath 3};
\node  at (1.6, -0.8) {\smallmath 4};
\node  at (3.1, 0.2) {\smallmath 4};
\node  at (0.9, -0.7) {\smallmath 5};
\node  at (3.0, -0.6) {\smallmath 5};
\draw[line width=1] (0,0)--(3.5, 0);
\draw[line width=1] (0.5, 1.1)--(1, 0)--(0.5, -1.1);
\draw[line width=1] (1,0)--(1.75, -1.6);
\draw[line width=1] (3.0, 1.1)--(2.5, 0)--(3.0, -1.1);
\draw[line width=1] (2.5,0)--(1.75, 1.6);
\draw[fill=white] (1, 0) circle[radius=0.1];
\draw[fill=white] (2.5, 0) circle[radius=0.1];
%octagon boundary
\draw (4.5, -0.6)--(4.5, 0.6)--(5.5, 1.6)--(7, 1.6)--(8, 0.6)--(8, -0.6)--(7, -1.6)--(5.5, -1.6)--(4.5, -0.6);
\node  at (4.3, 0) {$a$};
\node  at (4.8, 1.3) {$b$};
\node  at (6.3, 1.8) {$a$};
\node  at (7.8, 1.3) {$b$};
\node  at (8.2, 0) {$c$};
\node  at (7.8, -1.3) {$d$};
\node  at (6.3, -1.8) {$c$};
\node  at (4.8, -1.3) {$d$};
\node  at (5.7, 0.2) {$u$};
\node  at (6.7, -0.2) {$v$};
\node  at (4.7, 0.2) {\smallmath 1};
\node  at (6.6, 1.2) {\smallmath 1};
\node  at (4.9, 0.5) {\smallmath 2};
\node  at (7.2, 1.0) {\smallmath 2};
\node  at (5.7, 1.4) {\smallmath 2};
\node  at (5.4, 0.7) {\smallmath 3};
\node  at (7.5, 0.7) {\smallmath 3};
\node  at (6.1, -0.8) {\smallmath 4};
\node  at (7.6, 0.2) {\smallmath 4};
\node  at (5.4, -0.7) {\smallmath 5};
\node  at (7.5, -0.6) {\smallmath 5};
\draw[line width=1] (4.5,0)--(5.5, 0)--(4.5, 0.5);
\draw[line width=1] (8,0)--(7, 0)--(7.1, 1.5);
\draw[line width=1] (5, 1.1)--(5.5, 0)--(5, -1.1);
\draw[line width=1] (5.5,0)--(6.25, -1.6);
\draw[line width=1] (7.5, 1.1)--(7, 0)--(7.5, -1.1);
\draw[line width=1] (7,0)--(6.25, 1.6);
\draw[line width=1] (5.8, 1.6)--(5.3, 1.4);
\draw[fill=white] (5.5, 0) circle[radius=0.1];
\draw[fill=white] (7, 0) circle[radius=0.1];
%
%octagon boundary
\draw (9, -0.6)--(9, 0.6)--(10, 1.6)--(11.5, 1.6)--(12.5, 0.6)--(12.5, -0.6)--(11.5, -1.6)--(10, -1.6)--(9, -0.6);
\node  at (8.8, 0) {$a$};
\node  at (9.3, 1.3) {$b$};
\node  at (10.8, 1.8) {$a$};
\node  at (12.3, 1.3) {$b$};
\node  at (12.7, 0) {$c$};
\node  at (12.3, -1.3) {$d$};
\node  at (10.8, -1.8) {$c$};
\node  at (9.3, -1.3) {$d$};
\node  at (10.2, 0.2) {$u$};
\node  at (11.7, -0.2) {$v$};
\node  at (9.2, 0.2) {\smallmath 1};
\node  at (11.1, 1.2) {\smallmath 1};
\node  at (9.9, 0.7) {\smallmath 2};
\node  at (11.7, 0.9) {\smallmath 2};
\node  at (10.8, 0.2) {\smallmath 3};
\node  at (11.2, -0.6) {\smallmath 4};
\node  at (11.5, 1.3) {\smallmath 4};
\node  at (9.3, -0.5) {\smallmath 4};
\node  at (10.6, -0.8) {\smallmath 5};
\node  at (12.1, 0.2) {\smallmath 5};
\draw[line width=1] (9,0)--(10, 0);
\draw[line width=1] (12.5,0)--(11.5, 0)--(11.3, 1.6);
\draw[line width=1] (9.5, 1.1)--(10, 0);
\draw[line width=1] (10,0)--(10.75, -1.6);
\draw[line width=1] (12, 1.1)--(11.5, 0);
\draw[line width=1] (11.5, 0)--(10, 0)--(12.3, -0.8);
\draw[line width=1] (11.5,0)--(10.75, 1.6);
\draw[line width=1] (9.0, -0.3)--(9.25, -0.85);
\draw[fill=white] (10, 0) circle[radius=0.1];
\draw[fill=white] (11.5, 0) circle[radius=0.1];
\node  at (1.75, -2.5) {$\Theta_5^{\# 1}$};
\node  at (6.25, -2.5) {$\Theta_5^{\# 2}$};
\node  at (10.75, -2.5) {$\Theta_5^{\# 3}$};
\end{tikzpicture}
\caption{Three distinct 2-cell embeddings of $\Theta_5$ on the double torus. The edges
are numbered.}
\label{theta5}
\end{center}
\end{figure}
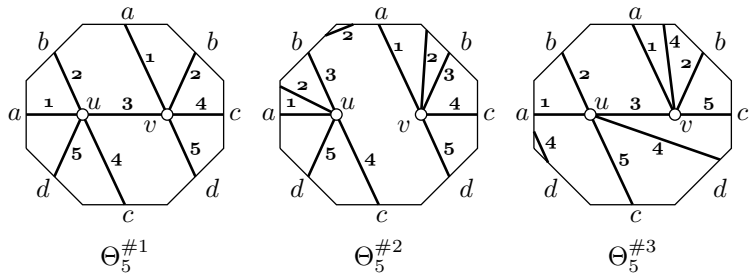

We first need to show that these comprise all distinct 2-cell embeddings of $\Theta_5$.   
They all have a single face, homeomorphic to an open disc, with facial boundary $C$
consisting of 10 edges.
The octagon representation $a^+b^+a^-b^-c^+d^+c^-d^-$  of the double torus
determines a {\em tiling of
the hyperbolic plane by regular octagons}, which is invariant under a group of translation
symmetries, see~\cite{KocayKreher, Stillwell1}.   The translations are those that
map the fundamental octagon to other octagons in the tiling, such that side $a$ is mapped to 
another side $a$,
side $b$ to another side $b$, and so forth.  
The group of translations is transitive on the octagons. When the double torus
is cut (see~\cite{FrechetFan,Hilbert}) to produce the standard octagon $a^+b^+a^-b^-c^+d^+c^-d^-$, the octagon becomes
the fundamental region of the group of translation symmetries, and
its boundary cycle $C$ corresponds to a Jordan curve in the hyperbolic plane.

In general, a Jordan curve bounding a fundamental region in the hyperbolic plane, which we will also denote by $C$, has an interior that is equivalent to the interior of the fundamental octagon, in the sense
that every point in the octagon, except points of its boundary cycle, is equivalent to exactly one point in
the interior of the Jordan curve $C$ under the action of the translation symmetries.
In other words, the interior of $C$ can also be taken as a fundamental
region for the group of translations of the hyperbolic plane.
Every 2-cell embedding of $\Theta_5$ on the double torus has an associated Jordan
curve with this property.

%%%%%%%%%%%
\begin{theorem}
\label{theta5thm}
There are exactly three distinct 2-cell embeddings of $\Theta_5$ on the double torus.
\end{theorem}
%%%%%%%%%%
\begin{proof}
There are eight octagons that meet
at each of the octagon corner vertices in the hyperbolic plane (e.g., see Figure~\ref{octagon}).
Let the vertices of $\Theta_5$ be $\{u, v\}$.  By Euler's formula, there must be just one face
in a $2$-cell embedding of $\Theta_5$ on the double torus.
%therefore no digon is a facial boundary.  
Without loss of generality, place $u$ near
that corner of the octagon where side $a$ meets side $d$,  and place $v$ near
the corner of the octagon where side $b$ meets side $c$.
In the tiling of the hyperbolic plane generated by translations of the fundamental
octagon, a copy of $u$ and $v$ will appear in each octagon.
At every corner vertex of the octagon, the sides appear in the cyclic clockwise order as 
$(b^+, a^-, b^-, a^+, d^+, c^-, d^-, c^+)$.   This is illustrated in Figure~\ref{octagon},
where the eight octagons surrounding a corner vertex are illustrated schematically.
Here copies of vertex $u$ are shaded light gray, and copies
of $v$ are shaded dark gray.

\begin{figure}[ht]
\begin{center}
\begin{tikzpicture}[scale=0.75, line width = 0.5]
%octagon boundary
\draw (0, 1.828)--(0, -1.828);
\draw (-2, 0)--(2, 0);
\draw (-1.4, 1.28)--(1.4, -1.28);
\draw (-1.4, -1.28)--(1.4, 1.28);
% draw arrows
\draw[fill=black] (0, 1.45)--(0.1, 1.3)--(0, 1.35)--(-0.1, 1.3)--(0, 1.45);
\draw[fill=black] (0, -1.45)--(0.1, -1.3)--(0, -1.35)--(-0.1, -1.3)--(0, -1.45);
\draw[fill=black] (-1.2, 0)--(-1.3, -0.1)--(-1.27, 0)--(-1.3, 0.1)--(-1.2, 0);
\draw[fill=black] (1.2, 0)--(1.3, -0.1)--(1.27, 0)--(1.3, 0.1)--(1.2, 0);
\draw[fill=black] (-1.05, -0.96)--(-1.122, -1.18)--(-1.16, -1.045)--(-1.27, -1)--(-1.05, -0.96);
\draw[fill=black] (1.05, 0.96)--(1.122, 1.18)--(1.16, 1.045)--(1.27, 1)--(1.05, 0.96);
\draw[fill=black] (-1.05, 0.96)--(-0.9, 0.95)--(-1.02, 0.99)--(-1.03, 0.80)--(-1.05, 0.96);
\draw[fill=black] (1.05, -0.96)--(0.9, -0.95)--(1.02, -0.99)--(1.03, -0.80)--(1.05, -0.96);
% arcs
\node  at (0.2, 1) {$b$};
\node  at (0.75, 0.95) {$a$};
\node  at (1, 0.25) {$b$};
\node  at (0.95, -0.6) {$a$};
\node  at (-0.2, -1) {$d$};
\node  at (-1, -0.7) {$c$};
\node  at (-1, 0.25) {$d$};
\node  at (-0.75, 0.95) {$c$};
\draw (0, 1.828) .. controls (0.5, 2.5) and (1.4, 2.2) .. (1.4, 1.28);
\node at (0.2, 2.2) {\smallmath a};
\node at (0.57, 2.38) {\smallmath b};
\node at (0.9, 2.3) {\smallmath c};
\node at (1.2, 2.15) {\smallmath d};
\node at (1.4, 1.9) {\smallmath c};
\node at (1.48, 1.6) {\smallmath d};
\draw (0, -1.828) .. controls (-0.5, -2.5) and (-1.4, -2.2) .. (-1.4, -1.28);
\node at (-0.2, -2.2) {\smallmath c};
\node at (-0.57, -2.38) {\smallmath d};
\node at (-0.9, -2.3) {\smallmath a};
\node at (-1.2, -2.15) {\smallmath b};
\node at (-1.4, -1.9) {\smallmath a};
\node at (-1.48, -1.6) {\smallmath b};
\draw (0, 1.828) .. controls (-0.5, 2.5) and (-1.4, 2.2) .. (-1.4, 1.28);
\node at (-0.2, 2.2) {\smallmath a};
\node at (-0.57, 2.38) {\smallmath b};
\node at (-0.9, 2.3) {\smallmath a};
\node at (-1.2, 2.15) {\smallmath d};
\node at (-1.4, 1.9) {\smallmath c};
\node at (-1.48, 1.6) {\smallmath d};
\draw (0, -1.828) .. controls ( 0.5, -2.5) and (1.4, -2.2) .. (1.4, -1.28);
\node at (0.2, -2.2) {\smallmath c};
\node at (0.57, -2.38) {\smallmath d};
\node at (0.9, -2.3) {\smallmath c};
\node at (1.2, -2.15) {\smallmath b};
\node at (1.4, -1.9) {\smallmath a};
\node at (1.48, -1.6) {\smallmath b};
\draw (1.4, 1.28) .. controls (2.1, 1.4) and (2.4, 0.6) .. (2, 0);
\node at (1.8, 1.4) {\smallmath b};
\node at (2.05, 1.25) {\smallmath c};
\node at (2.25, 1.08) {\smallmath d};
\node at (2.3, 0.8) {\smallmath c};
\node at (2.3, 0.5) {\smallmath d};
\node at (2.25, 0.2) {\smallmath a};
\draw (-1.4, -1.28) .. controls (-2.1, -1.4) and (-2.4, -0.6) .. (-2, 0);
\node at (-1.8, -1.4) {\smallmath d};
\node at (-2.05, -1.25) {\smallmath a};
\node at (-2.25, -1.08) {\smallmath b};
\node at (-2.3, -0.8) {\smallmath a};
\node at (-2.3, -0.5) {\smallmath b};
\node at (-2.25, -0.2) {\smallmath c};
\draw (-1.4, 1.28) .. controls (-2.1, 1.4) and (-2.4, 0.6) .. (-2, 0);
\node at (-1.8, 1.4) {\smallmath d};
\node at (-2.05, 1.25) {\smallmath c};
\node at (-2.25, 1.08) {\smallmath b};
\node at (-2.3, 0.8) {\smallmath a};
\node at (-2.3, 0.5) {\smallmath b};
\node at (-2.25, 0.2) {\smallmath a};
\draw (1.4, -1.28) .. controls (2.1, -1.4) and (2.4, -0.6) .. (2, 0);
\node at (1.8, -1.4) {\smallmath b};
\node at (2.05, -1.25) {\smallmath a};
\node at (2.25, -1.08) {\smallmath d};
\node at (2.3, -0.8) {\smallmath c};
\node at (2.3, -0.5) {\smallmath d};
\node at (2.25, -0.2) {\smallmath c};
%
%\draw[line width=1] (0,0)--(3.5, 0);
%\draw[line width=1] (0.5, 1.1)--(1, 0)--(0.5, -1.1);
%\draw[line width=1] (1,0)--(1.75, -1.6);
%\draw[line width=1] (3.0, 1.1)--(2.5, 0)--(3.0, -1.1);
%\draw[line width=1] (2.5,0)--(1.75, 1.6);
%
\draw[fill=lightgray] (0.25, -0.5) circle[radius=0.08]; % u
\draw[fill=lightgray] (1.1, 1.35) circle[radius=0.08]; % u
\draw[fill=lightgray] (1.9, 0.3) circle[radius=0.08]; % u
\draw[fill=lightgray] (1.9, -0.9) circle[radius=0.08]; % u
\draw[fill=lightgray] (-0.6, -1.95) circle[radius=0.08]; % u
\draw[fill=lightgray] (-1.7, -1) circle[radius=0.08]; % u
\draw[fill=lightgray] (-1.85, 0.2) circle[radius=0.08]; % u
\draw[fill=lightgray] (-0.9, 1.9) circle[radius=0.08]; % u
\draw[fill=darkgray] (-0.25, 0.5) circle[radius=0.08]; % v
\draw[fill=darkgray] (0.65, 2) circle[radius=0.08]; % v
\draw[fill=darkgray] (1.8, 1) circle[radius=0.08]; % v
\draw[fill=darkgray] (1.85, -0.25) circle[radius=0.08]; % v
\draw[fill=darkgray] (0.8, -1.9) circle[radius=0.08]; % v
\draw[fill=darkgray] (-1.2, -1.4) circle[radius=0.08]; % v
\draw[fill=darkgray] (-2.05, -0.42) circle[radius=0.08]; % v
\draw[fill=darkgray] (-1.8, 1) circle[radius=0.08]; % v
\end{tikzpicture}
\caption{The eight octagons meeting at one corner vertex.}
\label{octagon}
\end{center}
\end{figure}
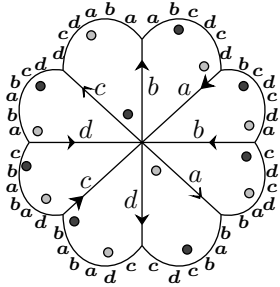

Denote by $C$ the Jordan curve in the hyperbolic plane corresponding to the 
unique facial cycle of $\Theta_5$ when embedded on the double torus with a $2$-cell face.
The five edges of $\Theta_5$ each appear twice on the boundary of $C$,
which is a 10-gon, so that the vertices of the cycle alternate $(v,u,v,u,v,u,v,u,v,u)$, and each edge is traversed exactly once in each direction. 
The translation symmetries of the hyperbolic plane will map $C$
to copies of $C$, such that their interiors are disjoint and, together with $C$ and its
translations, cover the entire hyperbolic plane.

Without loss
of generality, the interior of $C$ can be taken to
contain the central corner vertex of the diagram in Figure~\ref{octagon}.
However, it must not contain any other corner vertices, as those are obtained by translations
of the central corner vertex.  Therefore $C$ is a Jordan curve in the hyperbolic plane
whose interior contains the central corner vertex in Figure~\ref{octagon}.
A similar situation happens when graphs are embedded
on the torus.    For 
example, Figure~\ref{theta3} shows two (equivalent) embeddings
of $\Theta_3$ on the torus.  The embedding on the right is somewhat ``skewed'', crossing
the rectangle boundary several times, but
it is nevertheless an equivalent embedding with a single face (see also Figure~2 in \cite{GagarinKocayNeilsen} and corresponding explanation at the beginning of Section~2 in \cite{GagarinKocayNeilsen}).

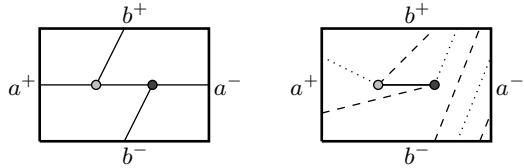
\begin{figure}[ht]
\begin{center}
\begin{tikzpicture}[scale=0.75, line width = 0.5]
% rectangles 
\draw[line width=1] (0, 0)--(3, 0)--(3, 2)--(0,2)--(0,0);
\draw[line width=1] (5, 0)--(8, 0)--(8, 2)--(5,2)--(5,0);
\coordinate (u) at (1, 1);
\coordinate (v) at (2, 1);
\coordinate (x) at (6, 1);
\coordinate (y) at (7, 1);
% draw arrows
%\draw[fill=black] (0, 1.45)--(0.1, 1.3)--(0, 1.35)--(-0.1, 1.3)--(0, 1.45);
%\draw[fill=black] (0, -1.45)--(0.1, -1.3)--(0, -1.35)--(-0.1, -1.3)--(0, -1.45);
%\draw[fill=black] (-1.2, 0)--(-1.3, -0.1)--(-1.27, 0)--(-1.3, 0.1)--(-1.2, 0);
%\draw[fill=black] (1.2, 0)--(1.3, -0.1)--(1.27, 0)--(1.3, 0.1)--(1.2, 0);
%\draw[fill=black] (-1.05, -0.96)--(-1.122, -1.18)--(-1.16, -1.045)--(-1.27, -1)--(-1.05, -0.96);
%\draw[fill=black] (1.05, 0.96)--(1.122, 1.18)--(1.16, 1.045)--(1.27, 1)--(1.05, 0.96);
%\draw[fill=black] (-1.05, 0.96)--(-0.9, 0.95)--(-1.02, 0.99)--(-1.03, 0.80)--(-1.05, 0.96);
%\draw[fill=black] (1.05, -0.96)--(0.9, -0.95)--(1.02, -0.99)--(1.03, -0.80)--(1.05, -0.96);
\draw (u)--(v);
\draw (x)--(y);
\draw (u)--(1.5, 2);
\draw (1.5, 0)--(v);
\draw (v)--(3, 1);
\draw (0, 1)--(u);
\draw (x)--(y);
\draw[dashed] (x)--(7, 2);
\draw[dashed] (7, 0)--(7.8,2);
\draw[dashed] (7.8, 0)--(8, 0.5);
\draw[dashed] (5, 0.5)--(y);
\draw[dotted] (y)--(7.4, 2);
\draw[dotted] (7.4, 0)--(8, 1.5);
\draw[dotted] (5, 1.5)--(x);
%%%
\node  at (-0.3, 1) {\small{$a^+$}};
\node  at (3.35, 1) {\small{$a^-$}};
\node  at (4.7, 1) {\small{$a^+$}};
\node  at (8.35, 1) {\small{$a^-$}};
\node  at (1.7, 2.27) {\small{$b^+$}};
\node  at (1.7, -0.23) {\small{$b^-$}};
\node  at (6.7, 2.27) {\small{$b^+$}};
\node  at (6.7, -0.23) {\small{$b^-$}};
%%%
%\node  at (0.75, 0.95) {$a$};
%\node  at (1, 0.25) {$b$};
%\node  at (0.95, -0.6) {$a$};
%\node  at (-0.2, -1) {$d$};
%\node  at (-1, -0.7) {$c$};
%\node  at (-1, 0.25) {$d$};
%\node  at (-0.75, 0.95) {$c$};
%
\draw[fill=lightgray] (u) circle[radius=0.08]; % u
\draw[fill=darkgray] (v) circle[radius=0.08]; % v
\draw[fill=lightgray] (x) circle[radius=0.08]; % x
\draw[fill=darkgray] (y) circle[radius=0.08]; % y
\end{tikzpicture}
\caption{Two equivalent embeddings of $\Theta_3$ on the torus.}
\label{theta3}
\end{center}
\end{figure}

%%%%%%%%%%
%%%%%%%%%%%%

As a result, the 10 edges of $C$ must intersect all of the
eight octagon sides incident on the central corner vertex.  
So, we can write $C=(v,s_1,u,s_2,v, \ldots, u,s_{10})$, where each $s_i$ is a label
corresponding to a traversal of an edge of the facial cycle of $\Theta_5$ --  
it represents a string of zero or more letters, corresponding to
the sides of the octagons crossed by the corresponding edge of $C$, $i=1,\ldots,10$. 
For example, if edge $i$ crosses an octagon side labelled $a$, then $s_i$ will be one of $a^+$ or
$a^-$, depending on the direction in which side $a$ is crossed. When travelling along $C$,
there is a region to the right and a region to the left.  When an edge of $C$ 
crosses a side of the octagon directed from the right to the left of $C$, this will be denoted by
$+$.  So $s_i=a^+$ means that edge $i$ of $C$ crosses the octagon side labelled $a$,
directed from the right to the left.  If the edge of $C$ crosses two or more sides of the octagon,
then $s_i$ will be a string of two or more letters, e.g., $s_i=b^+a^-$,
or  $s_i=b^+a^-c^+$, etc.  Note that an edge $uv$ of $C$ can cross any sequence
of octagon sides, as long as $C$ does not enclose any corners other than
the central corner.
%%%%%%%%
If the edge does not cross any edges of the octagon, then we write 
$s_i=\phi$.  

As each edge of $\Theta_5$ occurs twice on $C$, with
opposite orientations, there is an $s_j$ corresponding to each $s_i$ 
such that $s_j=s_i^{-1}$, where $j\not=i$. For example, if $s_i=b^+a^-c^+$, then $s_j=s_i^{-1}=c^-a^+b^-$.
Concatenating all the 
strings $s_i$ on the cycle $C$, $i=1,\ldots,10$,
gives the string $(s_1s_2\ldots s_{10})$, which must reduce to a cyclic
permutation of $(b^+a^-b^-a^+d^+c^-d^-c^+)$, namely the cyclic sequence
of octagon sides incident on the central corner vertex in Figure~\ref{octagon}.
Note that there will be at most one $s_i=\phi$ (and its inverse), for otherwise the
embedding would have a digon face, which is not the case here.
%%%%%%

Therefore, we have a cycle of ten edge labels $(s_1, s_2, \ldots, s_{10})$, 
consisting of five strings and their inverses.
Let the five strings be denoted by $\alpha, \beta, \rho, \sigma, \tau$ and their inverses by 
$\alpha^{-1}, \beta^{-1}, \rho^{-1}, \sigma^{-1}, \tau^{-1}$.
If $s_i=\alpha$, then $\alpha^{-1}$ can only be some
$s_j$ such that $j\not\equiv i\ ({\rm mod}\ 2)$, i.e., 
$i$ and $j$ must have opposite parity.  This is because each edge of 
$\Theta_5$ is traversed exactly twice on $C$, in {\em opposite} directions,
and the vertices of $C$ alternate $u,v,u,v,\ldots$. Furthermore, for
each string, for example $\alpha$, if $s_i=\alpha$, we have $\alpha^{-1}\not=s_{i\pm 1}$, 
or else $\alpha\alpha^{-1}$ or $\alpha^{-1}\alpha$ would cancel, 
which is not possible.  
Here arithmetic on the subscripts is done mod 10,
such that the result is a number in the range $1\ldots 10$.

We first show that there are at most five ways to assign the labels of 
strings to the cycle edges.  
Without loss of generality, let $s_1 = \alpha$.   Then $\alpha^{-1}$
is either $s_4$ or $s_6$.  Note that we need not consider $s_8$ as $\alpha^{-1}$, as this
is equivalent to reversing the cycle.
This gives two cases:

1) $(\alpha, s_2, s_3, \alpha^{-1}, s_5, s_6, s_7, s_8, s_9, s_{10})$

2) $(\alpha, s_2, s_3, s_4, s_5, \alpha^{-1}, s_7, s_8, s_9, s_{10})$

\noindent
We can characterize these cases by saying that in case (1) some pair 
$\alpha, \alpha^{-1}$ occurs as $s_i$ and $s_{i\pm 3}$,
and in case (2) no pair $\alpha, \alpha^{-1}$ occurs as $s_i$ and $s_{i\pm 3}$, i.e. 
every pair $\alpha, \alpha^{-1}$ occurs as $s_i$ and $s_{i\pm 5}$.
This division into two cases results in a number of subsequent cases. 
First of all, we have a unique completion for case (2):

$(\alpha, \beta, \sigma, \rho, \tau, \alpha^{-1}, \beta^{-1}, \sigma^{-1}, \rho^{-1}, \tau^{-1})$\qquad (*)

\noindent
In case (1), without loss of generality, we can choose $s_6$ as $\beta$.  Then $\beta^{-1}$ is
either $s_3$ or $s_9$, giving two sub-cases:

1a) $(\alpha, s_2, \beta^{-1}, \alpha^{-1}, s_5, \beta, s_7, s_8, s_9, s_{10})$

1b) $(\alpha, s_2, s_3, \alpha^{-1}, s_5, \beta, s_7, s_8, \beta^{-1}, s_{10})$

\noindent
In case (1a), without loss of generality, we can choose $s_8$ as $\sigma$.  Then $\sigma^{-1}$ is $s_5$, giving

1ai) $(\alpha, s_2, \beta^{-1}, \alpha^{-1}, \sigma^{-1}, \beta, s_7, \sigma, s_9, s_{10})$

\noindent
We then can choose $s_9$ as $\rho$, which requires $\rho^{-1}$ to be $s_2$. Then
$s_7$ is $\tau$ and $s_{10}$ is $\tau^{-1}$:

$(\alpha, \rho^{-1}, \beta^{-1}, \alpha^{-1}, \sigma^{-1}, \beta, \tau, \sigma, \rho, \tau^{-1})$\qquad (*)

\noindent
In case (1b), we can choose $s_7$ as $\sigma$.  Then $\sigma^{-1}$ is either 
$s_2$ or $s_{10}$:

1bi) $(\alpha, \sigma^{-1}, s_3, \alpha^{-1}, s_5, \beta, \sigma, s_8, \beta^{-1}, s_{10})$

1bii) $(\alpha, s_2, s_3, \alpha^{-1}, s_5, \beta, \sigma, s_8, \beta^{-1}, \sigma^{-1})$

\noindent
In case (1bi), we can choose $s_3$ as $\rho$.  Then $\rho^{-1}$ is either $s_8$ or $s_{10}$.
Each results in a unique completion:

$(\alpha, \sigma^{-1}, \rho, \alpha^{-1}, \tau, \beta, \sigma, \rho^{-1}, \beta^{-1}, \tau^{-1})$\qquad (*)

$(\alpha, \sigma^{-1}, \rho, \alpha^{-1}, \tau, \beta, \sigma, \tau^{-1}, \beta^{-1}, \rho^{-1})$\qquad (*)

\noindent
In case (1bii), we can choose $s_2$ as $\rho$. Then $\rho^{-1}$ is $s_5$.  There
is a unique completion:

$(\alpha,\rho, \tau, \alpha^{-1}, \rho^{-1}, \beta, \sigma, \tau^{-1}, \beta^{-1}, \sigma^{-1})$\qquad (*)

\noindent
This gives five possible solutions for labelling the facial cycle edges. 
We rename the letters so that each solution begins with $(\alpha,\beta,\rho,\alpha^{-1},\ldots)$, except for the first one found.  The results are:

\medbreak
$A=(\alpha,\beta,\rho,\sigma,\tau,\alpha^{-1},\beta^{-1},\rho^{-1},\sigma^{-1},\tau^{-1})$ 

$B= (\alpha,\beta,\rho,\alpha^{-1},\sigma,\rho^{-1},\tau,\sigma^{-1},\beta^{-1},\tau^{-1})$ 

$F= (\alpha, \beta, \rho, \alpha^{-1}, \sigma, \tau, \beta^{-1}, \rho^{-1}, \tau^{-1}, \sigma^{-1})$ 

$D= (\alpha, \beta, \rho, \alpha^{-1}, \sigma, \tau, \beta^{-1}, \sigma^{-1}, \tau^{-1}, \rho^{-1})$

$E= (\alpha,\beta,\rho, \alpha^{-1}, \beta^{-1}, \sigma, \tau, \rho^{-1}, \sigma^{-1}, \tau^{-1})$

\medbreak
We first show that cyclic sequences $A, B$, and $E$ can be realized by suitable choices
of strings for labels. If we choose $\alpha=b^+, \beta=a^-, \rho=\phi, \sigma=d^+, \tau=c^-$,
we obtain
$E=\allowbreak (b^+, a^-, \phi, b^-, a^+, d^+, c^-, \phi, d^-, c^+)$,
which reduces to $(b^+a^-b^-a^+d^+c^-d^-c^+)$, as required.
This solution is identical to the embedding $\Theta_5^{\#1}$ in Figure~\ref{theta5},
starting with $s_4$.

If we choose 
$\alpha=b^-, \beta=b^+a^+, \rho=a^-, \sigma=d^-, \tau=c^+$,
we obtain\\
\noindent
$A=\allowbreak (b^-, b^+a^+, a^-, d^-, c^+, b^+, a^-b^-, a^+, d^+, c^-)$,
which also reduces to\\
\noindent
$(b^+a^-b^-a^+d^+c^-d^-c^+)$. This solution is identical to the 
embedding $\Theta_5^{\#2}$ in Figure~\ref{theta5},
starting with $s_1$.

If we choose 
$\alpha=b^+, \beta=a^-, \rho=\phi, \sigma=a^+d^+, \tau=c^-$,
we obtain\\
\noindent
$B=\allowbreak (b^+, a^-, \phi, b^-, a^+d^+, \phi, c^-,d^-a^-, a^+, c^+)$,
which also reduces to \\
\noindent
$(b^+a^-b^-a^+d^+c^-d^-c^+)$. 
This solution is identical to the 
embedding $\Theta_5^{\#3}$ in Figure~\ref{theta5},
starting with $s_1$.

Therefore these three choices of strings for labels give three embeddings of $\Theta_5$,
each with one face.  These three embeddings are pairwise non-isomorphic. For any string $s_i$ 
occurring in $A$, its inverse is $s_{i+5}$, which is not the case with $B$ or $E$.
And for any string $s_i$ occurring in $B$, its inverse is 
either $s_{i+3}$ or $s_{i-3}$, which is not the case with $E$.

Clearly, any embeddings of $\Theta_5$ corresponding to solutions $A, B$ or $E$
are isomorphic to one of these three, because the facial cycle determines
the rotation system uniquely.  Therefore each of the solutions 
$A, B$ and $E$ defines an isomorphism class of embeddings.
Also, notice that renaming the strings of $D$ as
$\alpha\rightarrow\beta\rightarrow\rho\rightarrow\alpha^{-1}$
converts $D$ (starting at $s_{10}$) into $E$.
Therefore labellings $D$ and $E$ are equivalent.

We show that the remaining solution $F$
cannot be realized by an embedding. 
Consider the
cyclic product of strings $(s_1\ldots s_{10})$.   This must reduce to
$(b^+a^-b^-a^+d^+c^-d^-c^+)$.  Notice that the latter string has a decomposition
into a substring based on $\{a,b\}$ and another substring based on $\{c,d\}$,
and that they are disjoint.  The letter diametrically opposite any letter
$a$ or $b$ is $c$ or $d$, respectively, and conversely.  Consider the three edges of $F$ labelled $s_4, s_5, s_6$. 
Their inverses are the edges labelled $s_1,s_{10},s_9$, respectively.
Suppose $s_5=qrs$, where $q$ is the longest prefix that cancels with a suffix of $s_4$,
and $s$ is the longest suffix of $s_5$ that cancels with a prefix of $s_6$.
So, $s_4=pq^{-1}$ and $s_6=s^{-1}t$ for suitable strings $p,q,r,s,t$.
Then $s_4s_5s_6=pq^{-1}qrss^{-1}t = prt\not=\phi$.  
The diametrically opposite string is 
$s_9s_{10}s_1=s_6^{-1}s_5^{-1}s_4^{-1}=t^{-1}ss^{-1}r^{-1}q^{-1}qp^{-1}=t^{-1}r^{-1}p^{-1}=(prt)^{-1}$,
which contains the same letters as $prt$, a contradiction.

It follows that there are exactly three distinct 2-cell embeddings of $\Theta_5$ on the double torus, shown in Figure~\ref{theta5}.
\end{proof}

%%%%%%%%%%%

\subsection{Automorphisms, Rotations, Orientability}
The automorphisms of the embeddings of Figure~\ref{theta5}  will be helpful
in finding the embeddings of $K_{3,3}$.   We will need the rotations at
vertices $u$ and $v$ for each embedding, written explicitly below.  
The numbers below are the edge numbers
in the embeddings of Figure~\ref{theta5}. 
An automorphism is a permutation of vertices and/or edges that leaves the rotation system unchanged.
An embedding is \emph{non-orientable} if the embedding obtained by reversing
the rotations is isomorphic to the original embedding.  Otherwise it is
\emph{orientable}.
\medbreak
$\Theta_5^{\# 1}$\qquad\qquad
$u: (1,2,3,4,5)\qquad v: (1,2,4,5,3)$
\medbreak
$\Theta_5^{\# 2}$\qquad\qquad
$u: (1,2,3,4,5)\qquad v: (1,2,3,4,5)$
\medbreak

$\Theta_5^{\# 3}$\qquad\qquad
$u: (1,2,3,4,5)\qquad v: (1,4,2,5,3)$
\medbreak

%%%%%%%%%%%%%%

\begin{lemma}
\label{autotheta5_1}
Embedding $\Theta_5^{\# 1}$ of Figure~\ref{theta5} has the automorphism group of order two and is non-orientable.
\vspace{2mm}

\begin{proof}
The permutation $(u,v)(1,4)(2,5)$ interchanges the rotations of $u$ and $v$, and
so is an automorphism.  It is the only possible non-trivial automorphism due to the unique position of
edge $3$ in the facial cycle: this is the only edge whose opposite traversal is diametrically opposite to 
itself on the cycle (see Figure~\ref{hyperbolic}).  
The permutation $(1,5)(2,4)$ maps the rotations of $u$ and
$v$ to their reversals. Therefore, $\Theta_5^{\# 1}$ is non-orientable.
\end{proof}
\end{lemma}
\medbreak

%%%%%%%%

\begin{lemma}
\label{autotheta5_2}
Embedding $\Theta_5^{\# 2}$ of Figure~\ref{theta5} has the automorphism group of order ten,
is non-orientable, and 
all its edges are equivalent.
\vspace{2mm}

\begin{proof}
It is clear from the rotations that permutations $(u,v)$ and $(1,2,3,4,5)$ are automorphisms of $\Theta_5^{\# 2}$.  
Therefore the group has order ten. The permutation $(2,5)(3,4)$ maps the
rotations to their reversals, implying the embedding is non-orientable.
\end{proof}
\end{lemma}
\medbreak

%%%%%%%%

\begin{lemma}
\label{autotheta5_3}
Embedding $\Theta_5^{\# 3}$ of Figure~\ref{theta5} has the automorphism group of order five,
is non-orientable, and
all its edges are equivalent.
\vspace{2mm}

\begin{proof}
From the rotations, we see that permutation $(1,2,3,4,5)$ is an automorphism of $\Theta_5^{\# 3}$.
This permutation generates the whole automorphism group. 
Therefore the group has order five.  The permutation $(2,5)(3,4)$ maps the
rotations to their reversals, implying the embedding is non-orientable.
\end{proof}
\end{lemma}
\medbreak
%%%%%%%%%%

\begin{corollary}
\label{autotheta5}
In each of the embeddings $\Theta_5^{\# 1}$ and $\Theta_5^{\# 2}$,
vertices $u$ and $v$ are equivalent.  They are not equivalent in $\Theta_5^{\# 3}$.
\end{corollary}
\medbreak

If $uv$ is an edge of a graph $G$, we denote the graph obtained by contracting edge $uv$
by $G\cdot uv$. 
Clearly, if $uv$ belongs to a triangle in $G$, contracting $uv$ will result in parallel edges in $G\cdot uv$, and contracting an edge $xy$ from a set of parallel edges in $G$ would create a loop in $G\cdot xy$.
We will need the following simple lemma.  

\begin{lemma}
\label{contract}
Suppose $G^\tau$ is a $2$-cell embedding of a graph $G$ (parallel edges are possible) on the double torus, no face of $G^\tau$ is a digon, and $uv$ is a non-parallel edge of $G$ which is not on a boundary of a triangular face of $G^\tau$. Then $G^\tau\cdot uv$ is a $2$-cell embedding of $G\cdot uv$ with no loops or digon facial cycles.
\vspace{2mm}

\begin{proof}
Clearly, loops in $G^\tau\cdot uv$ and $G\cdot uv$ would only be possible if $uv$ were from a set of parallel edges. So, $G^\tau\cdot uv$ and $G\cdot uv$ have no loops. Since $uv$ is not on a boundary of a triangular face of $G^\tau$, contracting $uv$ doesn't create any digon faces in $G^\tau\cdot uv$.
Clearly, contracting an edge on the boundary of a $2$-cell face in $G^\tau$ leaves the corresponding face in the resulting embedding equivalent to an open disk (no loops in $G$ here).
\end{proof}
\end{lemma}

%%%%%%%%%%

%%%%%
%%%%%%%%
\section{$K_{3,3}$ on the double torus}
\label{section-K33}

\bigbreak
Denote now by $A,B,C,D,E,F$ the vertices of $K_{3,3}$, and
choose the edges $AC, AE, BD, BF$ (drawn in bold) of $K_{3,3}$ as shown in Figure~\ref{K33}.
Contract these four edges to form
a minor of $K_{3,3}$ isomorphic to $\Theta_5$.
{\em If we start with a 2-cell embedding of $K_{3,3}$ on the double torus, after contracting the edges, the resulting embedding of the minor will be a 2-cell embedding of $\Theta_5$ on the double torus.} 
This follows from Lemma~\ref{contract} and the fact
that no facial cycle of $K_{3,3}$ is a digon or triangle, as there is only one face in the embedding.
Therefore, every 2-cell embedding of $K_{3,3}$ on the double
torus can be contracted to a 2-cell embedding of $\Theta_5$, and $\Theta_5$ has only three 2-cell embeddings.
To find the distinct 2-cell embeddings of $K_{3,3}$, we restore the contracted edges
in all possible ways and compare the results for isomorphism.

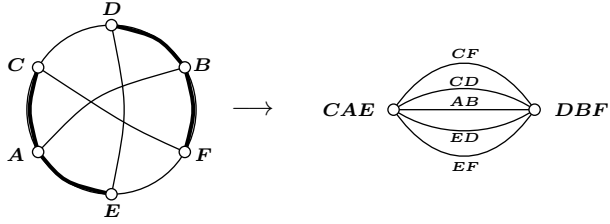
\begin{figure}[ht]
\begin{center}
\begin{tikzpicture}[scale=0.75, line width = 0.5]
\draw[fill=white] (0, 0) circle[radius=1.5];
\draw (0,1.5).. controls (0.3, 0) ..(0, -1.5);
\draw (-1.3, -0.75) .. controls (-0.3, 0.3) .. (1.3, 0.75);
\draw (-1.3, 0.75) .. controls (0.3, -0.3) .. (1.3, -0.75);
%subgraph edges
\draw[line width = 1.5] (-1.3, 0.75) .. controls (-1.5, 0) .. (-1.3, -0.75);
\draw[line width = 1.5] (0, -1.5) .. controls (-0.85, -1.3) .. (-1.3, -0.75);
\draw[line width = 1.5] (0, 1.5) .. controls (0.85, 1.3) .. (1.3, 0.75);
\draw[line width = 1.5] (1.3, 0.75) .. controls (1.5, 0) .. (1.3, -0.75);
\draw[fill=white] (0, 1.5) circle[radius=0.1];
\draw[fill=white] (0, -1.5) circle[radius=0.1];
\draw[fill=white] (-1.3, 0.75) circle[radius=0.1];
\draw[fill=white] (-1.3, -0.75) circle[radius=0.1];
\draw[fill=white] (1.3, 0.75) circle[radius=0.1];
\draw[fill=white] (1.3, -0.75) circle[radius=0.1];
\node  at (-1.7, -0.8) {\smallmath A};
\node  at (1.6, 0.8) {\smallmath B};
\node  at (-1.7, 0.8) {\smallmath C};
\node  at (0, 1.8) {\smallmath D};
\node  at (1.6, -0.8) {\smallmath F};
\node  at (0, -1.8) {\smallmath E};
\node  at (2.5, 0) {$\longrightarrow$};
\draw (5, 0)--(7.5, 0);
\draw (5, 0) .. controls (5.75, 0.5)  and (6.75, 0.5).. (7.5, 0);
\draw (5, 0) .. controls (5.75, -0.5)  and (6.75, -0.5).. (7.5, 0);
\draw (5, 0) .. controls (5.75, 1.1) and (6.75, 1.1) .. (7.5, 0);
\draw (5, 0) .. controls (5.75, -1.1) and (6.75, -1.1) .. (7.5, 0);
\draw[fill=white] (5, 0) circle[radius=0.1];
\draw[fill=white] (7.5, 0) circle[radius=0.1];
\node  at (4.2, 0) {\smallmath CAE};
\node  at (8.3, 0) {\smallmath DBF};
\node  at (6.25, 0.15) {\tinymath AB};
\node  at (6.25, 0.5) {\tinymath CD};
\node  at (6.25, 1.0) {\tinymath CF};
\node  at (6.25, -0.5) {\tinymath ED};
\node  at (6.25, -1.0) {\tinymath EF};
\end{tikzpicture}
\caption{$K_{3,3}$ and a minor isomorphic to $\Theta_5$.}
\label{K33}
\end{center}
\end{figure}

The vertices of the resulting minor $\Theta_5$ of $K_{3,3}$ will be denoted by $CAE$ and $DBF$, corresponding
to the vertices $\{A,C,E\}$ and $\{B,D,F\}$ of $K_{3,3}$ that were 
amalgamated during the edge-contractions.
The edges of $\Theta_5$ can be labelled $AB, CD, CF, ED, EF$,
in correspondence to the edges of $K_{3,3}$ they are derived from (see Figure~\ref{K33}).
There are $5! = 120$ different ways to assign the labels $AB, CD, CF, ED, EF$
to the five edges of the embeddings of $\Theta_5$ shown in Figure~\ref{theta5},
before attempting to restore the contracted edges.
However, many of them are equivalent.

\begin{lemma}
\label{autoK33}
There are eight automorphisms of $K_{3,3}$ that map the bold subgraph of Figure~\ref{K33}
to itself, generated by the permutations $(CE)$, $(DF)$, and $(AB)(CD)(EF)$.
\end{lemma}

It follows from Lemma~\ref{autoK33} that, without loss of generality, vertex 
$u$ of $\Theta_5$ can be taken
to be $CAE$, so that vertex $v$ is $DBF$.   Furthermore,
the \emph{central edge} $uv$ in the embeddings $\#1$ and $\#3$ of Figure~\ref{theta5}, 
i.e., the only edge not crossing the boundary of the octagon, can be taken to be 
either $AB$ or $CD$,
%%%%%%%
as the non-contracted edges of $K_{3,3}$ in Figure~\ref{K33}
are all equivalent to one of these.

So, we have an embedding of $\Theta_5$ of Figure~\ref{K33}, with the vertices
labelled $CAE$ and $DBF$, and the edges labelled $AB, CD, CF, ED, EF$.  The embedding
is one of $\Theta_5^{\#1}, \Theta_5^{\#2}, \Theta_5^{\#3}$ of Figure~\ref{theta5}.
Vertex $CAE$ is to be expanded into a path $[C,A,E]$, and $DBF$ is to be expanded
into a path $[D,B,F]$.   When $CAE$ is expanded, the five edges in the 
rotation of $CAE$ must be divided into two consecutive edges for $C$, one edge for $A$,
and two consecutive edges for $E$.  A similar observation holds for $DBF$.

%\midbreak
\begin{lemma}
\label{theta5lemma2}
There are no extensions of $\Theta_5^{\# 1}$ to $K_{3,3}$ with central edge $uv$ labelled $AB$.
\vspace{2mm}

\begin{proof}
Suppose that $uv$ represents edge $AB$, as in Figure~\ref{caseAB}.
There are two ways to expand vertex $u$ into the path $[C,A,E]$, having $AB$ as 
the central edge. For each of these, the edges $CD$
and $CF$ must intersect consecutive sides
of the octagon, and the edges $ED$ and $EF$
must also intersect consecutive sides
of the octagon.  One such arrangement is shown in Figure~\ref{caseAB}.
There are eight possible ways of arranging the labels of the edges incident
on $C$ and $D$ with $CD, CF, ED, EF$ satisfying this condition.

However, vertex $v$ must be expanded into a path $[D,B,F]$, such that the edges
$CD$ and $ED$ intersect consecutive sides
of the octagon, and edges $EF$ and $CF$ intersect consecutive sides.
By inspection of the eight arrangements, we see that this is not possible.
\end{proof}
\end{lemma}

\begin{figure}[ht]
\begin{center}
\begin{tikzpicture}[scale=0.75, line width = 0.5]
% octagon boundary
\draw (0, -0.6)--(0, 0.6)--(1, 1.6)--(2.5, 1.6)--(3.5, 0.6)--(3.5, -0.6)--(2.5, -1.6)--(1, -1.6)--(0, -0.6);
\node  at (-0.2, 0) {$a$};
\node  at (0.3, 1.3) {$b$};
\node  at (1.8, 1.8) {$a$};
\node  at (3.3, 1.3) {$b$};
\node  at (3.7, 0) {$c$};
\node  at (3.3, -1.3) {$d$};
\node  at (1.8, -1.8) {$c$};
\node  at (0.3, -1.3) {$d$};
%\node  at (1.2, 0.2) {\tinymath CAE};
%\node  at (2.2, -0.2) {\tinymath DBF};
\node  at (1.7, 0.15) {\tinymath AB};
\node  at (0.4, 0.15) {\tinymath CD};
\node  at (1.6, 1.2) {\tinymath CD};
\node  at (0.9, 0.7) {\tinymath CF};
\node  at (2.6, 0.85) {\tinymath CF};
\node  at (0.45, -0.5) {\tinymath ED};
\node  at (3.0, -0.5) {\tinymath ED};
\node  at (1.7, -1.0) {\tinymath EF};
\node  at (3.1, 0.15) {\tinymath EF};
\node  at (1.3, -0.18) {$u$};
\node  at (2.3, -0.18) {$v$};
\draw[line width=1] (0,0)--(3.5, 0);
\draw[line width=1] (0.5, 1.1)--(1, 0)--(0.5, -1.1);
\draw[line width=1] (1,0)--(1.75, -1.6);
\draw[line width=1] (3.0, 1.1)--(2.5, 0)--(3.0, -1.1);
\draw[line width=1] (2.5,0)--(1.75, 1.6);
\draw[fill=white] (1, 0) circle[radius=0.1];
\draw[fill=white] (2.5, 0) circle[radius=0.1];
\node  at (4.8, 0) {$\longrightarrow$};
%  second octagon
% octagon boundary
\draw (6, -0.6)--(6, 0.6)--(7, 1.6)--(8.5, 1.6)--(9.5, 0.6)--(9.5, -0.6)--(8.5, -1.6)--(7, -1.6)--(6, -0.6);
%\draw[line width=1] (6,0)--(9.5, 0);
\draw[line width=1] (6.5, 1.1)--(7, 0.5)--(7.2, 0)--(7, -0.5)--(6.5, -1.1);
\draw[line width=1] (6,0)--(7, 0.5);
\draw[line width=1] (7,-0.5)--(7.75, -1.6);
\draw[line width=1] (9.0, 1.1)--(8.5, 0)--(9.0, -1.1);
\draw[line width=1] (8.5,0)--(7.75, 1.6);
\draw[line width=1] (7.2,0)--(9.5, 0);
\node  at (5.8, 0) {$a$};
\node  at (6.3, 1.3) {$b$};
\node  at (7.8, 1.8) {$a$};
\node  at (9.3, 1.3) {$b$};
\node  at (9.7, 0) {$c$};
\node  at (9.3, -1.3) {$d$};
\node  at (7.8, -1.8) {$c$};
\node  at (6.3, -1.3) {$d$};
\node  at (7.25, 0.65) {\tinymath C};
\node  at (6.9, 0.0) {\tinymath A};
\node  at (6.7, -0.4) {\tinymath E};
\node  at (6.4, 0.45) {\tinymath CD};
\node  at (7.6, 1.2) {\tinymath CD};
\node  at (6.9, 1.0) {\tinymath CF};
\node  at (8.6, 0.85) {\tinymath CF};
\node  at (9.0, -0.5) {\tinymath ED};
\node  at (6.45, -0.75) {\tinymath ED};
\node  at (9.1, 0.15) {\tinymath EF};
\node  at (7.7, -1.0) {\tinymath EF};
\node  at (7.7, 0.15) {\tinymath AB};
\node  at (8.3, -0.2) {$v$};
\draw[fill=white] (7, 0.5) circle[radius=0.1];
\draw[fill=white] (7.2, 0) circle[radius=0.1];
\draw[fill=white] (7, -0.5) circle[radius=0.1];
\draw[fill=white] (8.5, 0) circle[radius=0.1];
\end{tikzpicture}
\caption{Case of central edge labelled $AB$, impossible to extend.}
\label{caseAB}
\end{center}
\end{figure}

\begin{lemma}
\label{CDlemma}
There are four extensions of $\Theta_5^{\#1}$ to $K_{3,3}$ with central edge $uv$ labelled $CD$.
\vspace{2mm}

\begin{proof}
Suppose that central edge $uv$ represents edge $CD$.
There are eight ways of expanding vertex $CAE$ into a path $[C,A,E]$,
some of which are shown in Figure~\ref{caseCD1}. For each of these,
edges $ED$ and $EF$ must be consecutive in the rotation of $E$, i.e. intersect consecutive sides of the octagon.
However, edges $CF, EF$  must also be consecutive in the rotation of $DBF$.  
This reduces the number of ways of expanding $u$ to a path $[C,A,E]$ to four, 
which are shown in Figure~\ref{caseCD1}.

Vertex $DBF$ must then be expanded into a path $[D,B,F]$.  In each of the four cases,
there is exactly one way to do this, as shown in Figure~\ref{caseCD2}.

\end{proof}
\end{lemma}

\begin{figure}[ht]
\begin{center}
\begin{tikzpicture}[scale=0.75, line width = 0.5]
% octagon boundary
\draw (0, -0.6)--(0, 0.6)--(1, 1.6)--(2.5, 1.6)--(3.5, 0.6)--(3.5, -0.6)--(2.5, -1.6)--(1, -1.6)--(0, -0.6);
\node  at (-0.2, 0) {$a$};
\node  at (0.3, 1.3) {$b$};
\node  at (1.8, 1.8) {$a$};
\node  at (3.3, 1.3) {$b$};
\node  at (3.7, 0) {$c$};
\node  at (3.3, -1.3) {$d$};
\node  at (1.8, -1.8) {$c$};
\node  at (0.3, -1.3) {$d$};
%\node  at (1.2, 0.2) {\tinymath CAE};
%\node  at (2.2, -0.2) {\tinymath DBF};
\node  at (1.7, 0.45) {\tinymath CD};
\node  at (0.5, 0.15) {\tinymath AB};
\node  at (1.2, 0.65) {\tinymath C};
\node  at (1.45, 0) {\tinymath A};
\node  at (1.25, -0.5) {\tinymath E};
\node  at (1.6, 1.2) {\tinymath AB};
\node  at (0.9, 1.0) {\tinymath CF};
\node  at (2.6, 0.85) {\tinymath CF};
\node  at (0.5, -0.7) {\tinymath ED};
\node  at (3.0, -0.5) {\tinymath ED};
\node  at (1.7, -1.0) {\tinymath EF};
\node  at (3.1, 0.15) {\tinymath EF};
\node  at (2.2, -0.2) {$v$};
\draw[line width=1] (0.5, 1.1)--(1, 0.5)--(1.2, 0)--(1, -0.5)--(0.5, -1.1);
\draw[line width=1] (0,0)--(1.2, 0);
\draw[line width=1] (1, 0.5)--(2.5, 0);
\draw[line width=1] (1,-0.5)--(1.75, -1.6);
\draw[line width=1] (3.0, 1.1)--(2.5, 0)--(3.0, -1.1);
\draw[line width=1] (2.5,0)--(1.75, 1.6);
\draw[line width=1] (2.5,0)--(3.5, 0);
\draw[fill=white] (1, 0.5) circle[radius=0.1];
\draw[fill=white] (1.2, 0) circle[radius=0.1];
\draw[fill=white] (1, -0.5) circle[radius=0.1];
\draw[fill=white] (2.5, 0) circle[radius=0.1];
%
%  second octagon
% octagon boundary
\draw (4.5, -0.6)--(4.5, 0.6)--(5.5, 1.6)--(7, 1.6)--(8, 0.6)--(8, -0.6)--(7, -1.6)--(5.5, -1.6)--(4.5, -0.6);
\draw[line width=1] (5, 1.1)--(5.7, 0.5)--(5.7, -0.5)--(5.3, 0)--(5, -1.1);
\draw[line width=1] (4.5,0)--(5.3, 0);
\draw[line width=1] (5.7,-0.5)--(6.25, -1.6);
\draw[line width=1] (7.5, 1.1)--(7, 0)--(7.5, -1.1);
\draw[line width=1] (7,0)--(6.25, 1.6);
\draw[line width=1] (7, 0)--(8, 0);
\draw[line width=1] (5.7, 0.5)--(7, 0);
\node  at (4.3, 0) {$a$};
\node  at (4.8, 1.3) {$b$};
\node  at (6.3, 1.8) {$a$};
\node  at (7.8, 1.3) {$b$};
\node  at (8.2, 0) {$c$};
\node  at (7.8, -1.3) {$d$};
\node  at (6.3, -1.8) {$c$};
\node  at (4.8, -1.3) {$d$};
\node  at (5.9, 0.7) {\tinymath C};
\node  at (5.95, -0.5) {\tinymath A};
\node  at (5.35, 0.25) {\tinymath E};
\node  at (6.2, -1.0) {\tinymath AB};
\node  at (4.85, 0.15) {\tinymath EF};
\node  at (4.9, -0.5) {\tinymath ED};
\node  at (6.2, 1.2) {\tinymath EF};
\node  at (5.4, 1.0) {\tinymath CF};
\node  at (7.1, 0.85) {\tinymath CF};
\node  at (7.5, -0.5) {\tinymath ED};
\node  at (7.6, 0.15) {\tinymath AB};
\node  at (6.25, 0.45) {\tinymath CD};
\node  at (6.7, -0.2) {$v$};
\draw[fill=white] (5.7, 0.5) circle[radius=0.1];
\draw[fill=white] (5.3, 0) circle[radius=0.1];
\draw[fill=white] (5.7, -0.5) circle[radius=0.1];
\draw[fill=white] (7, 0) circle[radius=0.1];
%
%  third octagon
% octagon boundary
\draw (9, -0.6)--(9, 0.6)--(10, 1.6)--(11.5, 1.6)--(12.5, 0.6)--(12.5, -0.6)--(11.5, -1.6)--(10, -1.6)--(9, -0.6);
%\draw[line width=1] (6,0)--(9.5, 0);
\draw[line width=1] (9.5, 1.1)--(10, 0.5)--(10, 0)--(10.2, -0.5);
\draw[line width=1] (10.0, 0)--(9.5, -1.1);
\draw[line width=1] (9,0)--(10, 0.5);
\draw[line width=1] (10.2, -0.5)--(10.75, -1.6);
\draw[line width=1] (10.2, -0.5)--(11.5, 0);
\draw[line width=1] (12, 1.1)--(11.5, 0)--(12, -1.1);
\draw[line width=1] (11.5, 0)--(10.75, 1.6);
\draw[line width=1] (11.5,0)--(12.5, 0);
\node  at (8.8, 0) {$a$};
\node  at (9.3, 1.3) {$b$};
\node  at (10.8, 1.8) {$a$};
\node  at (12.3, 1.3) {$b$};
\node  at (12.7, 0) {$c$};
\node  at (12.3, -1.3) {$d$};
\node  at (10.8, -1.8) {$c$};
\node  at (9.3, -1.3) {$d$};
\node  at (10.25, 0.65) {\tinymath E};
\node  at (10.2, 0.0) {\tinymath A};
\node  at (10.55, -0.55) {\tinymath C};
\node  at (10.6, 1.25) {\tinymath ED};
\node  at (11.6, 0.85) {\tinymath EF};
\node  at (12, -0.5) {\tinymath AB};
\node  at (12.1, 0.15) {\tinymath CF};
\node  at (10.75, -0.1) {\tinymath CD};
\node  at (10.7, -1.0) {\tinymath CF};
\node  at (9.4, -0.5) {\tinymath AB};
\node  at (9.9, 0.95) {\tinymath EF};
\node  at (9.4, 0.35) {\tinymath ED};
\node  at (11.3, -0.28) {$v$};
\draw[fill=white] (10, 0.5) circle[radius=0.1];
\draw[fill=white] (10, 0) circle[radius=0.1];
\draw[fill=white] (10.2, -0.5) circle[radius=0.1];
\draw[fill=white] (11.5, 0) circle[radius=0.1];
%
%  fourth octagon
% octagon boundary
\draw (13.5, -0.6)--(13.5, 0.6)--(14.5, 1.6)--(16, 1.6)--(17, 0.6)--(17, -0.6)--(16, -1.6)--(14.5, -1.6)--(13.5, -0.6);
%\draw[line width=1] (6,0)--(9.5, 0);
\draw[line width=1] (14, 1.1)--(14.5, 0.5)--(14.3, -0.1)--(14, -1.1);
\draw[line width=1] (13.5,0)--(14.3, -0.1);
\draw[line width=1] (14.5, 0.5)--(14.8, -0.5);
\draw[line width=1] (14.8, -0.5)--(15.25, -1.6);
\draw[line width=1] (16.5, 1.1)--(16, 0)--(16.5, -1.1);
\draw[line width=1] (16, 0)--(15.25, 1.6);
\draw[line width=1] (14.8,-0.5)--(16, 0)--(17, 0);
\node  at (13.3, 0) {$a$};
\node  at (13.8, 1.3) {$b$};
\node  at (15.3, 1.8) {$a$};
\node  at (16.8, 1.3) {$b$};
\node  at (17.2, 0) {$c$};
\node  at (16.8, -1.3) {$d$};
\node  at (15.3, -1.8) {$c$};
\node  at (13.8, -1.3) {$d$};
\node  at (14.75, 0.65) {\tinymath A};
\node  at (14.5, -0.15) {\tinymath E};
\node  at (15.05, -0.6) {\tinymath C};
\node  at (15.1, 1.25) {\tinymath ED};
\node  at (16.1, 0.85) {\tinymath AB};
\node  at (16.5, -0.5) {\tinymath EF};
\node  at (16.6, 0.15) {\tinymath CF};
\node  at (15.25, -0.1) {\tinymath CD};
\node  at (14.4, 0.95) {\tinymath AB};
\node  at (13.9, 0.15) {\tinymath ED};
\node  at (15.2, -1.0) {\tinymath CF};
\node  at (13.9, -0.5) {\tinymath EF};
\node  at (15.8, -0.28) {$v$};
\draw[fill=white] (14.5, 0.5) circle[radius=0.1];
\draw[fill=white] (14.3, -0.1) circle[radius=0.1];
\draw[fill=white] (14.8, -0.5) circle[radius=0.1];
\draw[fill=white] (16, 0) circle[radius=0.1];
\end{tikzpicture}
\caption{Case of central edge $CD$.}
\label{caseCD1}
\end{center}
\end{figure}

%%%%%%%%%%%%%%
%%%%%%%%%%%%%%%%%%

\begin{figure}[ht]
\begin{center}
\begin{tikzpicture}[scale=0.75, line width = 0.5]
% octagon boundary
\draw (0, -0.6)--(0, 0.6)--(1, 1.6)--(2.5, 1.6)--(3.5, 0.6)--(3.5, -0.6)--(2.5, -1.6)--(1, -1.6)--(0, -0.6);
\node  at (-0.2, 0) {$a$};
\node  at (0.3, 1.3) {$b$};
\node  at (1.8, 1.8) {$a$};
\node  at (3.3, 1.3) {$b$};
\node  at (3.7, 0) {$c$};
\node  at (3.3, -1.3) {$d$};
\node  at (1.8, -1.8) {$c$};
\node  at (0.3, -1.3) {$d$};
%\node  at (1.2, 0.2) {\tinymath CAE};
%\node  at (2.2, -0.2) {\tinymath DBF};
\node  at (1.7, 0.45) {\tinymath CD};
\node  at (0.5, 0.15) {\tinymath AB};
\node  at (1.2, 0.65) {\tinymath C};
\node  at (1.45, 0) {\tinymath A};
\node  at (1.25, -0.5) {\tinymath E};
\node  at (1.7, 1.2) {\tinymath AB};
\node  at (0.9, 1.0) {\tinymath CF};
\node  at (2.65, 0.9) {\tinymath CF};
\node  at (0.5, -0.7) {\tinymath ED};
\node  at (2.45, -0.7) {\tinymath ED};
\node  at (1.7, -1.0) {\tinymath EF};
\node  at (3.18, 0.22) {\tinymath EF};
\node  at (2.1, -0.2) {\tinymath D};
\node  at (2.15, 0.55) {\tinymath B};
\node  at (2.9, -0.13) {\tinymath F};
\draw[line width=1] (0.5, 1.1)--(1, 0.5)--(1.2, 0)--(1, -0.5)--(0.5, -1.1);
\draw[line width=1] (0,0)--(1.2, 0);
\draw[line width=1] (1, 0.5)--(2.3, 0);
\draw[line width=1] (1,-0.5)--(1.75, -1.6);
\draw[line width=1] (3.0, 1.1)--(2.9, 0.1);
\draw[line width=1] (2.3, 0)--(3.0, -1.1);
\draw[line width=1] (2.4, 0.5)--(1.75, 1.6);
\draw[line width=1] (2.9, 0.1)--(3.5, 0);
\draw[line width=1] (2.3,0)--(2.4, 0.5)--(2.9, 0.1);
\draw[fill=white] (1, 0.5) circle[radius=0.1];
\draw[fill=white] (1.2, 0) circle[radius=0.1];
\draw[fill=white] (1, -0.5) circle[radius=0.1];
%\draw[fill=white] (2.5, 0) circle[radius=0.1];
\draw[fill=white] (2.4, 0.5) circle[radius=0.1];
\draw[fill=white] (2.3, 0) circle[radius=0.1];
\draw[fill=white] (2.9, 0.1) circle[radius=0.1];
%
%  second octagon
% octagon boundary
\draw (4.5, -0.6)--(4.5, 0.6)--(5.5, 1.6)--(7, 1.6)--(8, 0.6)--(8, -0.6)--(7, -1.6)--(5.5, -1.6)--(4.5, -0.6);
\draw[line width=1] (5, 1.1)--(5.7, 0.5)--(5.7, -0.5)--(5.3, 0)--(5, -1.1);
\draw[line width=1] (4.5,0)--(5.3, 0);
\draw[line width=1] (5.7,-0.5)--(6.25, -1.6);
\draw[line width=1] (7.5, 1.1)--(7, 0.5);
\draw[line width=1] (7, -0.5)--(7.5, -1.1);
\draw[line width=1] (7,0.5)--(6.25, 1.6);
\draw[line width=1] (7.3, 0)--(8, 0);
\draw[line width=1] (5.7, 0.5)--(7, -0.5);
\draw[line width=1] (7,0.5)--(7.3, 0)--(7, -0.5);
\node  at (4.3, 0) {$a$};
\node  at (4.8, 1.3) {$b$};
\node  at (6.3, 1.8) {$a$};
\node  at (7.8, 1.3) {$b$};
\node  at (8.2, 0) {$c$};
\node  at (7.8, -1.3) {$d$};
\node  at (6.3, -1.8) {$c$};
\node  at (4.8, -1.3) {$d$};
\node  at (5.9, 0.7) {\tinymath C};
\node  at (5.95, -0.5) {\tinymath A};
\node  at (5.35, 0.25) {\tinymath E};
\node  at (6.2, -1.0) {\tinymath AB};
\node  at (4.85, 0.15) {\tinymath EF};
\node  at (4.9, -0.5) {\tinymath ED};
\node  at (6.2, 1.2) {\tinymath EF};
\node  at (5.4, 1.0) {\tinymath CF};
\node  at (7.1, 1.0) {\tinymath CF};
\node  at (7.5, -0.65) {\tinymath ED};
\node  at (7.6, 0.15) {\tinymath AB};
\node  at (6.55, 0.15) {\tinymath CD};
%\node  at (6.9, -0.6) {\tinymath E};
\node  at (7.05, 0) {\tinymath B};
\node  at (6.85, -0.7) {\tinymath D};
\node  at (6.7, 0.5) {\tinymath F};
\draw[fill=white] (5.7, 0.5) circle[radius=0.1];
\draw[fill=white] (5.3, 0) circle[radius=0.1];
\draw[fill=white] (5.7, -0.5) circle[radius=0.1];
%\draw[fill=white] (7, 0) circle[radius=0.1];
\draw[fill=white] (7, 0.5) circle[radius=0.1];
\draw[fill=white] (7.3, 0) circle[radius=0.1];
\draw[fill=white] (7, -0.5) circle[radius=0.1];

%
%  third octagon
% octagon boundary
\draw (9, -0.6)--(9, 0.6)--(10, 1.6)--(11.5, 1.6)--(12.5, 0.6)--(12.5, -0.6)--(11.5, -1.6)--(10, -1.6)--(9, -0.6);
%\draw[line width=1] (6,0)--(9.5, 0);
\draw[line width=1] (9.5, 1.1)--(10, 0.5)--(10, 0)--(10.2, -0.5);
\draw[line width=1] (10.0, 0)--(9.5, -1.1);
\draw[line width=1] (9,0)--(10, 0.5);
\draw[line width=1] (10.2, -0.5)--(10.75, -1.6);
\draw[line width=1] (10.2, -0.5)--(11.2, 0.5);
\draw[line width=1] (12, 1.1)--(11.8, 0.2);
\draw[line width=1] (11.5, -0.5)--(12, -1.1);
\draw[line width=1] (11.2, 0.5)--(10.75, 1.6);
\draw[line width=1] (11.8, 0.2)--(12.5, 0);
\draw[line width=1] (11.2, 0.5)--(11.5, -0.5)--(11.8, 0.2);
\node  at (8.8, 0) {$a$};
\node  at (9.3, 1.3) {$b$};
\node  at (10.8, 1.8) {$a$};
\node  at (12.3, 1.3) {$b$};
\node  at (12.7, 0) {$c$};
\node  at (12.3, -1.3) {$d$};
\node  at (10.8, -1.8) {$c$};
\node  at (9.3, -1.3) {$d$};
\node  at (10.25, 0.65) {\tinymath E};
\node  at (10.2, 0.0) {\tinymath A};
\node  at (10.50, -0.55) {\tinymath C};
\node  at (10.6, 1.25) {\tinymath ED};
\node  at (11.65, 0.9) {\tinymath EF};
\node  at (12, -0.6) {\tinymath AB};
\node  at (12.18, 0.25) {\tinymath CF};
\node  at (10.65, 0.2) {\tinymath CD};
\node  at (10.7, -1.0) {\tinymath CF};
\node  at (9.4, -0.5) {\tinymath AB};
\node  at (9.8, 1.0) {\tinymath EF};
\node  at (9.4, 0.35) {\tinymath ED};
\node  at (11.2, -0.5) {\tinymath B};
\node  at (10.95, 0.53) {\tinymath D};
\node  at (11.55, 0.25) {\tinymath F};
\draw[fill=white] (10, 0.5) circle[radius=0.1];
\draw[fill=white] (10, 0) circle[radius=0.1];
\draw[fill=white] (10.2, -0.5) circle[radius=0.1];
%\draw[fill=white] (11.5, 0) circle[radius=0.1];
\draw[fill=white] (11.2, 0.5) circle[radius=0.1];
\draw[fill=white] (11.5, -0.5) circle[radius=0.1];
\draw[fill=white] (11.8, 0.2) circle[radius=0.1];
%
%  fourth octagon
% octagon boundary
\draw (13.5, -0.6)--(13.5, 0.6)--(14.5, 1.6)--(16, 1.6)--(17, 0.6)--(17, -0.6)--(16, -1.6)--(14.5, -1.6)--(13.5, -0.6);
%\draw[line width=1] (6,0)--(9.5, 0);
\draw[line width=1] (14, 1.1)--(14.5, 0.5)--(14.3, -0.1)--(14, -1.1);
\draw[line width=1] (13.5,0)--(14.3, -0.1);
\draw[line width=1] (14.5, 0.5)--(14.8, -0.5);
\draw[line width=1] (14.8, -0.5)--(15.25, -1.6);
\draw[line width=1] (16.5, 1.1)--(16.1, 0.3)--(15.6, 0.7);
\draw[line width=1] (16.5, -0.5)--(16.5, -1.1);
\draw[line width=1] (15.6, 0.7)--(15.25, 1.6);
\draw[line width=1] (14.8,-0.5)--(15.6, 0.7);
\draw[line width=1] (16.5, -0.5)--(17, 0);
\draw[line width=1] (16.1, 0.3)--(16.5, -0.5);
\node  at (13.3, 0) {$a$};
\node  at (13.8, 1.3) {$b$};
\node  at (15.3, 1.8) {$a$};
\node  at (16.8, 1.3) {$b$};
\node  at (17.2, 0) {$c$};
\node  at (16.8, -1.3) {$d$};
\node  at (15.3, -1.8) {$c$};
\node  at (13.8, -1.3) {$d$};
\node  at (14.75, 0.65) {\tinymath A};
\node  at (14.5, -0.15) {\tinymath E};
\node  at (15.05, -0.6) {\tinymath C};
\node  at (15.1, 1.25) {\tinymath ED};
\node  at (16.1, 0.85) {\tinymath AB};
\node  at (16.2, -0.85) {\tinymath EF};
\node  at (16.65, -0.05) {\tinymath CF};
\node  at (15.25, -0.1) {\tinymath CD};
\node  at (14.45, 1) {\tinymath AB};
\node  at (13.9, 0.15) {\tinymath ED};
\node  at (15.2, -1.0) {\tinymath CF};
\node  at (13.9, -0.5) {\tinymath EF};
\node  at (15.9, 0.15) {\tinymath B};
\node  at (15.35, 0.7) {\tinymath D};
\node  at (16.25, -0.5) {\tinymath F};
\draw[fill=white] (14.5, 0.5) circle[radius=0.1];
\draw[fill=white] (14.3, -0.1) circle[radius=0.1];
\draw[fill=white] (14.8, -0.5) circle[radius=0.1];
%\draw[fill=white] (16, 0) circle[radius=0.1];
\draw[fill=white] (15.6, 0.7) circle[radius=0.1];
\draw[fill=white] (16.1, 0.3) circle[radius=0.1];
\draw[fill=white] (16.5, -0.5) circle[radius=0.1];
\end{tikzpicture}
\caption{Completion of the case of central edge $CD$.}
\label{caseCD2}
\end{center}
\end{figure}

%%%%%%%

\begin{lemma}
\label{theta5lemma1}
There are no extensions of $\Theta_5^{\# 2}$ to $K_{3,3}$.
\vspace{2mm}

\begin{proof}
We have vertex $u$ of $\Theta_5^{\# 2}$ representing the vertex $CAE$.  The five
incident edges correspond to $AB, CD, CF, DE, EF$.  When $CAE$
is expanded to the path $[C,A,E]$,  the edges $CD, CF$ must be consecutive,
because they are both incident on $C$;
and $EF, ED$ must be consecutive, because they are both incident on $E$.
The remaining edge must be $AB$.
There are a number of ways of assigning names to the edges incident on vertex $u$
satisfing this requirement.  However, when vertex $DBF$ ($v$ of $\Theta_5^{\# 2}$ in Figure~\ref{theta5}) is expanded
to the path $[D,B,F]$, edges $CD, ED$ must be consecutive because they are
both incident on $D$;  and $CF, EF$ must be consecutive because they
are both incident on $F$. One can see that it is not possible to satisfy these
conditions, for any labelling of
the edges incident on $u$ ($CAE$).  
An example of one of the cases is illustrated in Figure~\ref{theta5_2}.
The other cases are similar. 
(All the cases can be easily considered by fixing one of the five edges to be $AB$.)
\end{proof}
\end{lemma}

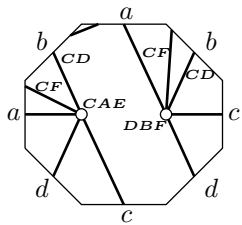
\begin{figure}[ht]
\begin{center}
\begin{tikzpicture}[scale=0.75, line width = 0.5]
%octagon boundary
\draw (4.5, -0.6)--(4.5, 0.6)--(5.5, 1.6)--(7, 1.6)--(8, 0.6)--(8, -0.6)--(7, -1.6)--(5.5, -1.6)--(4.5, -0.6);
\node  at (4.3, 0) {$a$};
\node  at (4.8, 1.3) {$b$};
\node  at (6.3, 1.8) {$a$};
\node  at (7.8, 1.3) {$b$};
\node  at (8.2, 0) {$c$};
\node  at (7.8, -1.3) {$d$};
\node  at (6.3, -1.8) {$c$};
\node  at (4.8, -1.3) {$d$};
\node  at (5.9, 0.2) {\tinymath CAE};
\node  at (6.6, -0.2) {\tinymath DBF};
\node  at (5.4, 1) {\tinymath CD};
\node  at (4.9, 0.5) {\tinymath CF};
\node  at (7.6, 0.7) {\tinymath CD};
\node  at (6.8, 1.1) {\tinymath CF};
\draw[line width=1] (4.5,0)--(5.5, 0)--(4.5, 0.5);
\draw[line width=1] (8,0)--(7, 0)--(7.1, 1.5);
\draw[line width=1] (5, 1.1)--(5.5, 0)--(5, -1.1);
\draw[line width=1] (5.5,0)--(6.25, -1.6);
\draw[line width=1] (7.5, 1.1)--(7, 0)--(7.5, -1.1);
\draw[line width=1] (7,0)--(6.25, 1.6);
\draw[line width=1] (5.8, 1.6)--(5.3, 1.4);
\draw[fill=white] (5.5, 0) circle[radius=0.1];
\draw[fill=white] (7, 0) circle[radius=0.1];
\end{tikzpicture}
\caption{There are no extensions of $\Theta_5^{\#2}$ to $K_{3,3}$.}
\label{theta5_2}
\end{center}
\end{figure}

\begin{lemma}
\label{theta5lemma3}
Up to isomorphism, there are at most three extensions of $\Theta_5^{\#3}$ to $K_{3,3}$.
\vspace{2mm}

\begin{proof}
Without loss of generality, the central edge of $\Theta_5^{\#3}$
can be taken as one of $CD$ or $AB$.\\
%%%%%%%%

\noindent
{\bf Case} $CD$.

Amongst the edges incident on vertex $CAE$ in Figure~\ref{theta5_3case1},
edges $CD$ and $CF$ must be consecutive.  This gives two choices for $CF$,
which are shown in Figure~\ref{theta5_3case1} (on the left).
%%%%%%%
%%%%%%%%%
Also, edges $ED$ and $EF$ must be consecutive.  The remaining edge
must be $AB$.   However, at vertex $DBF$,
edges $DC$ and $DE$ must be consecutive, as must $FC$ and $FE$.
This forces the rest of the labelling of 
the edges, which then extends to an embedding of $K_{3,3}$
in two possible ways.

\begin{figure}[ht]
\begin{center}
\begin{tikzpicture}[scale=0.70, line width = 0.5]
%octagon boundary
\draw (9, -0.6)--(9, 0.6)--(10, 1.6)--(11.5, 1.6)--(12.5, 0.6)--(12.5, -0.6)--(11.5, -1.6)--(10, -1.6)--(9, -0.6);
\node  at (8.8, 0) {$a$};
\node  at (9.3, 1.3) {$b$};
\node  at (10.8, 1.8) {$a$};
\node  at (12.3, 1.3) {$b$};
\node  at (12.7, 0) {$c$};
\node  at (12.3, -1.3) {$d$};
\node  at (10.8, -1.8) {$c$};
\node  at (9.3, -1.3) {$d$};
\node  at (9.77, 0.21) {\tinymath CAE};
\node  at (11.6, -0.19) {\tinymath DBF};
\node  at (10.75, 0.13) {\tinymath CD};
\node  at (9.9, 0.83) {\tinymath CF};
\node  at (12.05, 0.65) {\tinymath CF};
\node  at (9.32, -0.13) {\tinymath AB};
\node  at (10.57, 1.3) {\tinymath AB};
%\node  at (12.1, 0.15) {\tinymath ED};
%\node  at (12.1, 0.15) {\tinymath EF};
%
\draw[line width=1] (9,0)--(10, 0);
\draw[line width=1] (12.5,0)--(11.5, 0)--(11.3, 1.6);
\draw[line width=1] (9.5, 1.1)--(10, 0);
\draw[line width=1] (10,0)--(10.75, -1.6);
\draw[line width=1] (12, 1.1)--(11.5, 0);
\draw[line width=1] (11.5, 0)--(10, 0)--(12.3, -0.8);
\draw[line width=1] (11.5,0)--(10.75, 1.6);
\draw[line width=1] (9.0, -0.3)--(9.25, -0.85);
\draw[fill=white] (10, 0) circle[radius=0.1];
\draw[fill=white] (11.5, 0) circle[radius=0.1];
\node  at (13.75, 0) {$\longrightarrow$};
%octagon boundary
\draw (15, -0.6)--(15, 0.6)--(16, 1.6)--(17.5, 1.6)--(18.5, 0.6)--(18.5, -0.6)--(17.5, -1.6)--(16, -1.6)--(15, -0.6);
\coordinate (u1) at (16.2, 0.4);%C
\coordinate (u2) at (16, 0);%A
\coordinate (u3) at (16.2, -0.4);%E
\coordinate (v1) at (17.5, 0);%D
\coordinate (v2) at (17.3, 0.4);%B
\coordinate (v3) at (17.8, 0.4);%F
\node  at (14.8, 0) {$a$};
\node  at (15.3, 1.3) {$b$};
\node  at (16.8, 1.8) {$a$};
\node  at (18.3, 1.3) {$b$};
\node  at (18.7, 0) {$c$};
\node  at (18.3, -1.3) {$d$};
\node  at (16.8, -1.8) {$c$};
\node  at (15.3, -1.3) {$d$};
\node  at (16.23, 0) {\tinymath A};
\node  at (16.4, 0.55) {\tinymath C};
\node  at (17.45, -0.22) {\tinymath D};
\node  at (15.97, -0.55) {\tinymath E};
\node  at (16.87, 0.01) {\tinymath CD};
\node  at (17.06, 0.49) {\tinymath B};
\node  at (17.99, 0.33) {\tinymath F};
\node  at (18.19, -0.13) {\tinymath ED};
\node  at (17.68, 1.38) {\tinymath EF};
\node  at (17.82, -0.89) {\tinymath EF};
\node  at (15.48, 0.73) {\tinymath CF};
\node  at (18.19, 0.84) {\tinymath CF};
\node  at (16.85, -1.1) {\tinymath ED};
\node  at (15.37, -0.13) {\tinymath AB};
\node  at (16.57, 1.3) {\tinymath AB};
\draw[line width=1] (u1)--(u2)--(u3);
\draw[line width=1] (v1)--(v2)--(v3);
\draw[line width=1] (15,0)--(u2);
\draw[line width=1] (18.5,0)--(v1);
\draw[line width=1] (v3)--(17.3, 1.6);
%\draw[line width=1] (18.5,0)--(17.5, 0)--(17.3, 1.6);
\draw[line width=1] (15.5, 1.1)--(u1);
\draw[line width=1] (u3)--(16.75, -1.6);
\draw[line width=1] (18, 1.1)--(v3);
%\draw[line width=1] (18, 1.1)--(17.5, 0);
\draw[line width=1] (17.5, 0)--(u1);
\draw[line width=1] (u3)--(18.3, -0.8);
\draw[line width=1] (17.5,0)--(16.75, 1.6);
%\draw[line width=1] (17.5,0)--(16.75, 1.6);
\draw[line width=1] (15.0, -0.3)--(15.25, -0.85);
%
%\draw[fill=white] (16, 0) circle[radius=0.1];
%\draw[fill=white] (17.5, 0) circle[radius=0.1];
\draw[fill=white] (u1) circle[radius=0.1];
\draw[fill=white] (u2) circle[radius=0.1];
\draw[fill=white] (u3) circle[radius=0.1];
\draw[fill=white] (v1) circle[radius=0.1];
\draw[fill=white] (v2) circle[radius=0.1];
\draw[fill=white] (v3) circle[radius=0.1];
%
%%%%%
%%%%%% new/extended diagram %%%%
%octagon boundary
\draw (9, -5.1)--(9, -3.9)--(10, -2.9)--(11.5, -2.9)--(12.5, -3.9)--(12.5, -5.1)--(11.5, -6.1)--(10, -6.1)--(9, -5.1);
\node  at (8.8, -4.5) {$a$};
\node  at (9.3, -3.2) {$b$};
\node  at (10.8, -2.7) {$a$};
\node  at (12.3, -3.2) {$b$};
\node  at (12.7, -4.5) {$c$};
\node  at (12.3, -5.8) {$d$};
\node  at (10.8, -6.3) {$c$};
\node  at (9.3, -5.8) {$d$};
\node  at (9.58, -4.7) {\tinymath CAE};
\node  at (11.6, -4.7) {\tinymath DBF};
\node  at (10.7, -4.35) {\tinymath CD};
\node  at (11.1, -5.1) {\tinymath CF};
\node  at (11.62, -3.4) {\tinymath CF};
\node  at (12.12, -4.35) {\tinymath AB};
\node  at (10.87, -5.65) {\tinymath AB};
%\node  at (12.1, 0.15) {\tinymath ED};
%\node  at (12.1, 0.15) {\tinymath EF};
%
\draw[line width=1] (9,-4.5)--(10, -4.5);
\draw[line width=1] (12.5,-4.5)--(11.5, -4.5)--(11.3, -2.9);
\draw[line width=1] (9.5, -3.4)--(10, -4.5);
\draw[line width=1] (10,-4.5)--(10.75, -6.1);
\draw[line width=1] (12, -3.4)--(11.5, -4.5);
\draw[line width=1] (11.5, -4.5)--(10, -4.5)--(12.3, -5.3);
\draw[line width=1] (11.5,-4.5)--(10.75, -2.9);
\draw[line width=1] (9.0, -4.8)--(9.25, -5.35);
\draw[fill=white] (10, -4.5) circle[radius=0.1];
\draw[fill=white] (11.5, -4.5) circle[radius=0.1];
\node  at (13.75, -4.5) {$\longrightarrow$};
%%%%
%%%%%
%octagon boundary
\draw (15, -5.1)--(15, -3.9)--(16, -2.9)--(17.5, -2.9)--(18.5, -3.9)--(18.5, -5.1)--(17.5, -6.1)--(16, -6.1)--(15, -5.1);
\coordinate (u1) at (16.7, -4.6);%C
\coordinate (u2) at (16.2, -4.5);%A
\coordinate (u3) at (16.0, -4.1);%E
\coordinate (v1) at (17.5, -4.5);%D
\coordinate (v2) at (17.8, -4.5);%B
\coordinate (v3) at (17.8, -4.1);%F
\node  at (14.8, -4.5) {$a$};
\node  at (15.3, -3.2) {$b$};
\node  at (16.8, -2.7) {$a$};
\node  at (18.3, -3.2) {$b$};
\node  at (18.7, -4.5) {$c$};
\node  at (18.3, -5.8) {$d$};
\node  at (16.8, -6.3) {$c$};
\node  at (15.3, -5.8) {$d$};
%%%%%%%
\node  at (16.7, -4.82) {\tinymath C};
\node  at (16.23, -3.97) {\tinymath E};
\node  at (17.49, -4.72) {\tinymath D};
\node  at (16.05, -4.67) {\tinymath A};
\node  at (17.07, -4.41) {\tinymath CD};
\node  at (17.97, -4.66) {\tinymath B};
\node  at (17.55, -4.05) {\tinymath F};
\node  at (18.21, -4.37) {\tinymath AB};
\node  at (17.7, -3.2) {\tinymath CF};
\node  at (17.7, -5.28) {\tinymath CF};
\node  at (15.42, -3.79) {\tinymath EF};
\node  at (18.13, -3.78) {\tinymath EF};
\node  at (16.9, -5.59) {\tinymath AB};
\node  at (15.45, -4.52) {\tinymath ED};
\node  at (16.56, -3.17) {\tinymath ED};
\draw[line width=1] (u1)--(u2)--(u3);
\draw[line width=1] (v1)--(v2)--(v3);
\draw[line width=1] (15,-4.5)--(u3);
\draw[line width=1] (18.5,-4.5)--(v1);
\draw[line width=1] (v3)--(17.3, -2.9);
%\draw[line width=1] (18.5,0)--(17.5, 0)--(17.3, 1.6);
\draw[line width=1] (15.5, -3.4)--(u3);
\draw[line width=1] (u2)--(16.75, -6.1);
\draw[line width=1] (18, -3.4)--(v3);
%\draw[line width=1] (18, 1.1)--(17.5, 0);
\draw[line width=1] (17.5, -4.5)--(u1);
\draw[line width=1] (u1)--(18.3, -5.3);
\draw[line width=1] (17.5,-4.5)--(16.75, -2.9);
%\draw[line width=1] (17.5,0)--(16.75, 1.6);
\draw[line width=1] (15.0, -4.8)--(15.25, -5.35);
%
%\draw[fill=white] (16, 0) circle[radius=0.1];
%\draw[fill=white] (17.5, 0) circle[radius=0.1];
\draw[fill=white] (u1) circle[radius=0.1];
\draw[fill=white] (u2) circle[radius=0.1];
\draw[fill=white] (u3) circle[radius=0.1];
\draw[fill=white] (v1) circle[radius=0.1];
\draw[fill=white] (v2) circle[radius=0.1];
\draw[fill=white] (v3) circle[radius=0.1];
%%%%%
\end{tikzpicture}
\caption{$\Theta_5^{\#3}$ with central edge $CD$.}
\label{theta5_3case1}
\end{center}
\end{figure}
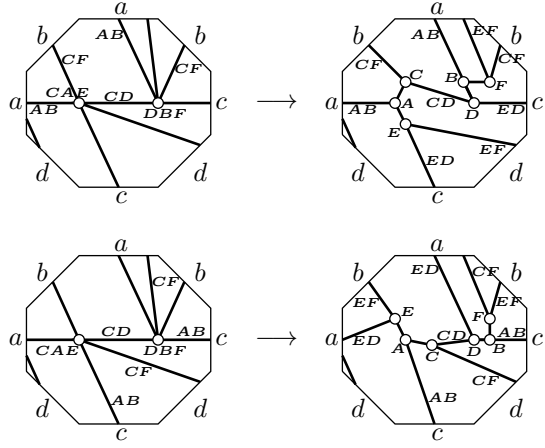

%%%%%%%
%%%%
%
\medbreak
\noindent
{\bf Case} $AB$.

Refer to Figures~\ref{theta5_3case3} and~\ref{caseAB22}.  Edges $EF$ and $ED$ must be consecutive at
vertex $CAE$, and also edges $CF$ and $CD$.  At vertex $DBF$, edges $CF$ and $EF$
must be consecutive, as must edges $CD$ and $ED$. This
leads to four extensions to $K_{3,3}$.  However, it is easy to see from the diagram in Figure~\ref{caseAB22}
that all four are isomorphic, giving only one extension with central edge $AB$.
\end{proof}
\end{lemma}

\begin{figure}[ht]
\begin{center}
\begin{tikzpicture}[scale=0.70, line width = 0.5]
%octagon boundary
\draw (9, -0.6)--(9, 0.6)--(10, 1.6)--(11.5, 1.6)--(12.5, 0.6)--(12.5, -0.6)--(11.5, -1.6)--(10, -1.6)--(9, -0.6);
\node  at (8.8, 0) {$a$};
\node  at (9.3, 1.3) {$b$};
\node  at (10.8, 1.8) {$a$};
\node  at (12.3, 1.3) {$b$};
\node  at (12.7, 0) {$c$};
\node  at (12.3, -1.3) {$d$};
\node  at (10.8, -1.8) {$c$};
\node  at (9.3, -1.3) {$d$};
\node  at (9.6, -0.22) {\tinymath CAE};
\node  at (11.59, -0.21) {\tinymath DBF};
\node  at (10.8, 0.15) {\tinymath AB};
%\node  at (11.7, -0.8) {\tinymath CD};
%\node  at (11.6, 1.2) {\tinymath CD};
%\node  at (9.4, -0.15) {\tinymath AB};
%\node  at (10.6, 1.3) {\tinymath AB};
%\node  at (12.1, 0.15) {\tinymath ED};
%\node  at (12.1, 0.15) {\tinymath EF};
%
\draw[line width=1] (9,0)--(10, 0);
\draw[line width=1] (12.5,0)--(11.5, 0)--(11.3, 1.6);
\draw[line width=1] (9.5, 1.1)--(10, 0);
\draw[line width=1] (10,0)--(10.75, -1.6);
\draw[line width=1] (12, 1.1)--(11.5, 0);
\draw[line width=1] (11.5, 0)--(10, 0)--(12.3, -0.8);
\draw[line width=1] (11.5,0)--(10.75, 1.6);
\draw[line width=1] (9.0, -0.3)--(9.25, -0.85);
\draw[fill=white] (10, 0) circle[radius=0.1];
\draw[fill=white] (11.5, 0) circle[radius=0.1];
\end{tikzpicture}
\caption{$\Theta_5^{\# 3}$ with central edge $AB$.}
\label{theta5_3case3}
\end{center}
\end{figure}
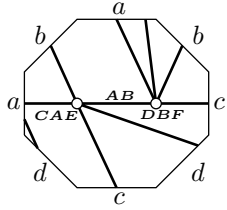

%%%%%%%%%%%%%%%%%%%%%%%

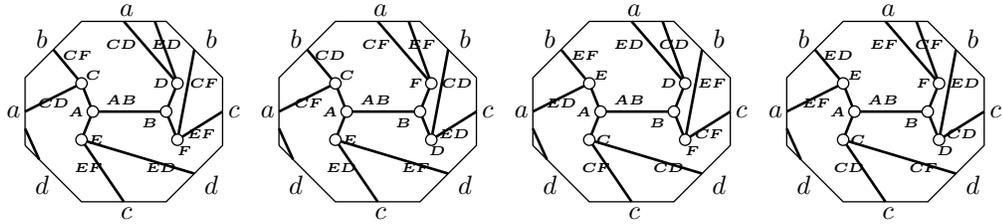
\begin{figure}[ht]
\begin{center}
\begin{tikzpicture}[scale=0.75, line width = 0.5]
% octagon boundary
\draw (0, -0.6)--(0, 0.6)--(1, 1.6)--(2.5, 1.6)--(3.5, 0.6)--(3.5, -0.6)--(2.5, -1.6)--(1, -1.6)--(0, -0.6);
\node  at (-0.2, 0) {$a$};
\node  at (0.3, 1.3) {$b$};
\node  at (1.8, 1.8) {$a$};
\node  at (3.3, 1.3) {$b$};
\node  at (3.7, 0) {$c$};
\node  at (3.3, -1.3) {$d$};
\node  at (1.8, -1.8) {$c$};
\node  at (0.3, -1.3) {$d$};
%\node  at (1.2, 0.2) {\tinymath CAE};
%\node  at (2.2, -0.2) {\tinymath DBF};
\node  at (1.7, 0.2) {\tinymath AB};
\node  at (0.5, 0.15) {\tinymath CD};
\node  at (1.2, 0.65) {\tinymath C};
\node  at (0.9, 0) {\tinymath A};
\node  at (1.25, -0.5) {\tinymath E};
\node  at (1.7, 1.2) {\tinymath CD};
\node  at (2.5, 1.2) {\tinymath ED};
\node  at (0.9, 1.0) {\tinymath CF};
\node  at (3.15, 0.5) {\tinymath CF};
\node  at (2.4, -1.0) {\tinymath ED};
\node  at (2.8, -0.7) {\tinymath F};
\node  at (1.1, -1.0) {\tinymath EF};
\node  at (3.1, -0.4) {\tinymath EF};
\node  at (2.2, -0.2) {\tinymath B};
\node  at (2.4, 0.5) {\tinymath D};
\draw[line width=1] (0.5, 1.1)--(1, 0.5)--(1.2, 0)--(1, -0.5);
\draw[line width=1] (0,0)--(1, 0.5);
\draw[line width=1] (1.2, 0)--(2.5, 0);
\draw[line width=1] (1,-0.5)--(1.75, -1.6);
\draw[line width=1] (3.0, 1.1)--(2.7, -0.5);
\draw[line width=1] (2.7,0.5)--(1.75, 1.6);
\draw[line width=1] (2.7,-0.5)--(3.5, 0);
\draw[line width=1] (2.7, 0.5)--(2.5, 0)--(2.7, -0.5);
\draw[line width=1] (1, -0.5)--(3.0, -1.1);
\draw[line width=1] (2.7, 0.5)--(2.3, 1.6);
\draw[line width=1] (0, -0.3)--(0.25, -0.85);
\draw[fill=white] (1, 0.5) circle[radius=0.1];
\draw[fill=white] (1.2, 0) circle[radius=0.1];%A
\draw[fill=white] (1, -0.5) circle[radius=0.1];
\draw[fill=white] (2.5, 0) circle[radius=0.1];%B
\draw[fill=white] (2.7, 0.5) circle[radius=0.1];
\draw[fill=white] (2.7, -0.5) circle[radius=0.1];
%
%  second octagon
% octagon boundary
\draw (4.5, -0.6)--(4.5, 0.6)--(5.5, 1.6)--(7, 1.6)--(8, 0.6)--(8, -0.6)--(7, -1.6)--(5.5, -1.6)--(4.5, -0.6);
\draw[line width=1] (5.5, 0.5)--(5.7, 0)--(5.5, -0.5);
\draw[line width=1] (5, 1.1)--(5.5, 0.5);
\draw[line width=1] (4.5,0)--(5.5, 0.5);
\draw[line width=1] (7.5, 1.1)--(7.2, -0.5);
\draw[line width=1] (7.2,0.5)--(6.25, 1.6);
\draw[line width=1] (7, 0)--(5.7, 0);
\draw[line width=1] (7.2,-0.5)--(8.0, 0);
\draw[line width=1] (7.2, 0.5)--(7.0, 0)--(7.2, -0.5);
\draw[line width=1] (5.5,-0.5)--(6.25, -1.6);
\draw[line width=1] (5.5, -0.5)--(7.5, -1.1);
\draw[line width=1] (7.2, 0.5)--(6.8, 1.6);
\draw[line width=1] (4.5, -0.3)--(4.75, -0.85);
\node  at (4.3, 0) {$a$};
\node  at (4.8, 1.3) {$b$};
\node  at (6.3, 1.8) {$a$};
\node  at (7.8, 1.3) {$b$};
\node  at (8.2, 0) {$c$};
\node  at (7.8, -1.3) {$d$};
\node  at (6.3, -1.8) {$c$};
\node  at (4.8, -1.3) {$d$};
\node  at (6.2, 0.2) {\tinymath AB};
\node  at (5, 0.15) {\tinymath CF};
\node  at (5.7, 0.65) {\tinymath C};
\node  at (5.4, 0) {\tinymath A};
\node  at (5.75, -0.5) {\tinymath E};
\node  at (6.2, 1.2) {\tinymath CF};
\node  at (7, 1.2) {\tinymath EF};
\node  at (5.4, 1.0) {\tinymath CD};
\node  at (7.65, 0.5) {\tinymath CD};
\node  at (6.9, -1.0) {\tinymath EF};
\node  at (7.3, -0.7) {\tinymath D};
\node  at (5.6, -1.0) {\tinymath ED};
\node  at (7.6, -0.4) {\tinymath ED};
\node  at (6.7, -0.2) {\tinymath B};
\node  at (6.9, 0.5) {\tinymath F};
\draw[fill=white] (5.5, 0.5) circle[radius=0.1];
\draw[fill=white] (5.7, 0) circle[radius=0.1];%A
\draw[fill=white] (5.5, -0.5) circle[radius=0.1];
\draw[fill=white] (7.0, 0) circle[radius=0.1];%B
\draw[fill=white] (7.2, 0.5) circle[radius=0.1];
\draw[fill=white] (7.2, -0.5) circle[radius=0.1];
%
%  third octagon
% octagon boundary
\draw (9, -0.6)--(9, 0.6)--(10, 1.6)--(11.5, 1.6)--(12.5, 0.6)--(12.5, -0.6)--(11.5, -1.6)--(10, -1.6)--(9, -0.6);
%\draw[line width=1] (6,0)--(9.5, 0);
\draw[line width=1] (10.0, 0.5)--(10.2, 0)--(10.0, -0.5);
\draw[line width=1] (9.5, 1.1)--(10, 0.5);
\draw[line width=1] (9,0)--(10, 0.5);
\draw[line width=1] (10.2, 0)--(11.5, 0);
\draw[line width=1] (12, 1.1)--(11.7, -0.5);
\draw[line width=1] (11.7,0.5)--(10.75, 1.6);
\draw[line width=1] (11.7,-0.5)--(12.5, 0);
\draw[line width=1] (11.7, 0.5)--(11.5, 0)--(11.7, -0.5);
\draw[line width=1] (10.0, -0.5)--(12.0, -1.1);
\draw[line width=1] (10.0,-0.5)--(10.75, -1.6);
\draw[line width=1] (11.7, 0.5)--(11.3, 1.6);
\draw[line width=1] (9, -0.3)--(9.25, -0.85);
\node  at (8.8, 0) {$a$};
\node  at (9.3, 1.3) {$b$};
\node  at (10.8, 1.8) {$a$};
\node  at (12.3, 1.3) {$b$};
\node  at (12.7, 0) {$c$};
\node  at (12.3, -1.3) {$d$};
\node  at (10.8, -1.8) {$c$};
\node  at (9.3, -1.3) {$d$};
\node  at (10.7, 0.2) {\tinymath AB};
\node  at (9.5, 0.15) {\tinymath ED};
\node  at (10.2, 0.65) {\tinymath E};
\node  at (9.9, 0) {\tinymath A};
\node  at (10.25, -0.5) {\tinymath C};
\node  at (10.7, 1.2) {\tinymath ED};
\node  at (11.5, 1.2) {\tinymath CD};
\node  at (9.9, 1.0) {\tinymath EF};
\node  at (12.15, 0.5) {\tinymath EF};
\node  at (11.4, -1.0) {\tinymath CD};
\node  at (11.8, -0.7) {\tinymath F};
\node  at (10.1, -1.0) {\tinymath CF};
\node  at (12.1, -0.4) {\tinymath CF};
\node  at (11.2, -0.2) {\tinymath B};
\node  at (11.4, 0.5) {\tinymath D};
\draw[fill=white] (10.0, 0.5) circle[radius=0.1];
\draw[fill=white] (10.2, 0) circle[radius=0.1];%A
\draw[fill=white] (10.0, -0.5) circle[radius=0.1];
\draw[fill=white] (11.5, 0) circle[radius=0.1];%B
\draw[fill=white] (11.7, 0.5) circle[radius=0.1];
\draw[fill=white] (11.7, -0.5) circle[radius=0.1];
%
%  fourth octagon
% octagon boundary
\draw (13.5, -0.6)--(13.5, 0.6)--(14.5, 1.6)--(16, 1.6)--(17, 0.6)--(17, -0.6)--(16, -1.6)--(14.5, -1.6)--(13.5, -0.6);
\draw[line width=1] (14.5, 0.5)--(14.7, 0)--(14.5, -0.5);
\draw[line width=1] (14, 1.1)--(14.5, 0.5);
\draw[line width=1] (13.5,0)--(14.5, 0.5);
\draw[line width=1] (16, 0)--(14.7, 0);
\draw[line width=1] (16.2,0.5)--(15.25, 1.6);
\draw[line width=1] (16.2,-0.5)--(17, 0);
\draw[line width=1] (16.2, 0.5)--(16.0, 0)--(16.2, -0.5);
\draw[line width=1] (14.5, -0.5)--(16.5, -1.1);
\draw[line width=1] (14.5,-0.5)--(15.25, -1.6);
\draw[line width=1] (16.2, 0.5)--(15.8, 1.6);
\draw[line width=1] (13.5, -0.3)--(13.75, -0.85);
\draw[line width=1] (16.5, 1.1)--(16.2, -0.5);
\node  at (13.3, 0) {$a$};
\node  at (13.8, 1.3) {$b$};
\node  at (15.3, 1.8) {$a$};
\node  at (16.8, 1.3) {$b$};
\node  at (17.2, 0) {$c$};
\node  at (16.8, -1.3) {$d$};
\node  at (15.3, -1.8) {$c$};
\node  at (13.8, -1.3) {$d$};
\node  at (15.2, 0.2) {\tinymath AB};
\node  at (14, 0.15) {\tinymath EF};
\node  at (14.7, 0.65) {\tinymath E};
\node  at (14.4, 0) {\tinymath A};
\node  at (14.75, -0.5) {\tinymath C};
\node  at (15.2, 1.2) {\tinymath EF};
\node  at (16, 1.2) {\tinymath CF};
\node  at (14.4, 1.0) {\tinymath ED};
\node  at (16.65, 0.5) {\tinymath ED};
\node  at (15.9, -1.0) {\tinymath CF};
\node  at (16.3, -0.7) {\tinymath D};
\node  at (14.6, -1.0) {\tinymath CD};
\node  at (16.6, -0.4) {\tinymath CD};
\node  at (15.7, -0.2) {\tinymath B};
\node  at (15.9, 0.5) {\tinymath F};
%\node  at (14.75, 0.65) {\tinymath A};
%\node  at (14.5, -0.15) {\tinymath E};
%\node  at (15.05, -0.6) {\tinymath C};
%\node  at (15.1, 1.25) {\tinymath ED};
%\node  at (16.1, 0.85) {\tinymath AB};
%\node  at (16.5, -0.5) {\tinymath EF};
%\node  at (16.6, 0.15) {\tinymath CF};
%\node  at (15.25, 0.2) {\tinymath AB};
%\node  at (14.4, 0.95) {\tinymath AB};
%\node  at (13.9, 0.15) {\tinymath ED};
%\node  at (15.2, -1.0) {\tinymath CF};
%\node  at (13.9, -0.5) {\tinymath EF};
%\node  at (15.8, -0.28) {\tinymath B};
%
\draw[fill=white] (14.5, 0.5) circle[radius=0.1];
\draw[fill=white] (14.7, 0) circle[radius=0.1];%A
\draw[fill=white] (14.5, -0.5) circle[radius=0.1];
\draw[fill=white] (16.0, 0) circle[radius=0.1];%B
\draw[fill=white] (16.2, 0.5) circle[radius=0.1];
\draw[fill=white] (16.2, -0.5) circle[radius=0.1];
\end{tikzpicture}
\caption{Extensions of $\Theta_5^{\#3}$ to $K_{3,3}$  with central edge $AB$.}
\label{caseAB22}
\end{center}
\end{figure}

%%%%%%%%%%%%%%

%%%%%%%%%%%%%%

\begin{theorem}
\label{K33theorem}
Up to isomorphism, there is a unique 2-cell embedding of $K_{3,3}$ on the double torus.
\end{theorem}

\begin{proof}
There are four labelled 2-cell embeddings found in Lemma~\ref{CDlemma}
and three more in Lemma~\ref{theta5lemma3}.
In order to compare these for isomorphism, 
we reduce the problem to digraph isomorphism as follows.
%%%%%%%%%
Each labelled embedding can be described by a rotation system in a unique way,
i.e., there is a one-to-one correspondence between the embeddings 
and their rotation systems (see Introduction and \cite{KocayKreher} for 
more details). For example, the rotation system of the left embedding in Figure~\ref{caseCD2}
can be represented as
\medbreak
$A: (B,C,E)$

$B: (A,F,D)$

$C: (A,F,D)$

$D: (B,E,C)$

$E: (A,F,D)$

$F: (B,C,E)$
\medbreak

We convert an embedding to its medial digraph -- 
each edge of the embedding is subdivided with a new vertex, and a directed cycle is drawn around
each vertex using the subdividing vertices, according to the cyclic order and direction in the
rotation system. Isomorphic embeddings produce isomorphic medial digraphs,
and vice-versa (see \cite{GagarinKocayNeilsen, KocayKreher} for more details).   
By constructing the medial digraphs, and comparing them for
isomorphism using graph isomorphism software (e.g. see \cite{Match}), we find that all
the embeddings of $K_{3,3}$ found above are isomorphic.
The unique 2-cell embedding is shown in Figure~\ref{K33embedding}.

\end{proof}

\begin{figure}[ht]
\begin{center}
\begin{tikzpicture}[scale=0.75, line width = 0.5]
%octagon boundary
\draw (0, -0.6)--(0, 0.6)--(1, 1.6)--(2.5, 1.6)--(3.5, 0.6)--(3.5, -0.6)--(2.5, -1.6)--(1, -1.6)--(0, -0.6);
\node  at (-0.2, 0) {$a$};
\node  at (0.3, 1.3) {$b$};
\node  at (1.8, 1.8) {$a$};
\node  at (3.3, 1.3) {$b$};
\node  at (3.7, 0) {$c$};
\node  at (3.3, -1.3) {$d$};
\node  at (1.8, -1.8) {$c$};
\node  at (0.3, -1.3) {$d$};
%\node  at (1.2, 0.2) {$u$};
%\node  at (2.2, -0.2) {$v$};
%
\draw[line width=1] (0,0)--(0.7, 0.4)--(1.15, 0.1)--(1.6, -0.1)--(2.05, -0.1)--(2.5, 0.0)--(2.95, 0.2)--(3.5, 0);
\draw[line width=1] (2.05, -0.1)--(1.75, 1.6);
\draw[line width=1] (0.7, 0.4)--(0.5, 1.1);
\draw[line width=1] (1.15, 0.1)--(0.5, -1.1);
\draw[line width=1] (1.6, -0.1)--(1.75, -1.6);
\draw[line width=1] (2.5, 0.0)--(3, -1.1);
\draw[line width=1] (2.95, 0.2)--(3, 1.1);
\node  at (0.9, 0.6) {\smallmath 1};
\node  at (1.3, 0.35) {\smallmath 2};
\node  at (1.65, 0.15) {\smallmath 3};
\node  at (2.05, -0.35) {\smallmath 4};
\node  at (2.45, 0.3) {\smallmath 5};
\node  at (2.95, -0.05) {\smallmath 6};
\draw[fill=white] (0.7, 0.4) circle[radius=0.1];
\draw[fill=white] (1.15, 0.1) circle[radius=0.1];
\draw[fill=white] (1.6, -0.1) circle[radius=0.1];
\draw[fill=white] (2.05, -0.1) circle[radius=0.1];
\draw[fill=white] (2.5, 0.0) circle[radius=0.1];
\draw[fill=white] (2.95, 0.2) circle[radius=0.1];
\end{tikzpicture}
\caption{The unique 2-cell embedding of $K_{3,3}$ on the double torus.}
\label{K33embedding}
\end{center}
\end{figure}
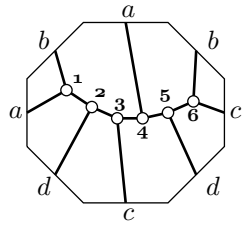

It is easy to see that the embedding in Figure~\ref{K33embedding} is non-orientable: the permutation $(1)(4)(26)(35)$ of vertices of $K_{3,3}$, which is an automorphism of $K_{3,3}$, maps the rotations to their reversals, implying the non-orientability.

%%%%
%%%%%%%
%%%%%%%%%%%%%%%%%%
\section{Hyperbolic Tilings}
\label{section-reps}

It is well known (see~\cite{Thurston}) that the torus has symbolic representations $a^+b^+a^-b^-$
and $a^+b^+c^+a^-b^-c^-$, which also represent tilings of the Euclidean plane by rectangles and
by regular hexagons, respectively.  In the tiling by rectangles, four rectangles meet at each vertex.
In the tiling by hexagons, three hexagons meet at each vertex. The corresponding 
translation groups of the plane
map rectangles to rectangles, and hexagons to hexagons. 
Any graph embedding on the torus 
which has exactly one face, and the face is equivalent to an open disk,
provides a polygonal representation of the torus.  
By Euler's formula, the number of
vertices and edges then satisfy $n + 1 -\varepsilon = 0$, so that $\varepsilon = n+ 1$.
Since we require the minimum degree to be at least three, this limits the possible graphs.
We must then have $n\le 2$, from which it follows that the rectangle and hexagon are the only
one-polygon representations of the torus.

With the double torus, there are many more possibilities for such one-face embeddings.
The standard representation $a^+b^+a^-b^-c^+d^+c^-d^-$ of the double torus
produces a tiling of the hyperbolic plane by regular octagons, in which eight octagons meet at each vertex.
The 2-cell embeddings of $\Theta_5$ and $K_{3,3}$ 
produce additional tilings of the hyperbolic plane.  The three embeddings of $\Theta_5$ produce tilings
by regular $10$-gons, i.e. polygons with $10$ sides.  
One has fundamental region $a^+b^+c^+d^+e^+c^-d^-a^-b^-e^-$, the second
has fundamental region $a^+b^+c^+d^+e^+a^-b^-c^-d^-e^-$,
and the third has fundamental region $a^+b^+c^+a^-d^+c^-e^+d^-b^-e^-$.
In each case five 10-gons meet at each vertex.  
The polygons of the fundamental regions determined by the embeddings of $\Theta_5$ 
are shown in Figure~\ref{hyperbolic}.   Dotted lines are used to show the
pairing of edges of the 10-gons.  The polygon boundaries are traversed in a clockwise direction,
such that for each pair of corresponding edges, the orientations of the 
two edges are opposite.

$K_{3,3}$ produces a tiling of the hyperbolic plane
by regular 18-gons, i.e., polygons with $18$ sides, with 
three 18-gons meeting at each vertex.
It gives a polygonal representation $a^+b^+c^+d^+e^+f^+b^-g^+h^+c^-f^-i^+g^-a^-d^-h^-i^-e^-$ 
of the double torus. The edges of the fundamental polygon are the nine edges of 
$K_{3,3}$, each one appearing twice on the polygon boundary.  
This is illustrated in Figure~\ref{18gon}, where dotted lines
show the pairing of edges of the polygon.

Euler's formula for the double torus and a one-face embedding gives $\varepsilon = n + 3$. Since the minimum degree
is required to be at least three, this gives $2\varepsilon \ge 3n$, or $2(n+3) \ge 3n$,
giving $n\le 6$.   There are a number of graphs satisfying this condition.  Each one will give
a number of representations of the double torus as a fundamental region
of the hyperbolic plane.  We use some
of them to find the 2-cell embeddings of $K_5$ on the double torus 
in the next section.

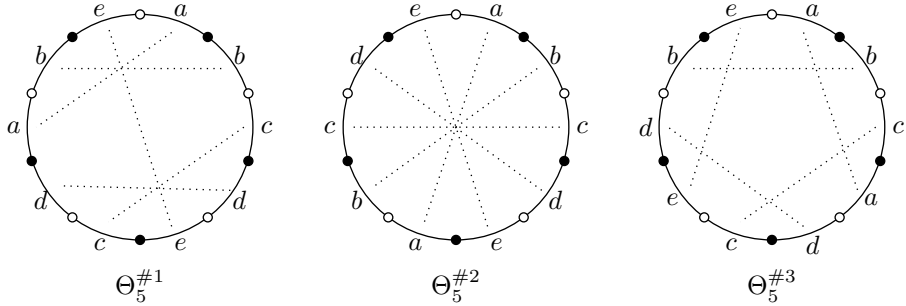
\begin{figure}[ht]
\begin{center}
\begin{tikzpicture}[scale=0.6, line width = 0.5]
% 10-gon boundary
\draw[fill=white] (0, 0) circle[radius=2.5];
\node  at (0.9, 2.6) {$a$};
\node  at (2.2, 1.6) {$b$};
\node  at (2.8, 0) {$c$};
\node  at (2.2, -1.6) {$d$};
\node  at (0.9, -2.6) {$e$};
\node  at (-0.9, -2.6) {$c$};
\node  at (-2.2, -1.6) {$d$};
\node  at (-2.8, 0) {$a$};
\node  at (-2.2, 1.6) {$b$};
\node  at (-0.9, 2.6) {$e$};
\draw[fill=white] (0, 2.5) circle[radius=0.1];
\draw[fill=black] (0, -2.5) circle[radius=0.1];
\draw[fill=white] (2.4, 0.75) circle[radius=0.1];
\draw[fill=black] (2.4, -0.75) circle[radius=0.1];
\draw[fill=white] (-2.4, 0.75) circle[radius=0.1];
\draw[fill=black] (-2.4, -0.75) circle[radius=0.1];
\draw[fill=black] (1.5, 2) circle[radius=0.1];
\draw[fill=white] (1.5, -2) circle[radius=0.1];
\draw[fill=black] (-1.5, 2) circle[radius=0.1];
\draw[fill=white] (-1.5, -2) circle[radius=0.1];
%
%\draw[dotted, line width=0.5] (0.7, 2.1) .. controls (-0.5, 0.5) .. (-2.3, 0); % a
\draw[dotted, line width=0.5] (0.7, 2.1)--(-2.3, 0); % a
\draw[dotted, line width=0.5] (1.8, 1.3)--(-1.8, 1.3); % b
\draw[dotted, line width=0.5] (2.3, 0)--(-0.7, -2.1); % c
\draw[dotted, line width=0.5] (2.0, -1.4)--(-1.8, -1.3); % d
\draw[dotted, line width=0.5] (0.7, -2.2)--(-0.7, 2.2); % e
%
% 10-gon boundary
\draw[fill=white] (7, 0) circle[radius=2.5];
\node  at (7.9, 2.6) {$a$};
\node  at (9.2, 1.6) {$b$};
\node  at (9.8, 0) {$c$};
\node  at (9.2, -1.6) {$d$};
\node  at (7.9, -2.6) {$e$};
\node  at (6.1, -2.6) {$a$};
\node  at (4.8, -1.6) {$b$};
\node  at (4.2, 0) {$c$};
\node  at (4.8, 1.6) {$d$};
\node  at (6.1, 2.6) {$e$};
\draw[dotted, line width=0.5] (7.7, 2.1)--(6.3, -2.1); % a
\draw[dotted, line width=0.5] (8.8, 1.3)--(5.2, -1.3); % b
\draw[dotted, line width=0.5] (9.3, 0)--(4.7, 0); % c
\draw[dotted, line width=0.5] (8.9, -1.4)--(5.2, 1.3); % d
\draw[dotted, line width=0.5] (7.7, -2.2)--(6.3, 2.2); % e
\draw[fill=white] (7, 2.5) circle[radius=0.1];
\draw[fill=black] (7, -2.5) circle[radius=0.1];
\draw[fill=white] (9.4, 0.75) circle[radius=0.1];
\draw[fill=black] (9.4, -0.75) circle[radius=0.1];
\draw[fill=white] (4.6, 0.75) circle[radius=0.1];
\draw[fill=black] (4.6, -0.75) circle[radius=0.1];
\draw[fill=black] (8.5, 2) circle[radius=0.1];
\draw[fill=white] (8.5, -2) circle[radius=0.1];
\draw[fill=black] (5.5, 2) circle[radius=0.1];
\draw[fill=white] (5.5, -2) circle[radius=0.1];
\
%
% 10-gon boundary
\draw[fill=white] (14, 0) circle[radius=2.5];
\node  at (14.9, 2.6) {$a$};
\node  at (16.2, 1.6) {$b$};
\node  at (16.8, 0) {$c$};
\node  at (16.2, -1.6) {$a$};
\node  at (14.9, -2.6) {$d$};
\node  at (13.1, -2.6) {$c$};
\node  at (11.8, -1.6) {$e$};
\node  at (11.2, 0) {$d$};
\node  at (11.8, 1.6) {$b$};
\node  at (13.1, 2.6) {$e$};
\draw[dotted, line width=0.5] (14.7, 2.1)--(15.9, -1.4); % a
\draw[dotted, line width=0.5] (15.8, 1.3)--(12.2, 1.3); % b
\draw[dotted, line width=0.5] (16.3, 0)--(13.3, -2.1); % c
\draw[dotted, line width=0.5] (14.7, -2.2)--(11.7, 0); % d
\draw[dotted, line width=0.5] (12.2, -1.3)--(13.3, 2.2); % e
\draw[fill=white] (14, 2.5) circle[radius=0.1];
\draw[fill=black] (14, -2.5) circle[radius=0.1];
\draw[fill=white] (16.4, 0.75) circle[radius=0.1];
\draw[fill=black] (16.4, -0.75) circle[radius=0.1];
\draw[fill=white] (11.6, 0.75) circle[radius=0.1];
\draw[fill=black] (11.6, -0.75) circle[radius=0.1];
\draw[fill=black] (15.5, 2) circle[radius=0.1];
\draw[fill=white] (15.5, -2) circle[radius=0.1];
\draw[fill=black] (12.5, 2) circle[radius=0.1];
\draw[fill=white] (12.5, -2) circle[radius=0.1];
\node at (0, -3.5) {$\Theta_5^{\# 1}$};
\node at (7, -3.5) {$\Theta_5^{\# 2}$};
\node at (14, -3.5) {$\Theta_5^{\# 3}$};
\end{tikzpicture}
\caption{Three 10-gons representing the double torus as a fundamental region of the hyperbolic plane.}
\label{hyperbolic}
\end{center}
\end{figure}

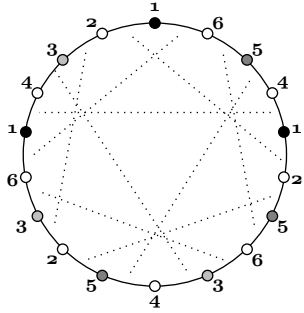
\begin{figure}[ht]
\begin{center}
\begin{tikzpicture}[scale=0.7, line width = 0.5]
% 18-gon boundary
\draw[fill=white] (7, 0) circle[radius=2.5];
\node  at (7, 2.8) {\smallmath 1};
\node  at (8.2, 2.5) {\smallmath 6};
\node  at (8.95, 2) {\smallmath 5};
\node  at (9.43, 1.37) {\smallmath 4};
\node  at (9.7, 0.48) {\smallmath 1};
\node  at (9.7, -0.48) {\smallmath 2};
\node  at (9.43, -1.37) {\smallmath 5};
\node  at (8.95, -2) {\smallmath 6};
\node  at (8.2, -2.5) {\smallmath 3};
\node  at (7, -2.8) {\smallmath 4};
\node  at (5.8, -2.5) {\smallmath 5};
\node  at (5.05, -2) {\smallmath 2};
\node  at (4.47, -1.37) {\smallmath 3};
\node  at (4.3, -0.48) {\smallmath 6};
\node  at (4.3, 0.48) {\smallmath 1};
\node  at (4.57, 1.37) {\smallmath 4};
\node  at (5.05, 2) {\smallmath 3};
\node  at (5.8, 2.5) {\smallmath 2};
\draw[dotted, line width=0.5] (4.7, 0)--(7.5, 2.3); % 16
\draw[dotted, line width=0.5] (8.3, 1.9)--(8.9, -1.4); % 56
\draw[dotted, line width=0.5] (8.8, 1.4)--(6.5, -2.3); % 45
\draw[dotted, line width=0.5] (9.2, 0.8)--(4.8, 0.8); % 14
\draw[dotted, line width=0.5] (9.4, -0.1)--(6.5, 2.3); % 12
\draw[dotted, line width=0.5] (9.2, -0.8)--(5.7, -2.0); % 25
\draw[dotted, line width=0.5] (8.3, -2.0)--(4.8, -0.7); % 36
\draw[dotted, line width=0.5] (7.6, -2.2)--(5.1, 1.6); % 34
\draw[dotted, line width=0.5] (5.1, -1.3)--(5.7, 2.0); % 23
\draw[fill=black] (7, 2.5) circle[radius=0.1];
\draw[fill=white] (7, -2.5) circle[radius=0.1];
\draw[fill=black] (9.45, 0.43) circle[radius=0.1];
\draw[fill=white] (9.45, -0.43) circle[radius=0.1];
\draw[fill=white] (9.23, 1.17) circle[radius=0.1];
\draw[fill=gray] (9.23, -1.17) circle[radius=0.1];
\draw[fill=gray] (8.75, 1.8) circle[radius=0.1];
\draw[fill=white] (8.75, -1.8) circle[radius=0.1];
\draw[fill=white] (8.0, 2.3) circle[radius=0.1];
\draw[fill=lightgray] (8.0, -2.3) circle[radius=0.1];
\draw[fill=white] (6.0, 2.3) circle[radius=0.1];
\draw[fill=gray] (6.0, -2.3) circle[radius=0.1];
\draw[fill=lightgray] (5.25, 1.8) circle[radius=0.1];
\draw[fill=white] (5.25, -1.8) circle[radius=0.1];
\draw[fill=white] (4.77, 1.17) circle[radius=0.1];
\draw[fill=lightgray] (4.77, -1.17) circle[radius=0.1];
\draw[fill=black] (4.55, 0.43) circle[radius=0.1];
\draw[fill=white] (4.55, -0.43) circle[radius=0.1];
\end{tikzpicture}
\caption{The 18-gon representation of the double torus derived from $K_{3,3}$.}
\label{18gon}
\end{center}
\end{figure}

%%%%%%%
%%%%%%%%

%%%%%%
%%%%%%%%%
\section{$K_5$ on the double torus}
\label{section-K5}

We find the distinct 2-cell embeddings of $K_5$ using a number of intermediate
graphs. They are related to $K_{3,3}$ and $\Theta_5$, and can be used as 
building blocks for 2-cell embeddings of 
various other graphs.

The double torus is an orientable surface.  If the rotations of an embedding are all
reversed, an equivalent, but possibly non-isomorphic, embedding results.
Two such embeddings will be considered as \emph{equivalent}. As before, an embedding $G^\tau$ whose
reversal is isomorphic to $G^\tau$ is said to be {\em non-orientable}.
But if the reversal is non-isomorphic to $G^\tau$, then the embedding
is said to be \emph{orientable}.  
%
%
%%%%%%%%
%%%

Let $T_{i,j,k}$ denote the graph of a triangle in which one edge has multiplicity $i$,
one edge has multiplicity $j$, and the third edge has multiplicity $k$.  For example,
$T_{1,1,1} = K_3$.   $T_{i,j,k}$ has $i+j+k$ edges, so that a $2$-cell embedding of $T_{i,j,k}$
on the double torus has $f=i+j+k-5$ faces.   The most interesting case is when $f=1$, for
then the embedding must be a 2-cell embedding, with no digon faces.
We look at $T_{1,2,3}$.
%%%%%

\begin{theorem}
Up to equivalence, $T_{1,2,3}$ has two orientable 2-cell embeddings,
and four non-orientable 2-cell embeddings on the double torus.
\end{theorem}
\begin{proof}
Let $uv$ be the edge of $T_{1,2,3}$ with multiplicity one.  
Then $T_{1,2,3}\cdot uv \cong \Theta_5$.
By Lemma~\ref{contract}, a 2-cell embedding of $T_{1,2,3}$ can be contracted to a
2-cell embedding of $\Theta_5$, because $uv$ is not part of a digon, and because
the facial boundary of the unique face of $T_{1,2,3}$ is not a triangle.
Consequently every embedding of $T_{1,2,3}$ can be constructed by reversing the
edge contraction.  Each vertex of $\Theta_5$ has degree five.   

%%%%%%%%%

By Corollary~\ref{autotheta5}, the two vertices of $\Theta_5^{\#1}$ or $\Theta_5^{\#2}$ are equivalent, 
so that there are just five
ways of reversing the edge contraction for each of these embeddings 
(because each vertex has degree five and any three edges must be consecutive in the resulting rotation).  
For embedding $\Theta_5^{\#3}$, 
there are $10$ ways of reversing the edge contraction -- five for each vertex.
Comparison of the resulting $20$ embeddings of $T_{1,2,3}$ shows
that eight are non-isomorphic.  However, reversing the rotations shows that four are
non-orientable, and the other four fall into two pairs of equivalent orientable embeddings.  
The six inequivalent embeddings are shown in Figure~\ref{T123-ineq}.
\end{proof}

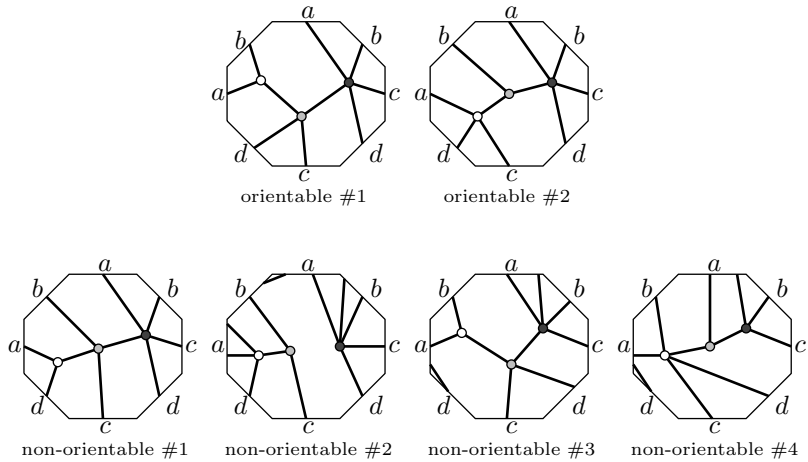
\begin{figure}[ht]
\begin{center}
\begin{tikzpicture}[scale=0.6, line width = 0.5]
% octagon boundary
\draw (4.5, 5)--(4.5, 6.2)--(5.5, 7.2)--(7, 7.2)--(8, 6.2)--(8, 5)--(7, 4)--(5.5, 4)--(4.5, 5);
\node  at (4.3, 5.6) {$a$};
\node  at (4.8, 6.8) {$b$};
\node  at (6.3, 7.4) {$a$};
\node  at (7.8, 6.9) {$b$};
\node  at (8.2, 5.6) {$c$};
\node  at (7.8, 4.3) {$d$};
\node  at (6.2, 3.8) {$c$};
\node  at (4.8, 4.3) {$d$};
\draw[line width=1] (4.5, 5.6)--(5.25, 5.9)--(6.15, 5.1)--(7.2, 5.85)--(8, 5.6);
\draw[line width=1] (5,6.7)--(5.25, 5.9);
\draw[line width=1] (5.1,4.4)--(6.15, 5.1)--(6.25, 4.0);
\draw[line width=1] (6.25, 7.2)--(7.2, 5.85)--(7.5, 6.7);
\draw[line width=1] (7.2, 5.85)--(7.5, 4.5);
\draw[fill=white] (5.25, 5.9) circle[radius=0.1];
\draw[fill=lightgray] (6.15, 5.1) circle[radius=0.1];
\draw[fill=darkgray] (7.2, 5.85) circle[radius=0.1];
%
%  second octagon
% octagon boundary
\draw (9, 5)--(9, 6.2)--(10, 7.2)--(11.5, 7.2)--(12.5, 6.2)--(12.5, 5)--(11.5, 4)--(10, 4)--(9, 5);
\draw[line width=1] (9, 5.6)--(10.05, 5.1)--(10.75, 5.6)--(11.7, 5.85)--(12.5, 5.6);
\draw[line width=1] (9.5, 6.7)--(10.75, 5.6);
\draw[line width=1] (9.6,4.4)--(10.05, 5.1)--(10.75, 4);
\draw[line width=1] (10.75, 7.2)--(11.7, 5.85)--(12, 6.7);
\draw[line width=1] (11.7, 5.85)--(12, 4.5);
\node  at (8.8, 5.6) {$a$};
\node  at (9.3, 6.9) {$b$};
\node  at (10.8, 7.4) {$a$};
\node  at (12.3, 6.9) {$b$};
\node  at (12.7, 5.6) {$c$};
\node  at (12.3, 4.3) {$d$};
\node  at (10.8, 3.8) {$c$};
\node  at (9.3, 4.3) {$d$};
\draw[fill=white] (10.05, 5.1) circle[radius=0.1];
\draw[fill=lightgray] (10.75, 5.6) circle[radius=0.1];
\draw[fill=darkgray] (11.7, 5.85) circle[radius=0.1];
%
%  third octagon
% octagon boundary
\draw (0, -0.6)--(0, 0.6)--(1, 1.6)--(2.5, 1.6)--(3.5, 0.6)--(3.5, -0.6)--(2.5, -1.6)--(1, -1.6)--(0, -0.6);
\draw[line width=1] (0, 0)--(0.75, -0.35)--(0.5, -1.1);
\draw[line width=1] (0.75, -0.35)--(1.65, -0.05)--(2.7, 0.25)--(3.5, 0);
\draw[line width=1] (0.5, 1.1)--(1.65, -0.05)--(1.75, -1.6);
\draw[line width=1] (1.75, 1.6)--(2.7, 0.25)--(3, 1.1);
\draw[line width=1] (2.7, 0.25)--(3, -1.1);
\node  at (-0.2, 0) {$a$};
\node  at (0.3, 1.3) {$b$};
\node  at (1.8, 1.8) {$a$};
\node  at (3.3, 1.3) {$b$};
\node  at (3.7, 0) {$c$};
\node  at (3.3, -1.3) {$d$};
\node  at (1.8, -1.8) {$c$};
\node  at (0.3, -1.3) {$d$};
\draw[fill=white] (0.75, -0.35) circle[radius=0.1];
\draw[fill=lightgray] (1.65, -0.05) circle[radius=0.1];
\draw[fill=darkgray] (2.7, 0.25) circle[radius=0.1];
%
% fourth octagon
% octagon boundary
\draw (4.5, -0.6)--(4.5, 0.6)--(5.5, 1.6)--(7, 1.6)--(8, 0.6)--(8, -0.6)--(7, -1.6)--(5.5, -1.6)--(4.5, -0.6);
\node  at (4.3, 0) {$a$};
\node  at (4.8, 1.3) {$b$};
\node  at (6.3, 1.8) {$a$};
\node  at (7.8, 1.3) {$b$};
\node  at (8.2, 0) {$c$};
\node  at (7.8, -1.3) {$d$};
\node  at (6.3, -1.8) {$c$};
\node  at (4.8, -1.3) {$d$};
\draw[line width=1] (4.5, -0.2)--(5.2, -0.2)--(4.5, 0.5);
\draw[line width=1] (8,0)--(7, 0)--(7.1, 1.5);
\draw[line width=1] (5, 1.1)--(5.9, -0.1)--(6.25, -1.6);
\draw[line width=1] (5.2, -0.2)--(5, -1.1);
\draw[line width=1] (7.5, 1.1)--(7, 0)--(7.5, -1.1);
\draw[line width=1] (7,0)--(6.4, 1.6);
\draw[line width=1] (5.8, 1.6)--(5.3, 1.4);
\draw[line width=1] (5.2, -0.2)--(5.9, -0.1);
\draw[fill=white] (5.2, -0.2) circle[radius=0.1];
\draw[fill=lightgray] (5.9, -0.1) circle[radius=0.1];
\draw[fill=darkgray] (7, 0) circle[radius=0.1];
%
% fifth octagon
% octagon boundary
\draw (9, -0.6)--(9, 0.6)--(10, 1.6)--(11.5, 1.6)--(12.5, 0.6)--(12.5, -0.6)--(11.5, -1.6)--(10, -1.6)--(9, -0.6);
\node  at (8.8, 0) {$a$};
\node  at (9.3, 1.3) {$b$};
\node  at (10.8, 1.8) {$a$};
\node  at (12.3, 1.3) {$b$};
\node  at (12.7, 0) {$c$};
\node  at (12.3, -1.3) {$d$};
\node  at (10.8, -1.8) {$c$};
\node  at (9.3, -1.3) {$d$};
\draw[line width=1] (9.5, 1.1)--(9.7, 0.3)--(9, 0);
\draw[line width=1] (9.7, 0.3)--(10.8, -0.4)--(11.5, 0.4);
\draw[line width=1] (10.7, -1.6)--(10.8, -0.4)--(12.2, -0.9);
\draw[line width=1] (10.7, 1.6)--(11.5, 0.4)--(11.4, 1.6);
\draw[line width=1] (12.1, 1.0)--(11.5, 0.4)--(12.5, 0);
\draw[line width=1] (9, -0.4)--(9.4, -1.0);
\draw[fill=white] (9.7, 0.3) circle[radius=0.1];
\draw[fill=lightgray] (10.8, -0.4) circle[radius=0.1];
\draw[fill=darkgray] (11.5, 0.4) circle[radius=0.1];
% sixth octagon
% octagon boundary
\draw (13.5, -0.6)--(13.5, 0.6)--(14.5, 1.6)--(16, 1.6)--(17, 0.6)--(17, -0.6)--(16, -1.6)--(14.5, -1.6)--(13.5, -0.6);
\node  at (13.3, 0) {$a$};
\node  at (13.8, 1.3) {$b$};
\node  at (15.3, 1.8) {$a$};
\node  at (16.8, 1.3) {$b$};
\node  at (17.2, 0) {$c$};
\node  at (16.8, -1.3) {$d$};
\node  at (15.3, -1.8) {$c$};
\node  at (13.8, -1.3) {$d$};
\draw[line width=1] (13.5, -0.2)--(14.2, -0.2)--(15.2, 0)--(16, 0.4)--(16.5, 1.1);
\draw[line width=1] (17,0)--(16, 0.4)--(15.8, 1.6);
\draw[line width=1] (14, 1.1)--(14.2, -0.2)--(15.25, -1.6);
\draw[line width=1] (15.2, 1.6)--(15.2, 0);
\draw[line width=1] (14.2, -0.2)--(16.5, -1.1);
\draw[line width=1] (13.5, -0.4)--(13.9, -1.0);
\draw[fill=white] (14.2, -0.2) circle[radius=0.1];
\draw[fill=lightgray] (15.2, 0) circle[radius=0.1];
\draw[fill=darkgray] (16, 0.4) circle[radius=0.1];
\node at (6.2, 3.3) {\scriptsize orientable \#1};
\node at (10.7, 3.3) {\scriptsize orientable \#2};
\node at (1.8, -2.3) {\scriptsize non-orientable \#1};
\node at (6.3, -2.3) {\scriptsize non-orientable \#2};
\node at (10.8, -2.3) {\scriptsize non-orientable \#3};
\node at (15.3, -2.3) {\scriptsize non-orientable \#4};
\end{tikzpicture}
\caption{The six inequivalent 2-cell embeddings of $T_{1,2,3}$.}
\label{T123-ineq}
\end{center}
\end{figure}

Denote by $K_4^+$ the graph obtained from $K_4$ by doubling one edge.
Euler's formula tells us that a 2-cell embedding of $K_4^+$ on the double
torus has just one face.

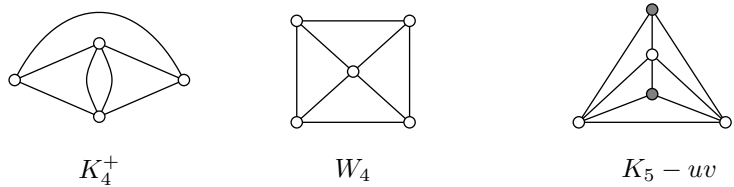
\begin{figure}[ht]
\label{K4W4}
\begin{center}
\begin{tikzpicture}[scale=0.75, line width = 0.5]
%
% K4+
\draw (0, 0)--(1.5, 0.65)--(3, 0)--(1.5, -0.65)--(0, 0);
\draw (1.5, 0.65) .. controls (1.8, 0) .. (1.5, -0.65);
\draw (1.5, 0.65) .. controls (1.2, 0) .. (1.5, -0.65);
\draw (0, 0) .. controls (0.8, 1.6) and (2.2, 1.6) .. (3, 0);
\draw[fill=white] (0, 0) circle[radius=0.1];
\draw[fill=white] (1.5, 0.65) circle[radius=0.1];
\draw[fill=white] (1.5, -0.65) circle[radius=0.1];
\draw[fill=white] (3, 0) circle[radius=0.1];
\node at (1.5, -1.6) {$K_4^+$};
%
% W4
\draw (5, -0.75)--(5, 1.05)--(7, 1.05)--(7, -0.75)--(5, -0.75)--(7, 1.05);
\draw (5, 1.05)--(7, -0.75);
\draw[fill=white] (5, -0.75) circle[radius=0.1];
\draw[fill=white] (5, 1.05) circle[radius=0.1];
\draw[fill=white] (7, 1.05) circle[radius=0.1];
\draw[fill=white] (7, -0.75) circle[radius=0.1];
\draw[fill=white] (6, 0.15) circle[radius=0.1];
\node at (6, -1.6) {$W_4$};
% K5-uv
\draw (10, -0.75)--(11.3, 1.25)--(12.6, -0.75)--(10, -0.75)--(11.3, -0.25)--(11.3, 1.25);
\draw (10, -0.75)--(11.3, 0.45)--(12.6, -0.75)--(11.3, -0.25);
\draw[fill=white] (10, -0.75) circle[radius=0.1];
\draw[fill=white] (12.6, -0.75) circle[radius=0.1];
\draw[fill=gray] (11.3, 1.25) circle[radius=0.1];
\draw[fill=gray] (11.3, -0.25) circle[radius=0.1];
\draw[fill=white] (11.3, 0.45) circle[radius=0.1];
\node at (11.6, -1.6) {$K_5-uv$};
\end{tikzpicture}
\caption{The graphs $K_4^+$, $W_4$, and $K_5-uv$.}
\label{K4W4}
\end{center}
\end{figure}

\begin{theorem}
\label{K4theorem}
Up to equivalence, $K_4^+$ has two orientable 2-cell embeddings,
and three non-orientable 2-cell embeddings on the double torus.
\end{theorem}
\begin{proof}
Consider a 2-cell embedding of $K_4^+$.  It has just one face, so that
there is no digon or triangular face. 
By Lemma~\ref{contract} and by Figure~\ref{K4W4}, we see that every 2-cell embedding of $K_4^+$ has four 
edges whose contraction 
results in a 2-cell embedding of $T_{1,2,3}$.  The edge that was contracted
maps to the unique vertex of $T_{1,2,3}$ of degree five.  Consequently, every 2-cell
embedding of $K_4^+$ can be obtained from some 2-cell embedding of 
$T_{1,2,3}$ by reversing an edge-contraction using the vertex of
degree five.   Starting from the
embeddings of Figure~\ref{T123-ineq}, there are at most five ways of splitting the vertex of
degree five into vertices of degree three and four.  The vertex must be split so that
the triple and double edge result in one double edge, and various single edges. 
There are at most 30 embeddings that can be obtained like this.
Using isomorphism testing software to distinguish them,
we find that there are two orientable embeddings and three non-orientable embeddings,
shown in Figures~\ref{K4plusorient} and~\ref{K4plusnonorient}.
\end{proof}

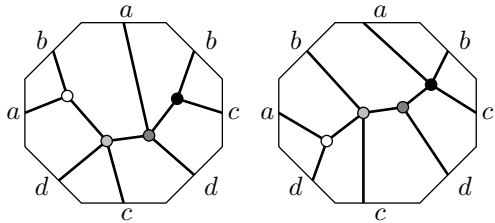
\begin{figure}[ht]
\begin{center}
\begin{tikzpicture}[scale=0.75, line width = 0.5]
% octagon boundary
\draw (0, -0.6)--(0, 0.6)--(1, 1.6)--(2.5, 1.6)--(3.5, 0.6)--(3.5, -0.6)--(2.5, -1.6)--(1, -1.6)--(0, -0.6);
\node  at (-0.2, 0) {$a$};
\node  at (0.3, 1.3) {$b$};
\node  at (1.8, 1.8) {$a$};
\node  at (3.3, 1.3) {$b$};
\node  at (3.7, 0) {$c$};
\node  at (3.3, -1.3) {$d$};
\node  at (1.8, -1.8) {$c$};
\node  at (0.3, -1.3) {$d$};
\draw[line width=1] (0, 0)--(0.75, 0.3)--(1.45, -0.5)--(2.2, -0.4)--(2.7, 0.25)--(3.5, 0);
\draw[line width=1] (0.5,1.1)--(0.75, 0.3);
\draw[line width=1] (0.6,-1.2)--(1.45, -0.5)--(1.75, -1.6);
\draw[line width=1] (1.75, 1.6)--(2.2, -0.4)--(3.0, -1.1);
\draw[line width=1] (3.0, 1.1)--(2.7, 0.25);
%
%\coordinate u (0.75, 0.3);
%\coordinate v (1.65, -0.5);
%\coordinate w (2.7, 0.25);
%\draw[fill=white] (u) circle[radius=0.1];
%\draw[fill=lightgray] (v) circle[radius=0.1];
%\draw[fill=darkgray] (w) circle[radius=0.1];
\draw[fill=white] (0.75, 0.3) circle[radius=0.1];
\draw[fill=lightgray] (1.45, -0.5) circle[radius=0.1];
\draw[fill=gray] (2.2, -0.4) circle[radius=0.1];
\draw[fill=black] (2.7, 0.25) circle[radius=0.1];
%
%  second octagon
% octagon boundary
\draw (4.5, -0.6)--(4.5, 0.6)--(5.5, 1.6)--(7, 1.6)--(8, 0.6)--(8, -0.6)--(7, -1.6)--(5.5, -1.6)--(4.5, -0.6);
\draw[line width=1] (4.5, 0)--(5.35, -0.5)--(6.0, 0)--(6.7, 0.1)--(7.2, 0.5)--(8, 0);
\draw[line width=1] (5, 1.1)--(6.0, 0)--(6.0, -1.6);
\draw[line width=1] (5.1,-1.2)--(5.35, -0.5);
\draw[line width=1] (6.0, 1.6)--(7.2, 0.5)--(7.5, 1.1);
\draw[line width=1] (6.7, 0.1)--(7.5, -1.1);
\node  at (4.3, 0) {$a$};
\node  at (4.8, 1.3) {$b$};
\node  at (6.3, 1.8) {$a$};
\node  at (7.8, 1.3) {$b$};
\node  at (8.2, 0) {$c$};
\node  at (7.8, -1.3) {$d$};
\node  at (6.3, -1.8) {$c$};
\node  at (4.8, -1.3) {$d$};
\draw[fill=white] (5.35, -0.5) circle[radius=0.1];
\draw[fill=lightgray] (6.0, 0) circle[radius=0.1];
\draw[fill=gray] (6.7, 0.1) circle[radius=0.1];
\draw[fill=black] (7.2, 0.5) circle[radius=0.1];
\end{tikzpicture}
\caption{The two inequivalent orientable 2-cell embeddings of $K_4^+$.}
\label{K4plusorient}
\end{center}
\end{figure}

\begin{figure}[ht]
\begin{center}
\begin{tikzpicture}[scale=0.75, line width = 0.5]
% octagon boundary
\draw (0, -0.6)--(0, 0.6)--(1, 1.6)--(2.5, 1.6)--(3.5, 0.6)--(3.5, -0.6)--(2.5, -1.6)--(1, -1.6)--(0, -0.6);
\node  at (-0.2, 0) {$a$};
\node  at (0.3, 1.3) {$b$};
\node  at (1.8, 1.8) {$a$};
\node  at (3.3, 1.3) {$b$};
\node  at (3.7, 0) {$c$};
\node  at (3.3, -1.3) {$d$};
\node  at (1.8, -1.8) {$c$};
\node  at (0.3, -1.3) {$d$};
\draw[line width=1] (0, 0)--(0.75, 0.3)--(1.45, -0.5)--(2.15, -0.2)--(2.85, 0)--(3.5, 0);
\draw[line width=1] (0.5,1.1)--(0.75, 0.3);
\draw[line width=1] (0.6,-1.2)--(1.45, -0.5)--(1.75, -1.6);
\draw[line width=1] (2.85, 0)--(3.0, 1.1);
\draw[line width=1] (1.75, 1.6)--(2.15, -0.2);
\draw[line width=1] (2.85, 0)--(3.0, -1.1);
%
%\coordinate u (0.75, 0.3);
%\coordinate v (1.65, -0.5);
%\coordinate w (2.7, 0.25);
%\draw[fill=white] (u) circle[radius=0.1];
%\draw[fill=lightgray] (v) circle[radius=0.1];
%\draw[fill=darkgray] (w) circle[radius=0.1];
\draw[fill=white] (0.75, 0.3) circle[radius=0.1];
\draw[fill=lightgray] (1.45, -0.5) circle[radius=0.1];
\draw[fill=gray] (2.15, -0.2) circle[radius=0.1];
\draw[fill=black] (2.85, 0) circle[radius=0.1];
%
%  second octagon
% octagon boundary
\draw (4.5, -0.6)--(4.5, 0.6)--(5.5, 1.6)--(7, 1.6)--(8, 0.6)--(8, -0.6)--(7, -1.6)--(5.5, -1.6)--(4.5, -0.6);
\draw[line width=1] (4.5, 0)--(5.25, 0)--(5.95, -0.1)--(6.75, 0)--(7.35, 0.5)--(8, 0);
\draw[line width=1] (5, 1.1)--(5.25, 0);
\draw[line width=1] (5.1,-1.2)--(5.25, 0);
\draw[line width=1] (5.95, -0.1)--(6.25, -1.6);
\draw[line width=1] (6.25, 1.6)--(7.35, 0.5)--(7.5, 1.1);
\draw[line width=1] (6.75, 0)--(7.5, -1.1);
\node  at (4.3, 0) {$a$};
\node  at (4.8, 1.3) {$b$};
\node  at (6.3, 1.8) {$a$};
\node  at (7.8, 1.3) {$b$};
\node  at (8.2, 0) {$c$};
\node  at (7.8, -1.3) {$d$};
\node  at (6.3, -1.8) {$c$};
\node  at (4.8, -1.3) {$d$};
\draw[fill=white] (5.25, 0) circle[radius=0.1];
\draw[fill=lightgray] (5.95, -0.1) circle[radius=0.1];
\draw[fill=gray] (6.75, 0) circle[radius=0.1];
\draw[fill=black] (7.35, 0.5) circle[radius=0.1];
%
%  third octagon
% octagon boundary
\draw (9, -0.6)--(9, 0.6)--(10, 1.6)--(11.5, 1.6)--(12.5, 0.6)--(12.5, -0.6)--(11.5, -1.6)--(10, -1.6)--(9, -0.6);
\draw[line width=1] (9, 0)--(9.75, 0)--(10.45, 0)--(11.15, 0)--(11.85, 0)--(12.5, 0);
\draw[line width=1] (9.5, 1.1)--(9.75, 0)--(9.5, -1.1);
\draw[line width=1] (10.75, -1.6)--(10.45, 0);
\draw[line width=1] (10.75, 1.6)--(11.15, 0);
\draw[line width=1] (12, 1.1)--(11.85, 0)--(12, -1.1);
\node  at (8.8, 0) {$a$};
\node  at (9.3, 1.3) {$b$};
\node  at (10.8, 1.8) {$a$};
\node  at (12.3, 1.3) {$b$};
\node  at (12.7, 0) {$c$};
\node  at (12.3, -1.3) {$d$};
\node  at (10.8, -1.8) {$c$};
\node  at (9.3, -1.3) {$d$};
\draw[fill=white] (9.75, 0) circle[radius=0.1];
\draw[fill=lightgray] (10.45, 0) circle[radius=0.1];
\draw[fill=gray] (11.15, 0) circle[radius=0.1];
\draw[fill=black] (11.85, 0) circle[radius=0.1];
\end{tikzpicture}
\caption{The three inequivalent non-orientable 2-cell embeddings of $K_4^+$.}
\label{K4plusnonorient}
\end{center}
\end{figure}
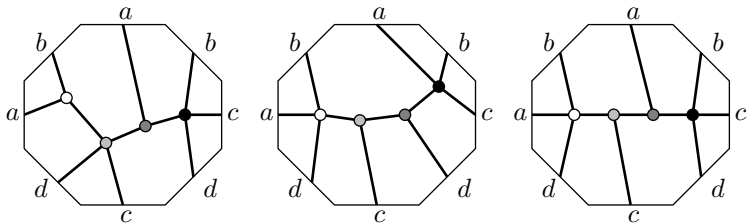

\bigbreak
Consider now a 2-cell embedding of the wheel graph $W_4$, shown in Figure~\ref{K4W4}.

\begin{theorem}
\label{W4thm}
Up to equivalence, $W_4$ has one orientable 2-cell embedding,
and three non-orientable 2-cell embeddings on the double torus.
\end{theorem}
\begin{proof}
Consider a 2-cell embedding of $W_4$.   By Euler's formula, it has just
one face.
From Figure~\ref{K4W4}, we see that contracting any one of the four ``peripheral'' edges 
results in  $K_4^+$.  By Lemma~\ref{contract}, this will be a 2-cell embedding of $K_4^+$.
Consequently, every 2-cell embedding of $W_4$ can be obtained from a 2-cell
embedding of  $K_4^+$.
The contracted edge maps onto one of the vertices of $K_4^+$ of degree four.
Using the five embeddings of $K_4^+$ given by Theorem~\ref{K4theorem}, we split each
of the vertices of degree four in all possible ways that result in $W_4$.
Then using isomorphism testing software to distinguish the resulting embeddings, 
we find that there is one orientable embedding and three non-orientable embeddings. 
They are shown in Figure~\ref{wheel4}.
\end{proof}

\begin{figure}[ht]
\begin{center}
\begin{tikzpicture}[scale=0.75, line width = 0.5]
% octagon boundary
\draw (0, -0.6)--(0, 0.6)--(1, 1.6)--(2.5, 1.6)--(3.5, 0.6)--(3.5, -0.6)--(2.5, -1.6)--(1, -1.6)--(0, -0.6);
\node  at (-0.2, 0) {$a$};
\node  at (0.3, 1.3) {$b$};
\node  at (1.8, 1.8) {$a$};
\node  at (3.3, 1.3) {$b$};
\node  at (3.7, 0) {$c$};
\node  at (3.3, -1.3) {$d$};
\node  at (1.8, -1.8) {$c$};
\node  at (0.3, -1.3) {$d$};
\draw[line width=1] (0, 0)--(0.75, 0.3)--(1.15, -0.4)--(1.75, -0.4)--(2.3, 0.1)--(2.9, 0.15)--(3.5, 0);
\draw[line width=1] (0.5,1.1)--(0.75, 0.3);
\draw[line width=1] (0.6,-1.2)--(1.15, -0.4);
\draw[line width=1] (1.75, -0.4)--(1.75, -1.6);
\draw[line width=1] (1.75, 1.6)--(2.3, 0.1);
\draw[line width=1] (3.0, -1.1)--(2.9, 0.15)--(3.0, 1.1);
\draw[fill=white] (0.75, 0.3) circle[radius=0.1];
\draw[fill=white] (1.15, -0.4) circle[radius=0.1];
\draw[fill=white] (1.75, -0.4) circle[radius=0.1];
\draw[fill=white] (2.3, 0.1) circle[radius=0.1];
\draw[fill=gray] (2.9, 0.15) circle[radius=0.1];
%
%  second octagon
% octagon boundary
\draw (4.5, -0.6)--(4.5, 0.6)--(5.5, 1.6)--(7, 1.6)--(8, 0.6)--(8, -0.6)--(7, -1.6)--(5.5, -1.6)--(4.5, -0.6);
\draw[line width=1] (4.5, 0)--(5.3, -0.5)--(5.7, 0)--(6.25, -0.2)--(6.85, -0.2)--(7.3, 0.45)--(8, 0);
\draw[line width=1] (5, 1.1)--(5.7, 0);
\draw[line width=1] (5.1,-1.2)--(5.3, -0.5);
\draw[line width=1] (6.25, -0.2)--(6.25, -1.6);
\draw[line width=1] (6.25, 1.6)--(7.3, 0.45)--(7.5, 1.1);
\draw[line width=1] (6.85, -0.2)--(7.5, -1.1);
\node  at (4.3, 0) {$a$};
\node  at (4.8, 1.3) {$b$};
\node  at (6.3, 1.8) {$a$};
\node  at (7.8, 1.3) {$b$};
\node  at (8.2, 0) {$c$};
\node  at (7.8, -1.3) {$d$};
\node  at (6.3, -1.8) {$c$};
\node  at (4.8, -1.3) {$d$};
\draw[fill=white] (5.3, -0.5) circle[radius=0.1];
\draw[fill=white] (5.7, 0) circle[radius=0.1];
\draw[fill=white] (6.25, -0.2) circle[radius=0.1];
\draw[fill=white] (6.85, -0.2) circle[radius=0.1];
\draw[fill=gray] (7.3, 0.45) circle[radius=0.1];
%
%  third octagon
% octagon boundary
\draw (9, -0.6)--(9, 0.6)--(10, 1.6)--(11.5, 1.6)--(12.5, 0.6)--(12.5, -0.6)--(11.5, -1.6)--(10, -1.6)--(9, -0.6);
\draw[line width=1] (9, 0)--(9.75, -0.35)--(10.5, -0.1)--(11.1, 0.05)--(11.5, 0.6)--(12, 1.1);
\draw[line width=1] (9.5, 1.1)--(10.5, -0.1)--(10.75, -1.6);
\draw[line width=1] (9.5, -1.1)--(9.75, -0.35);
\draw[line width=1] (10.75, 1.6)--(11.5, 0.6);
\draw[line width=1] (11.1, 0.05)--(11.6, -0.3)--(12.5, 0);
\draw[line width=1] (11.6, -0.3)--(12, -1.1);
\node  at (8.8, 0) {$a$};
\node  at (9.3, 1.3) {$b$};
\node  at (10.8, 1.8) {$a$};
\node  at (12.3, 1.3) {$b$};
\node  at (12.7, 0) {$c$};
\node  at (12.3, -1.3) {$d$};
\node  at (10.8, -1.8) {$c$};
\node  at (9.3, -1.3) {$d$};
\draw[fill=white] (9.75, -0.35) circle[radius=0.1];
\draw[fill=gray] (10.5, -0.1) circle[radius=0.1];
\draw[fill=white] (11.1, 0.05) circle[radius=0.1];
\draw[fill=white] (11.6, -0.3) circle[radius=0.1];
\draw[fill=white] (11.5, 0.6) circle[radius=0.1];
%
% fourth octagon
% octagon boundary
\draw (13.5, -0.6)--(13.5, 0.6)--(14.5, 1.6)--(16, 1.6)--(17, 0.6)--(17, -0.6)--(16, -1.6)--(14.5, -1.6)--(13.5, -0.6);
\node  at (13.3, 0) {$a$};
\node  at (13.8, 1.3) {$b$};
\node  at (15.3, 1.8) {$a$};
\node  at (16.8, 1.3) {$b$};
\node  at (17.2, 0) {$c$};
\node  at (16.8, -1.3) {$d$};
\node  at (15.3, -1.8) {$c$};
\node  at (13.8, -1.3) {$d$};
\draw[line width=1] (13.5, 0)--(14.2, 0.3)--(14.7, -0.1)--(15.3, -0.3)--(16.0, -0.1)--(16.5, 0.3)--(17, 0);
\draw[line width=1] (14, 1.1)--(14.2, 0.3);
\draw[line width=1] (14.7, -0.1)--(14, -1.1);
\draw[line width=1] (16.5, 1.1)--(16.5, 0.3);
\draw[line width=1] (15.25, 1.6)--(16.0, -0.1)--(16.5, -1.1);
\draw[line width=1] (15.3, -0.3)--(15.25, -1.6);
\draw[fill=white] (14.2, 0.3) circle[radius=0.1];
\draw[fill=white] (14.7, -0.1) circle[radius=0.1];
\draw[fill=white] (15.3, -0.3) circle[radius=0.1];
\draw[fill=gray] (16.0, -0.1) circle[radius=0.1];
\draw[fill=white] (16.5, 0.3) circle[radius=0.1];

\node at (1.7, -2.3) {\footnotesize orientable};
\node at (6.2, -2.3) {\scriptsize non-orientable \#1};
\node at (10.8, -2.3) {\scriptsize non-orientable \#2};
\node at (15.3, -2.3) {\scriptsize non-orientable \#3};
\end{tikzpicture}
\caption{The four inequivalent 2-cell embeddings of $W_4$.}
\label{wheel4}
\end{center}
\end{figure}
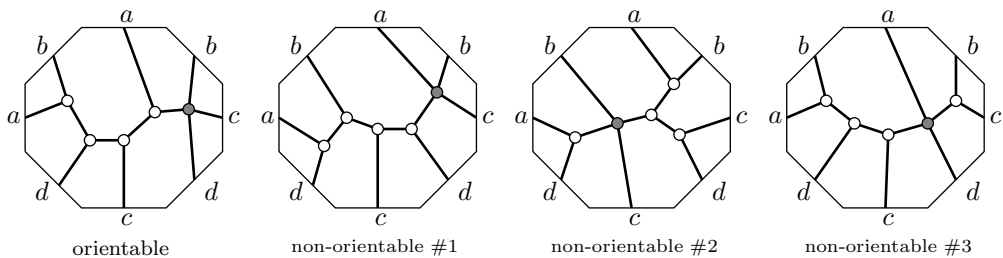

\bigbreak

We now turn to $K_5$.  

\begin{theorem}
\label{K5thm}
Up to equivalence, $K_5$ has $14$ orientable and 
$17$ non-orientable $2$-cell embeddings on the double torus.
\end{theorem}
\begin{proof}
A 2-cell embedding  of $K_5$ on the double torus has three faces.  Therefore 
in every 2-cell embedding, there
is at least one edge $uv$ that has different faces on its two sides.
If we delete this edge, we obtain a 2-cell embedding of $K_5-uv$,
an embedding with two faces.   Therefore $K_5-uv$ also has at least one
edge with different faces on its two sides.   By removing such
an edge, we obtain a 2-cell embedding of a graph with $f=1$.
$K_5-uv$ is shown in Figure~\ref{K4W4}, where $u$ and $v$ are shaded grey.  
It consists of a triangle $(x,y,z)$, with vertices $u$ and $v$ both adjacent
to each of $\{x,y,z\}$.   Consequently it has just two kinds of edges --- those
on the triangle $(x,y,z)$, and those incident on $u$ or $v$.
If an edge of the triangle $(x,y,z)$ is deleted, the result is the wheel graph $W_4$.
If an edge incident on $u$ or $v$ is deleted, the resulting graph is
homeomorphic to $K_4^+$, having a vertex of degree two.   Either case is possible.

We start with the 2-cell embeddings of $W_4$ and $K_4^+$.   With $W_4$, we add an edge
between two non-adjacent vertices in all possible ways, to obtain a 2-cell 
embedding of $K_5-uv$.  There are many ways of doing this, because a 2-cell
embedding of $W_4$ has a single facial cycle of length 16.  Each vertex of degree three
appears three times on the facial cycle, so that there are 18 ways of 
adding an edge to each embedding of $W_4$.   This gives $72$ embeddings of $K_5-uv$
derived from $W_4$.

With $K_4^+$, we subdivide one of the pair of double edges with a new vertex $x$, and then add an edge 
$xu$, where $u$ is a vertex of degree three.  There are $24$ ways to do this for each
$2$-cell embedding of $K_4^+$, giving $120$ more embeddings of $K_5-uv$.

The embeddings of $K_5-uv$ are then distinguished with isomorphism testing software,
giving $60$ non-isomorphic embeddings in total, which reduces to $21$ orientable 
and $18$ non-orientable $2$-cell 
embeddings of $K_5-uv$.  Finally,
the deleted edge $uv$ is added to each embedding of $K_5-uv$ in all possible ways.
There can be numerous ways to add $uv$ to each embedding.
The result is $45$ non-isomorphic embeddings in total, which reduces to $14$ orientable 
and $17$ non-orientable $2$-cell embeddings of $K_5$.
\end{proof}

\bigbreak
The addition of edges to $K_4^+$, $W_4$, and $K_5-uv$ in all possible ways 
in Theorem~\ref{K5thm} was done by using  a computer program. 
A list of rotation systems for the resulting $31$ inequivalent 
$2$-cell embeddings of $K_5$ is given in Appendix~A.
Most of the $31$ embeddings ($27$) have an automorphism group of order one (for example, see Figure~\ref{K5fig01} below). 
There is one embedding with
a group of order five, two with a group of order four, and one with a group of order two.

%%%

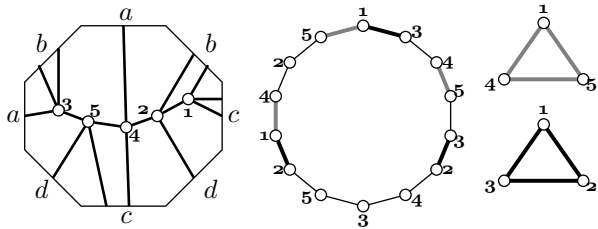
\begin{figure}[ht]
\begin{center}
\begin{tikzpicture}[scale=0.75, line width = 0.5]
%octagon boundary
\draw (0, -0.6)--(0, 0.6)--(1, 1.6)--(2.5, 1.6)--(3.5, 0.6)--(3.5, -0.6)--(2.5, -1.6)--(1, -1.6)--(0, -0.6);
\node  at (-0.2, 0) {$a$};
\node  at (0.3, 1.3) {$b$};
\node  at (1.8, 1.8) {$a$};
\node  at (3.3, 1.3) {$b$};
\node  at (3.7, 0) {$c$};
\node  at (3.3, -1.3) {$d$};
\node  at (1.8, -1.8) {$c$};
\node  at (0.3, -1.3) {$d$};
\node  at (0.75, 0.2) {\smallmath 3};
\node  at (1.25, 0.1) {\smallmath 5};
\node  at (1.95, -0.35) {\smallmath 4};
\node  at (2.1, 0.1) {\smallmath 2};
\node  at (2.9, 0.05) {\smallmath 1};
\draw[line width=1] (0,0)--(0.6, 0.1)--(1.125, -0.1)--(1.8, -0.2)--(2.35, 0)--(2.9, 0.3)--(3.5, 0);
\draw[line width=1] (0.6, 1.2)--(0.6, 0.1)--(0.25, 0.9); % left b edge to a to left b edge
\draw[line width=1] (0.5, -1.1)--(1.125, -0.1)--(1.45, -1.6); % left d edge to 5 to bottom c edge
\draw[line width=1] (1.75, 1.6)--(1.8, -0.2)--(1.85, -1.6); % top b to 4 to bottom c
\draw[line width=1] (3.0, 1.1)--(2.35, 0)--(3.0, -1.1); % right b to 2 to right d
\draw[line width=1] (3.25, 0.9)--(2.9, 0.3)--(3.5, 0.3); % right b edge to 1 to right c edge
\draw[fill=white] (0.6, 0.1) circle[radius=0.1];
\draw[fill=white] (1.125, -0.1) circle[radius=0.1];
\draw[fill=white] (1.8, -0.2) circle[radius=0.1];
\draw[fill=white] (2.35, 0) circle[radius=0.1];
\draw[fill=white] (2.9, 0.3) circle[radius=0.1];
%
%\draw[fill=white] (6, 0) circle[radius=1.6]; % 14-gon
%\draw[fill=white] (9.2, 0.9) circle[radius=0.75]; % 3-gon
%\draw[fill=white] (9.2, -0.9) circle[radius=0.75]; % 3-gon
%
\draw[line width=1.5, black] (6.0, 1.6)--(6.75, 1.4);
\draw[line width=1.5, black] (7.55, -0.35)--(7.3, -0.95);
\draw[line width=1.5, black] (4.7, -0.95)--(4.45, -0.35);
\draw[line width=1.5, gray] (5.25, 1.4)--(6.0, 1.6);
\draw[line width=1.5, gray] (7.3, 0.95)--(7.55, 0.35);
\draw[line width=1.5, gray] (4.45, -0.35)--(4.45, 0.35);
\draw (6.75, 1.4)--(7.3, 0.95);
\draw (7.55, 0.35)--(7.55, -0.35);
\draw (7.3, -0.95)--(6.75, -1.4)--(6.0, -1.6)--(5.25, -1.4)--(4.7, -0.95);
\draw (4.45, 0.35)--(4.7, 0.95)--(5.25, 1.4);
\draw[fill=white] (6.0, 1.6) circle[radius=0.1];
\draw[fill=white] (6.75, 1.4) circle[radius=0.1];
\draw[fill=white] (7.3, 0.95) circle[radius=0.1];
\draw[fill=white] (7.55, 0.35) circle[radius=0.1];
\draw[fill=white] (7.55, -0.35) circle[radius=0.1];
\draw[fill=white] (7.3, -0.95) circle[radius=0.1];
\draw[fill=white] (6.75, -1.4) circle[radius=0.1];
\draw[fill=white] (6.0, -1.6) circle[radius=0.1];
\draw[fill=white] (5.25, -1.4) circle[radius=0.1];
\draw[fill=white] (4.7, -0.95) circle[radius=0.1];
\draw[fill=white] (4.45, -0.35) circle[radius=0.1];
\draw[fill=white] (4.45, 0.35) circle[radius=0.1];
\draw[fill=white] (4.7, 0.95) circle[radius=0.1];
\draw[fill=white] (5.25, 1.4) circle[radius=0.1];
\draw[line width=1.5, gray] (9.2, 1.65)--(9.9, 0.65)--(8.5, 0.65)--(9.2, 1.65);
\draw[fill=white] (9.2, 1.65) circle[radius=0.1];
\draw[fill=white] (9.9, 0.65) circle[radius=0.1];
\draw[fill=white] (8.5, 0.65) circle[radius=0.1];
\draw[line width=1.5, black] (9.2, -0.15)--(9.9, -1.15)--(8.5, -1.15)--(9.2, -0.15);
\draw[fill=white] (9.2, -0.15) circle[radius=0.1];
\draw[fill=white] (9.9, -1.15) circle[radius=0.1];
\draw[fill=white] (8.5, -1.15) circle[radius=0.1];
\node  at (6.0, 1.85) {\smallmath 1};
\node  at (6.95, 1.5) {\smallmath 3};
\node  at (7.5, 0.95) {\smallmath 4};
\node  at (7.7, 0.45) {\smallmath 5};
\node  at (7.7, -0.45) {\smallmath 3};
\node  at (7.5, -0.95) {\smallmath 2};
\node  at (6.95, -1.5) {\smallmath 4};
\node  at (6.0, -1.85) {\smallmath 3};
\node  at (5.0, -1.45) {\smallmath 5};
\node  at (4.5, -0.95) {\smallmath 2};
\node  at (4.25, -0.35) {\smallmath 1};
\node  at (4.25, 0.35) {\smallmath 4};
\node  at (4.5, 0.95) {\smallmath 2};
\node  at (5.0, 1.45) {\smallmath 5};
\node  at (9.2, 1.9) {\smallmath 1};
\node  at (10.05, 0.55) {\smallmath 5};
\node  at (8.25, 0.55) {\smallmath 4};
\node  at (9.2, 0.1) {\smallmath 1};
\node  at (10.05, -1.25) {\smallmath 2};
\node  at (8.25, -1.25) {\smallmath 3};
\end{tikzpicture}
\caption{An orientable embedding of $K_5$ with the trivial automorphism group.}
\label{K5fig01}
\end{center}
\end{figure}

%%%%%

Several of the embeddings are interesting.
One of the non-orientable embeddings has an automorphism  group of order five.  It is shown in Figure~\ref{K5fig43},
together with its three facial cycles.
Its automorphism group is generated
by the permutation $(1,2,3,5,4)$.  The three faces determine a decomposition of
the double torus, as well as a tiling of the hyperbolic plane by decagons and pentagons,
in which two pentagons and two decagons meet at each corner.

\begin{figure}[ht]
\begin{center}
\begin{tikzpicture}[scale=0.75, line width = 0.5]
%octagon boundary
\draw (0, -0.6)--(0, 0.6)--(1, 1.6)--(2.5, 1.6)--(3.5, 0.6)--(3.5, -0.6)--(2.5, -1.6)--(1, -1.6)--(0, -0.6);
\node  at (-0.2, 0) {$a$};
\node  at (0.3, 1.3) {$b$};
\node  at (1.8, 1.8) {$a$};
\node  at (3.3, 1.3) {$b$};
\node  at (3.7, 0) {$c$};
\node  at (3.3, -1.3) {$d$};
\node  at (1.8, -1.8) {$c$};
\node  at (0.3, -1.3) {$d$};
\node  at (0.75, 0.2) {\smallmath 1};
\node  at (1.3, 0.0) {\smallmath 5};
\node  at (1.8, -0.15) {\smallmath 3};
\node  at (2.1, 0.1) {\smallmath 2};
\node  at (2.9, 0.05) {\smallmath 4};
\draw[line width=1] (0,0)--(0.6, 0)--(1.125, -0.2)--(1.8, -0.4)--(2.35, 0)--(2.9, 0.3)--(3.5, 0);
\draw[line width=1] (0.5, 1.1)--(0.6, 0)--(0.05, -0.7);
\draw[line width=1] (1, 1.6)--(1.125, -0.2)--(0.55, -1.15);
\draw[line width=1] (1.5, -1.6)--(1.8, -0.4)--(2.45, -1.6);
\draw[line width=1] (1.75, 1.6)--(2.35, 0)--(3.0, -1.1);
\draw[line width=1] (3.1, 1.0)--(2.9, 0.3)--(3.5, 0.6);
\draw[line width=1] (3.3, -0.8)--(3.5, -0.35);
\draw[fill=white] (0.6, 0) circle[radius=0.1];
\draw[fill=white] (1.125, -0.2) circle[radius=0.1];
\draw[fill=white] (1.8, -0.4) circle[radius=0.1];
\draw[fill=white] (2.35, 0) circle[radius=0.1];
\draw[fill=white] (2.9, 0.3) circle[radius=0.1];
%
%\draw[fill=white] (6, 0) circle[radius=1.6]; % 10-gon
%\draw[fill=white] (9.2, 0.9) circle[radius=0.75]; % 5-gon
%\draw[fill=white] (9.2, -0.9) circle[radius=0.75]; % 5-gon
%
\draw[line width=1, black] (6.0, 1.6)--(6.9, 1.3);
\draw[line width=1, black] (7.49, 0.53)--(7.49, -0.53);
\draw[line width=1, black] (6.9, -1.3)--(6.0, -1.6);
\draw[line width=1, black] (5.1, -1.3)--(4.51, -0.53);
\draw[line width=1, black] (4.51, 0.53)--(5.1, 1.3);
%;
\draw[line width=1, gray] (6.9, 1.3)--(7.49, 0.53);
\draw[line width=1, gray] (7.49, -0.53)--(6.9, -1.3);
\draw[line width=1, gray] (6.0, -1.6)--(5.1, -1.3);
\draw[line width=1, gray] (4.51, -0.53)--(4.51, 0.53);
\draw[line width=1, gray] (5.1, 1.3)--(6.0, 1.6);

\draw[fill=white] (6.0, 1.6) circle[radius=0.1];
\draw[fill=white] (6.0, -1.6) circle[radius=0.1];
\draw[fill=white] (7.49, 0.53) circle[radius=0.1];
\draw[fill=white] (6.9, 1.3) circle[radius=0.1];
\draw[fill=white] (7.49, -0.53) circle[radius=0.1];
\draw[fill=white] (6.9, -1.3) circle[radius=0.1];
\draw[fill=white] (4.51, 0.53) circle[radius=0.1];
\draw[fill=white] (5.1, 1.3) circle[radius=0.1];
\draw[fill=white] (4.51, -0.53) circle[radius=0.1];
\draw[fill=white] (5.1, -1.3) circle[radius=0.1];
\draw[line width=1, gray] (9.2, 1.65)--(9.88, 1.2)--(9.78, 0.45)--(8.62, 0.45)--(8.52, 1.2)--(9.2, 1.65);
\draw[fill=white] (9.2, 1.65) circle[radius=0.1];
\draw[fill=white] (9.88, 1.2) circle[radius=0.1];
\draw[fill=white] (9.78, 0.45) circle[radius=0.1];
\draw[fill=white] (8.62, 0.45) circle[radius=0.1];
\draw[fill=white] (8.52, 1.2) circle[radius=0.1];
\draw[line width=1, black] (9.2, -0.15)--(9.88, -0.6)--(9.78, -1.35)--(8.62, -1.35)--(8.52, -0.6)--(9.2, -0.15);
\draw[fill=white] (9.2, -0.15) circle[radius=0.1];
\draw[fill=white] (9.88, -0.6) circle[radius=0.1];
\draw[fill=white] (9.78, -1.35) circle[radius=0.1];
\draw[fill=white] (8.62, -1.35) circle[radius=0.1];
\draw[fill=white] (8.52, -0.6) circle[radius=0.1];
\node  at (6.0, 1.85) {\smallmath 1};
\node  at (7.0, 1.45) {\smallmath 4};
\node  at (7.65, 0.63) {\smallmath 2};
\node  at (7.65, -0.63) {\smallmath 1};
\node  at (7.0, -1.45) {\smallmath 3};
\node  at (6.0, -1.8) {\smallmath 2};
\node  at (4.85, -1.4) {\smallmath 5};
\node  at (4.30, -0.6) {\smallmath 3};
\node  at (4.30, 0.6) {\smallmath 4};
\node  at (4.85, 1.4) {\smallmath 5};
\node  at (9.2, 1.85) {\smallmath 1};
\node  at (10.05, 1.3) {\smallmath 5};
\node  at (9.95, 0.30) {\smallmath 2};
\node  at (8.45, 0.30) {\smallmath 4};
\node  at (8.3, 1.25) {\smallmath 3};
\node  at (9.2, 0.05) {\smallmath 1};
\node  at (10.05, -0.45) {\smallmath 2};
\node  at (9.95, -1.5) {\smallmath 3};
\node  at (8.45, -1.5) {\smallmath 5};
\node  at (8.3, -0.45) {\smallmath 4};
\end{tikzpicture}
\caption{The non-orientable embedding of $K_5$ with the automorphism group of order $5$.}
\label{K5fig43}
\end{center}
\end{figure}
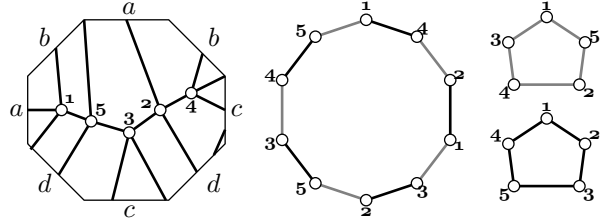

The 10-gon representations of Figure~\ref{hyperbolic} for the double torus are often more suitable for
drawing embeddings of $K_5$ than the standard octagon form of the double torus.  Figure~\ref{10gonK5}
shows a drawing of $K_5$ on two of the 10-gon representations.  As can be seen from the drawing, these
embeddings have an automorphism group of order five. They are isomorphic embeddings, 
although this is not evident from the diagram.

%
%%%%%%%%

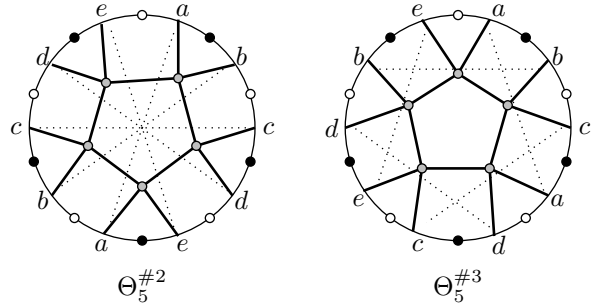
\begin{figure}[ht]
\begin{center}
\begin{tikzpicture}[scale=0.6, line width = 0.5]
% 10-gon boundary
% 10-gon boundary
\draw[fill=white] (7, 0) circle[radius=2.5];
\node  at (7.9, 2.6) {$a$};
\node  at (9.2, 1.6) {$b$};
\node  at (9.8, 0) {$c$};
\node  at (9.2, -1.6) {$d$};
\node  at (7.9, -2.6) {$e$};
\node  at (6.1, -2.6) {$a$};
\node  at (4.8, -1.6) {$b$};
\node  at (4.2, 0) {$c$};
\node  at (4.8, 1.6) {$d$};
\node  at (6.1, 2.6) {$e$};
\draw[dotted, line width=0.5] (7.7, 2.1)--(6.3, -2.1); % a
\draw[dotted, line width=0.5] (8.8, 1.3)--(5.2, -1.3); % b
\draw[dotted, line width=0.5] (9.3, 0)--(4.7, 0); % c
\draw[dotted, line width=0.5] (8.9, -1.4)--(5.2, 1.3); % d
\draw[dotted, line width=0.5] (7.7, -2.2)--(6.3, 2.2); % e
\draw[fill=white] (7, 2.5) circle[radius=0.1];
\draw[fill=black] (7, -2.5) circle[radius=0.1];
\draw[fill=white] (9.4, 0.75) circle[radius=0.1];
\draw[fill=black] (9.4, -0.75) circle[radius=0.1];
\draw[fill=white] (4.6, 0.75) circle[radius=0.1];
\draw[fill=black] (4.6, -0.75) circle[radius=0.1];
\draw[fill=black] (8.5, 2) circle[radius=0.1];
\draw[fill=white] (8.5, -2) circle[radius=0.1];
\draw[fill=black] (5.5, 2) circle[radius=0.1];
\draw[fill=white] (5.5, -2) circle[radius=0.1];
\draw[line width=1] (8.2, -0.4)--(7, -1.3)--(5.8, -0.4)--(6.2, 1)--(7.8, 1.1)--(8.2, -0.4); % pentagon
\draw[line width=1] (9.5, 0)--(8.2, -0.4)--(9.0, -1.5);
\draw[line width=1] (7.8, -2.4)--(7, -1.3)--(6.15, -2.35);
\draw[line width=1] (5.0, -1.5)--(5.8, -0.4)--(4.5, 0);
\draw[line width=1] (5.0, 1.4)--(6.2, 1)--(6.1, 2.3);
\draw[line width=1] (7.8, 2.4)--(7.8, 1.1)--(9.05, 1.4);
\draw[fill=lightgray] (8.2, -0.4) circle[radius=0.1];
\draw[fill=lightgray] (7, -1.3) circle[radius=0.1];
\draw[fill=lightgray] (5.8, -0.4) circle[radius=0.1];
\draw[fill=lightgray] (6.2, 1) circle[radius=0.1];
\draw[fill=lightgray] (7.8, 1.1) circle[radius=0.1];
%
% 10-gon boundary
\draw[fill=white] (14, 0) circle[radius=2.5];
\node  at (14.9, 2.6) {$a$};
\node  at (16.2, 1.6) {$b$};
\node  at (16.8, 0) {$c$};
\node  at (16.2, -1.6) {$a$};
\node  at (14.9, -2.6) {$d$};
\node  at (13.1, -2.6) {$c$};
\node  at (11.8, -1.6) {$e$};
\node  at (11.2, 0) {$d$};
\node  at (11.8, 1.6) {$b$};
\node  at (13.1, 2.6) {$e$};
\draw[dotted, line width=0.5] (14.7, 2.1)--(15.9, -1.4); % a
\draw[dotted, line width=0.5] (15.8, 1.3)--(12.2, 1.3); % b
\draw[dotted, line width=0.5] (16.3, 0)--(13.3, -2.1); % c
\draw[dotted, line width=0.5] (14.7, -2.2)--(11.7, 0); % d
\draw[dotted, line width=0.5] (12.2, -1.3)--(13.3, 2.2); % e
\draw[fill=white] (14, 2.5) circle[radius=0.1];
\draw[fill=black] (14, -2.5) circle[radius=0.1];
\draw[fill=white] (16.4, 0.75) circle[radius=0.1];
\draw[fill=black] (16.4, -0.75) circle[radius=0.1];
\draw[fill=white] (11.6, 0.75) circle[radius=0.1];
\draw[fill=black] (11.6, -0.75) circle[radius=0.1];
\draw[fill=black] (15.5, 2) circle[radius=0.1];
\draw[fill=white] (15.5, -2) circle[radius=0.1];
\draw[fill=black] (12.5, 2) circle[radius=0.1];
\draw[fill=white] (12.5, -2) circle[radius=0.1];
\draw[line width=1] (14, 1.2)--(15.1, 0.5)--(14.7, -0.9)--(13.2, -0.9)--(12.9, 0.5)--(14, 1.2); % pentagon
\draw[line width=1] (16.0, 1.5)--(15.1, 0.5)--(16.5, 0);
\draw[line width=1] (16.0, -1.5)--(14.7, -0.9)--(14.8, -2.4);
\draw[line width=1] (13.0, -2.3)--(13.2, -0.9)--(11.9, -1.4);
\draw[line width=1] (11.5, -0.0)--(12.9, 0.5)--(12.0, 1.5);
\draw[line width=1] (13.2, 2.4)--(14, 1.2)--(14.7, 2.4);
\draw[fill=lightgray] (14, 1.2) circle[radius=0.1];
\draw[fill=lightgray] (15.1, 0.5) circle[radius=0.1];
\draw[fill=lightgray] (12.9, 0.5) circle[radius=0.1];
\draw[fill=lightgray] (14.7, -0.9) circle[radius=0.1];
\draw[fill=lightgray] (13.2, -0.9) circle[radius=0.1];
%
%\node at (0, -3.5) {$\Theta_5^{\# 1}$};
\node at (7, -3.5) {$\Theta_5^{\# 2}$};
\node at (14, -3.5) {$\Theta_5^{\# 3}$};
\end{tikzpicture}
\caption{Embeddings of $K_5$ on two $10$-gon representations of the double torus.}
\label{10gonK5}
\end{center}
\end{figure}

%%%%% This section is taken out for further clarification/more details %%%%%

%%%%%%
%%%%%%%%%%%%%
%%%%%%%%%%%%%%%%
\section{Conclusion}
\label{section-conclusion}

%%%%%%%
This paper provides a method for constructing all distinct $2$-cell embeddings of various graphs on the double torus, 
and presents all the $2$-cell embeddings of fundamental non-planar graphs 
$K_{3,3}$ and $K_5$ on the orientable surfaces, as a complement to their well-known toroidal embeddings.
From Euler's formula, the maximum orientable genus of $K_{3,3}$ is two, the maximum orientable 
genus of $K_5$ is three, and the $2$-cell embeddings of $K_{3,3}$ on the double torus and 
$K_5$ on the triple torus have a single face. Therefore, $K_{3,3}$ has orientable genus 
spectrum $\{1,2\}$, and $K_5$ has orientable genus spectrum $\{1,2,3\}$. 

By reconstructing embeddings from graph minors, we obtain the unique $2$-cell embedding of $K_{3,3}$ and 
all distinct $2$-cell embeddings of $K_5$ on the double torus.
The $2$-cell embeddings of $K_5$ on the triple torus are obtained by using an exhaustive computer 
search of rotation systems.
Figure~\ref{K33embedding} shows the unique embedding of $K_{3,3}$ on the double torus, 
Appendix A provides all the $2$-cell embeddings of $K_5$ on the double torus (represented by their rotation systems), 
and Appendix B provides all the rotation systems for different $2$-cell embeddings of $K_5$ on the triple torus.
Thus, $K_{3,3}$ has $3(=2+1)$ and $K_5$ has $50(=6+31+13)$ inequivalent $2$-cell embeddings on orientable surfaces in total. 
In Appendix C, we report some corrections to previous computational results from \cite{GagarinKocayNeilsen} for embeddings of 
small vertex-transitive graphs (and $K_{3,4}, K_{3,5}, K_{3,6}$) on the torus --
these corrections result from a corrected bug in computer code (spotted by Professor William Dickinson, 
Grand Valley State University).

Notice that the description of rotation systems in Appendix A allows a drawing of the embedding on the double torus to be 
constructed from its rotation system.
However, in general, it is a non-trivial task to find a drawing 
of a graph on a polygonal representation of a surface from
its rotation system. An ad-hoc approach has been used to draw two of the rotation systems from Appendix B
for embeddings of $K_5$ on the triple torus shown in Figure~\ref{K5tt}. The reader is encouraged to draw the other $11$ embeddings from Appendix B.
It would be interesting to find a constructive method to obtain all $2$-cell embeddings of $K_5$ on the triple torus.
%%%%%%%
%%%%%%%
\begin{figure}[ht]
\begin{center}
\begin{tikzpicture}[scale=0.75, line width = 0.5]
%12-gon boundary
\draw (3, 0)--(2.598, 1.5)--(1.5, 2.598)
--(0,3)--(-1.5, 2.598)--(-2.598, 1.5)
--(-3,0)--(-2.598, -1.5)--(-1.5, -2.598)
--(0, -3)--(1.5, -2.598)--(2.598, -1.5)--(3,0);
%
%12-gon labels
\node  at (0.83, 2.97) {$a$};
\node  at (2.22, 2.22) {$b$};
\node  at (3.0, 0.8) {$a$};
\node  at (2.97, -0.79) {$b$};
\node  at (2.22, -2.08) {$c$};
\node  at (0.82, -2.98) {$d$};
\node  at (-0.81, -2.97) {$c$};
\node  at (-2.23, -2.23) {$d$};
\node  at (-2.97, -0.79) {$e$};
\node  at (-2.94, 0.82) {$f$};
\node  at (-2.18, 2.18) {$e$};
\node  at (-0.8, 3.05) {$f$};

%%%%%%
%%%%%%
% graph edges
\draw[line width=1] (-0.65, 0.375)--(0.65, 0.375);
%%%%
\draw[line width=1] (0, 1.5)--(0.75, 2.8);
\draw[line width=1] (1.299, -0.75)--(2.8, 0.75);

\draw[line width=1] (1.299, -0.75)--(2.05,-2.05);
\draw[line width=1] (-1.299, -0.75)--(-0.75,-2.8);

\draw[line width=1] (1.299, -0.75)--(1.775,-2.325);
\draw[line width=1] (-0.65, 0.375)--(-0.375,-2.9);

\draw[line width=1] (1.299, -0.75)--(2.8,-0.75);
\draw[line width=1] (0.65, 0.375)--(2.05,2.05);
%%%%
\draw[line width=1] (0, 1.5)--(1.775, 2.325);
\draw[line width=1] (2.7, -1.125)--(2.325, -1.775);
\draw[line width=1] (-1.299, -0.75)--(-1.125, -2.7);

\draw[line width=1] (0, 1.5)--(-0.75, 2.801);
\draw[line width=1] (-0.65, 0.375)--(-2.801, 0.75);

\draw[line width=1] (0, 1.5)--(1.125, 2.7);
\draw[line width=1] (0.65, 0.375)--(2.7, 1.125);
%%%%
\draw[line width=1] (-1.299, -0.75)--(-2.801, -0.75);
\draw[line width=1] (-0.65, 0.375)--(-2.05, 2.05);

\draw[line width=1] (-1.299, -0.75)--(-2.05, -2.05);
\draw[line width=1] (0.65, 0.375)--(0.75, -2.801);

% graph vertices
\draw[fill=white] (0, 1.5) circle[radius=0.1];
\draw[fill=white] (1.299, -0.75) circle[radius=0.1];
\draw[fill=white] (-1.299, -0.75) circle[radius=0.1];
\draw[fill=white] (0.65, 0.375) circle[radius=0.1];
\draw[fill=white] (-0.65, 0.375) circle[radius=0.1];

% vertex labels
\node  at (0, 1.27) {\smallmath 2};
\node  at (1.2, -0.53) {\smallmath 1};
\node  at (-1.2, -0.5) {\smallmath 3};
\node  at (0.55, 0.61) {\smallmath 5};
\node  at (-0.55, 0.61) {\smallmath 4};

%%%%%%%%%%%%%%%%%
%%%% 2nd embedding %%%%%
%%%%%%%%%%%%%%%

%12-gon boundary
\draw (12, 0)--(11.598, 1.5)--(10.5, 2.598)
--(9,3)--(7.5, 2.598)--(6.402, 1.5)
--(6,0)--(6.402, -1.5)--(7.5, -2.598)
--(9, -3)--(10.5, -2.598)--(11.598, -1.5)--(12,0);
%
%12-gon labels
\node  at (9.83, 2.97) {$a$};
\node  at (11.22, 2.22) {$b$};
\node  at (12.0, 0.8) {$a$};
\node  at (11.97, -0.79) {$b$};
\node  at (11.22, -2.08) {$c$};
\node  at (9.82, -2.98) {$d$};
\node  at (8.19, -2.97) {$c$};
\node  at (6.77, -2.23) {$d$};
\node  at (6.03, -0.79) {$e$};
\node  at (6.06, 0.82) {$f$};
\node  at (6.82, 2.18) {$e$};
\node  at (8.2, 3.05) {$f$};

%%%%%%
% graph edges
%% 4-5
\draw[line width=1] (8.35, 0.375)--(9.65, 0.375);
%%%% at 1
\draw[line width=1] (10.299, -0.75)--(11.8, 0.75);
\draw[line width=1] (10.299, -0.75)--(11.05,-2.05);
\draw[line width=1] (10.299, -0.75)--(10.775,-2.325);
\draw[line width=1] (10.299, -0.75)--(11.8,-0.75);
%%%% at 2
\draw[line width=1] (9, 1.5)--(9.75, 2.8);
\draw[line width=1] (9, 1.5)--(6.95, 2.05);
\draw[line width=1] (9, 1.5)--(8.25, 2.801);
\draw[line width=1] (9, 1.5)--(10.125, 2.7);
%%%% at 3
\draw[line width=1] (7.701, -0.75)--(8.25,-2.8);
\draw[line width=1] (7.701, -0.75)--(6.199, -0.75);
\draw[line width=1] (7.701, -0.75)--(6.95, -2.05);
\draw[line width=1] (7.701, -0.75)--(6.1, -0.375);
%%% at 4
\draw[line width=1] (8.35, 0.375)--(6.675, 1.775);
\draw[line width=1] (8.35, 0.375)--(6.199, 0.75);
\draw[line width=1] (8.35, 0.375)--(8.625,-2.9);
%%% at 5
\draw[line width=1] (9.65, 0.375)--(9.75, -2.801);

\draw[line width=1] (9.65, 0.375)--(11.7, 1.125);

\draw[line width=1] (9.65, 0.375)--(11.05,2.05);

% graph vertices
\draw[fill=white] (9, 1.5) circle[radius=0.1];
\draw[fill=white] (10.299, -0.75) circle[radius=0.1];
\draw[fill=white] (7.701, -0.75) circle[radius=0.1];
\draw[fill=white] (9.65, 0.375) circle[radius=0.1];
\draw[fill=white] (8.35, 0.375) circle[radius=0.1];

%% vertex labels - let's keep it unlabelled
%
\end{tikzpicture}
\caption{Orientable (K5\#a) and non-orientable (K5\#l) $2$-cell embeddings of $K_5$ on the triple torus.}
\label{K5tt}
\end{center}
\end{figure}
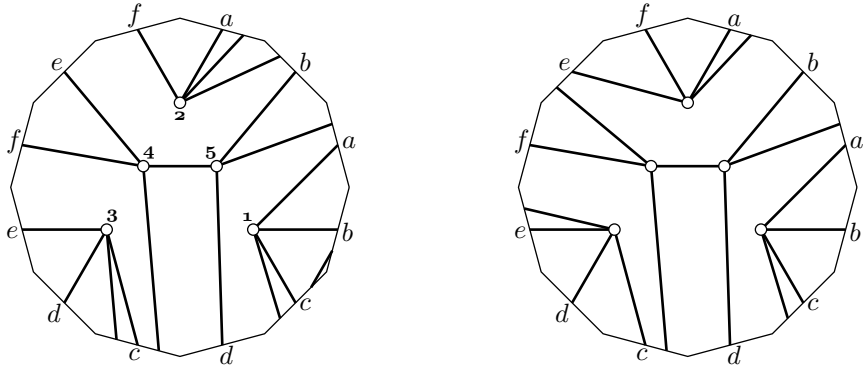
%%%%%%%
%%%%%

This paper also provides a number of different representations of the double torus and several examples 
of drawn $2$-cell embeddings, e.g. a symmetric non-orientable embedding of $K_5$ (Figures~\ref{K5fig43} and \ref{10gonK5}), an orientable embedding of $K_5$ with the trivial automorphism group (Figure~\ref{K5fig01}) on the double torus, two embeddings of $K_5$ on the triple torus (having one face, Figure~\ref{K5tt}). 
Different representations of the double torus can be used, for example, to have symmetric drawings of the embeddings (Figures~\ref{10gonK5}) or to solve geometric packing problems like in Brandt et al.~\cite{Dickinson19}.
For future research, it would be interesting to describe more precisely and obtain all polygonal representations of the double and triple tori.

Since the multi-graph $\Theta_5$ and its embeddings on the double torus play a crucial role in 
obtaining all the embeddings of $K_{3,3}$ and $K_5$ on this surface 
(similarly to the unique embedding of $\Theta_3$ on the torus), we ask the following:

\medbreak
\noindent
{\bf Question.}  How many distinct 
embeddings does $\Theta_{2g+1}$ have on 
the orientable surface of genus $g$, and how can they be obtained? 
\medbreak
\noindent By Euler's formula, these embeddings will have a single face.
When $g=1$, there is one embedding. When $g=2$, there
are three embeddings.

%%%%%%%%
%\bigbreak

%%%%%%%%%%%%%%

\newpage
\section{Appendices}

Rotation systems for all inequivalent embeddings of $K_5$ on the double and triple tori are given in 
Appendices A and B, and a corrected table for the embeddings of small vertex-transitive 
graphs on the torus from \cite{GagarinKocayNeilsen} is shown in Appendix C.
The rotation for vertex $k\in\{1,2,3,4,5\}$ is indicated by $-k$, followed by
the cyclic list of adjacent vertices.  
The orientable embeddings are indicated by ``or'', the non-orientable by ``non''.

\medbreak
\textbf{A}: Rotation systems for the $31$ inequivalent $2$-cell embeddings of $K_5$ on the double torus are listed below.
The numbers in square brackets indicate the edge number, followed by a list of which sides of the fundamental octagon region are cut by the edge, so that a drawing of the corresponding embedding can be reconstructed from the rotation system.

\smallbreak

{%\footnotesize
 \scriptsize
 
\noindent K5\#01 (or)\\
-1 5 [1 0 0 0 0] 4 [2 0 0 0 0] 2 [3 -3 0 0 0] 3 [4 -3 2 0 0]\\
-2 3 [5 2 0 0 0] 1 [3 3 0 0 0] 5 [6 4 0 0 0] 4 [7 0 0 0 0]\\
-3 4 [8 -1 0 0 0] 1 [4 -2 3 0 0] 2 [5 -2 0 0 0] 5 [9 0 0 0 0]\\
-4 2 [7 0 0 0 0] 1 [2 0 0 0 0] 5 [10 0 0 0 0] 3 [8 1 0 0 0]\\
-5 1 [1 0 0 0 0] 2 [6 -4 0 0 0] 3 [9 0 0 0 0] 4 [10 0 0 0 0]\\

\noindent K5\#02  (or)\\
-1 5 [1 0 0 0 0] 4 [2 0 0 0 0] 2 [3 -3 0 0 0] 3 [4 -3 2 -1 0]\\
-2 3 [5 2 0 0 0] 1 [3 3 0 0 0] 5 [6 4 0 0 0] 4 [7 0 0 0 0]\\
-3 5 [8 0 0 0 0] 4 [9 -1 0 0 0] 2 [5 -2 0 0 0] 1 [4 1 -2 3 0]\\
-4 2 [7 0 0 0 0] 1 [2 0 0 0 0] 5 [10 0 0 0 0] 3 [9 1 0 0 0]\\
-5 1 [1 0 0 0 0] 2 [6 -4 0 0 0] 3 [8 0 0 0 0] 4 [10 0 0 0 0]\\

\noindent K5\#03  (or)\\
-1 5 [1 0 0 0 0] 4 [2 0 0 0 0] 2 [3 -3 0 0 0] 3 [4 -3 2 -1 0]\\
-2 3 [5 2 0 0 0] 1 [3 3 0 0 0] 5 [6 4 0 0 0] 4 [7 0 0 0 0]\\
-3 5 [8 0 0 0 0] 4 [9 -1 0 0 0] 2 [5 -2 0 0 0] 1 [4 1 -2 3 0]\\
-4 2 [7 0 0 0 0] 1 [2 0 0 0 0] 3 [9 1 0 0 0] 5 [10 1 0 0 0]\\
-5 1 [1 0 0 0 0] 2 [6 -4 0 0 0] 4 [10 -1 0 0 0] 3 [8 0 0 0 0]\\

\noindent K5\#04  (or)\\
-1 5 [1 0 0 0 0] 4 [2 0 0 0 0] 2 [3 -3 0 0 0] 3 [4 -3 2 -1 0]\\
-2 3 [5 2 0 0 0] 1 [3 3 0 0 0] 5 [6 4 0 0 0] 4 [7 0 0 0 0]\\
-3 5 [8 0 0 0 0] 4 [9 -1 0 0 0] 2 [5 -2 0 0 0] 1 [4 1 -2 3 0]\\
-4 5 [10 -3 4 0 0] 1 [2 0 0 0 0] 3 [9 1 0 0 0] 2 [7 0 0 0 0]\\
-5 1 [1 0 0 0 0] 2 [6 -4 0 0 0] 4 [10 -4 3 0 0] 3 [8 0 0 0 0]\\

\noindent K5\#05  (or)\\
-1 5 [1 0 0 0 0] 4 [2 0 0 0 0] 2 [3 -3 0 0 0] 3 [4 -4 -3 4 0]\\
-2 3 [5 2 0 0 0] 1 [3 3 0 0 0] 5 [6 4 0 0 0] 4 [7 0 0 0 0]\\
-3 1 [4 -4 3 4 0] 4 [8 -1 0 0 0] 2 [5 -2 0 0 0] 5 [9 0 0 0 0]\\
-4 2 [7 0 0 0 0] 1 [2 0 0 0 0] 5 [10 0 0 0 0] 3 [8 1 0 0 0]\\
-5 1 [1 0 0 0 0] 2 [6 -4 0 0 0] 3 [9 0 0 0 0] 4 [10 0 0 0 0]\\

\noindent K5\#06  (or)\\
-1 5 [1 0 0 0 0] 4 [2 0 0 0 0] 2 [3 -3 0 0 0] 3 [4 -4 -3 4 0]\\
-2 3 [5 2 0 0 0] 1 [3 3 0 0 0] 5 [6 4 0 0 0] 4 [7 0 0 0 0]\\
-3 1 [4 -4 3 4 0] 4 [8 -1 0 0 0] 2 [5 -2 0 0 0] 5 [9 0 0 0 0]\\
-4 2 [7 0 0 0 0] 1 [2 0 0 0 0] 3 [8 1 0 0 0] 5 [10 2 0 0 0]\\
-5 1 [1 0 0 0 0] 2 [6 -4 0 0 0] 3 [9 0 0 0 0] 4 [10 -2 0 0 0]\\

\noindent K5\#07  (or)\\
-1 3 [1 0 0 0 0] 4 [2 0 0 0 0] 5 [3 -3 0 0 0] 2 [4 -4 0 0 0]\\
-2 3 [5 1 0 0 0] 5 [6 0 0 0 0] 1 [4 4 0 0 0] 4 [7 0 0 0 0]\\
-3 2 [5 -1 0 0 0] 4 [8 -1 0 0 0] 5 [9 -2 0 0 0] 1 [1 0 0 0 0]\\
-4 2 [7 0 0 0 0] 1 [2 0 0 0 0] 5 [10 1 -2 0 0] 3 [8 1 0 0 0]\\
-5 2 [6 0 0 0 0] 3 [9 2 0 0 0] 4 [10 2 -1 0 0] 1 [3 3 0 0 0]\\

\noindent K5\#08  (or)\\
-1 3 [1 0 0 0 0] 4 [2 0 0 0 0] 5 [3 -3 0 0 0] 2 [4 -4 0 0 0]\\
-2 3 [5 1 0 0 0] 5 [6 0 0 0 0] 1 [4 4 0 0 0] 4 [7 0 0 0 0]\\
-3 2 [5 -1 0 0 0] 4 [8 -1 0 0 0] 5 [9 -2 0 0 0] 1 [1 0 0 0 0]\\
-4 5 [10 -3 4 -1 0] 1 [2 0 0 0 0] 3 [8 1 0 0 0] 2 [7 0 0 0 0]\\
-5 4 [10 1 -4 3 0] 3 [9 2 0 0 0] 1 [3 3 0 0 0] 2 [6 0 0 0 0]\\

\noindent K5\#09  (or)\\
-1 3 [1 0 0 0 0] 4 [2 0 0 0 0] 5 [3 -3 0 0 0] 2 [4 -4 0 0 0]\\
-2 3 [5 2 1 0 0] 5 [6 0 0 0 0] 1 [4 4 0 0 0] 4 [7 0 0 0 0]\\
-3 4 [8 -1 0 0 0] 2 [5 -1 -2 0 0] 5 [9 -2 0 0 0] 1 [1 0 0 0 0]\\
-4 2 [7 0 0 0 0] 1 [2 0 0 0 0] 5 [10 -2 0 0 0] 3 [8 1 0 0 0]\\
-5 4 [10 2 0 0 0] 3 [9 2 0 0 0] 1 [3 3 0 0 0] 2 [6 0 0 0 0]\\

\noindent K5\#10  (or)\\
-1 3 [1 0 0 0 0] 4 [2 0 0 0 0] 5 [3 -3 0 0 0] 2 [4 -4 0 0 0]\\
-2 3 [5 2 1 0 0] 5 [6 0 0 0 0] 1 [4 4 0 0 0] 4 [7 0 0 0 0]\\
-3 4 [8 -1 0 0 0] 2 [5 -1 -2 0 0] 5 [9 -2 0 0 0] 1 [1 0 0 0 0]\\
-4 2 [7 0 0 0 0] 1 [2 0 0 0 0] 3 [8 1 0 0 0] 5 [10 2 1 -2 0]\\
-5 2 [6 0 0 0 0] 3 [9 2 0 0 0] 4 [10 2 -1 -2 0] 1 [3 3 0 0 0]\\

\noindent K5\#11  (or)\\
-1 3 [1 0 0 0 0] 4 [2 0 0 0 0] 5 [3 -3 0 0 0] 2 [4 -4 0 0 0]\\
-2 3 [5 2 0 0 0] 5 [6 0 0 0 0] 1 [4 4 0 0 0] 4 [7 0 0 0 0]\\
-3 1 [1 0 0 0 0] 4 [8 -1 0 0 0] 5 [9 -2 0 0 0] 2 [5 -2 0 0 0]\\
-4 5 [10 4 -3 0 0] 1 [2 0 0 0 0] 3 [8 1 0 0 0] 2 [7 0 0 0 0]\\
-5 2 [6 0 0 0 0] 3 [9 2 0 0 0] 4 [10 3 -4 0 0] 1 [3 3 0 0 0]\\

\noindent K5\#12  (or)\\
-1 3 [1 0 0 0 0] 4 [2 0 0 0 0] 5 [3 -3 0 0 0] 2 [4 -4 0 0 0]\\
-2 5 [5 0 0 0 0] 3 [6 4 0 0 0] 1 [4 4 0 0 0] 4 [7 0 0 0 0]\\
-3 2 [6 -4 0 0 0] 4 [8 -1 0 0 0] 5 [9 -2 0 0 0] 1 [1 0 0 0 0]\\
-4 2 [7 0 0 0 0] 1 [2 0 0 0 0] 5 [10 -2 0 0 0] 3 [8 1 0 0 0]\\
-5 4 [10 2 0 0 0] 3 [9 2 0 0 0] 1 [3 3 0 0 0] 2 [5 0 0 0 0]\\

\noindent K5\#13  (or)\\
-1 3 [1 0 0 0 0] 4 [2 0 0 0 0] 5 [3 -3 0 0 0] 2 [4 -4 0 0 0]\\
-2 4 [5 0 0 0 0] 5 [6 0 0 0 0] 1 [4 4 0 0 0] 3 [7 4 -3 2 0]\\
-3 1 [1 0 0 0 0] 4 [8 -1 0 0 0] 2 [7 -2 3 -4 0] 5 [9 -2 0 0 0]\\
-4 2 [5 0 0 0 0] 1 [2 0 0 0 0] 3 [8 1 0 0 0] 5 [10 0 0 0 0]\\
-5 4 [10 0 0 0 0] 3 [9 2 0 0 0] 1 [3 3 0 0 0] 2 [6 0 0 0 0]\\

\noindent K5\#14  (or)\\
-1 5 [1 0 0 0 0] 2 [2 -3 0 0 0] 3 [3 0 0 0 0] 4 [4 -2 0 0 0]\\
-2 3 [5 4 0 0 0] 5 [6 0 0 0 0] 1 [2 3 0 0 0] 4 [7 4 -1 0 0]\\
-3 2 [5 -4 0 0 0] 4 [8 -1 0 0 0] 5 [9 -1 0 0 0] 1 [3 0 0 0 0]\\
-4 2 [7 1 -4 0 0] 1 [4 2 0 0 0] 5 [10 0 0 0 0] 3 [8 1 0 0 0]\\
-5 4 [10 0 0 0 0] 2 [6 0 0 0 0] 1 [1 0 0 0 0] 3 [9 1 0 0 0]\\

\noindent K5\#15 (non)\\
-1 5 [1 0 0 0 0] 4 [2 0 0 0 0] 2 [3 -3 0 0 0] 3 [4 -3 2 -1 0]\\
-2 3 [5 2 0 0 0] 1 [3 3 0 0 0] 5 [6 4 0 0 0] 4 [7 0 0 0 0]\\
-3 5 [8 0 0 0 0] 4 [9 -1 0 0 0] 2 [5 -2 0 0 0] 1 [4 1 -2 3 0]\\
-4 2 [7 0 0 0 0] 1 [2 0 0 0 0] 3 [9 1 0 0 0] 5 [10 1 -4 3 4]\\
-5 4 [10 -4 -3 4 -1] 2 [6 -4 0 0 0] 3 [8 0 0 0 0] 1 [1 0 0 0 0]\\

\noindent K5\#16 (non)\\
-1 5 [1 0 0 0 0] 4 [2 0 0 0 0] 2 [3 -3 0 0 0] 3 [4 -3 2 -1 0]\\
-2 3 [5 2 0 0 0] 1 [3 3 0 0 0] 5 [6 4 0 0 0] 4 [7 0 0 0 0]\\
-3 5 [8 0 0 0 0] 4 [9 -1 0 0 0] 2 [5 -2 0 0 0] 1 [4 1 -2 3 0]\\
-4 5 [10 4 0 0 0] 1 [2 0 0 0 0] 3 [9 1 0 0 0] 2 [7 0 0 0 0]\\
-5 4 [10 -4 0 0 0] 2 [6 -4 0 0 0] 3 [8 0 0 0 0] 1 [1 0 0 0 0]\\

\noindent K5\#17 (non)\\
-1 3 [1 0 0 0 0] 4 [2 0 0 0 0] 5 [3 -3 0 0 0] 2 [4 -4 0 0 0]\\
-2 3 [5 1 0 0 0] 5 [6 0 0 0 0] 1 [4 4 0 0 0] 4 [7 0 0 0 0]\\
-3 2 [5 -1 0 0 0] 4 [8 -1 0 0 0] 5 [9 -2 0 0 0] 1 [1 0 0 0 0]\\
-4 5 [10 4 -3 0 0] 1 [2 0 0 0 0] 3 [8 1 0 0 0] 2 [7 0 0 0 0]\\
-5 2 [6 0 0 0 0] 3 [9 2 0 0 0] 4 [10 3 -4 0 0] 1 [3 3 0 0 0]\\

\noindent K5\#18 (non)\\
-1 3 [1 0 0 0 0] 4 [2 0 0 0 0] 5 [3 -3 0 0 0] 2 [4 -4 0 0 0]\\
-2 3 [5 1 0 0 0] 5 [6 0 0 0 0] 1 [4 4 0 0 0] 4 [7 0 0 0 0]\\
-3 2 [5 -1 0 0 0] 4 [8 -1 0 0 0] 5 [9 -2 0 0 0] 1 [1 0 0 0 0]\\
-4 5 [10 -3 0 0 0] 1 [2 0 0 0 0] 3 [8 1 0 0 0] 2 [7 0 0 0 0]\\
-5 2 [6 0 0 0 0] 3 [9 2 0 0 0] 1 [3 3 0 0 0] 4 [10 3 0 0 0]\\

\noindent K5\#19 (non)\\
-1 3 [1 0 0 0 0] 4 [2 0 0 0 0] 5 [3 -3 0 0 0] 2 [4 -4 0 0 0]\\
-2 3 [5 2 1 0 0] 5 [6 0 0 0 0] 1 [4 4 0 0 0] 4 [7 0 0 0 0]\\
-3 4 [8 -1 0 0 0] 2 [5 -1 -2 0 0] 5 [9 -2 0 0 0] 1 [1 0 0 0 0]\\
-4 2 [7 0 0 0 0] 1 [2 0 0 0 0] 3 [8 1 0 0 0] 5 [10 1 -4 0 0]\\
-5 2 [6 0 0 0 0] 3 [9 2 0 0 0] 1 [3 3 0 0 0] 4 [10 4 -1 0 0]\\

\noindent K5\#20 (non)\\
-1 3 [1 0 0 0 0] 4 [2 0 0 0 0] 5 [3 -3 0 0 0] 2 [4 -4 0 0 0]\\
-2 3 [5 2 1 0 0] 5 [6 0 0 0 0] 1 [4 4 0 0 0] 4 [7 0 0 0 0]\\
-3 4 [8 -1 0 0 0] 2 [5 -1 -2 0 0] 5 [9 -2 0 0 0] 1 [1 0 0 0 0]\\
-4 5 [10 4 -3 0 0] 1 [2 0 0 0 0] 3 [8 1 0 0 0] 2 [7 0 0 0 0]\\
-5 2 [6 0 0 0 0] 3 [9 2 0 0 0] 4 [10 3 -4 0 0] 1 [3 3 0 0 0]\\

\noindent K5\#21 (non)\\
-1 3 [1 0 0 0 0] 4 [2 0 0 0 0] 5 [3 -3 0 0 0] 2 [4 -4 0 0 0]\\
-2 3 [5 2 0 0 0] 5 [6 0 0 0 0] 1 [4 4 0 0 0] 4 [7 0 0 0 0]\\
-3 1 [1 0 0 0 0] 4 [8 -1 0 0 0] 5 [9 -2 0 0 0] 2 [5 -2 0 0 0]\\
-4 2 [7 0 0 0 0] 1 [2 0 0 0 0] 5 [10 1 -2 0 0] 3 [8 1 0 0 0]\\
-5 2 [6 0 0 0 0] 3 [9 2 0 0 0] 4 [10 2 -1 0 0] 1 [3 3 0 0 0]\\

\noindent K5\#22 (non)\\
-1 3 [1 0 0 0 0] 4 [2 0 0 0 0] 5 [3 -3 0 0 0] 2 [4 -4 0 0 0]\\
-2 3 [5 2 0 0 0] 5 [6 0 0 0 0] 1 [4 4 0 0 0] 4 [7 0 0 0 0]\\
-3 1 [1 0 0 0 0] 4 [8 -1 0 0 0] 5 [9 -2 0 0 0] 2 [5 -2 0 0 0]\\
-4 2 [7 0 0 0 0] 1 [2 0 0 0 0] 3 [8 1 0 0 0] 5 [10 1 -4 0 0]\\
-5 2 [6 0 0 0 0] 3 [9 2 0 0 0] 1 [3 3 0 0 0] 4 [10 4 -1 0 0]\\

\noindent K5\#23 (non)\\
-1 3 [1 0 0 0 0] 4 [2 0 0 0 0] 5 [3 -3 0 0 0] 2 [4 -4 0 0 0]\\
-2 4 [5 0 0 0 0] 5 [6 0 0 0 0] 1 [4 4 0 0 0] 3 [7 4 -3 2 0]\\
-3 1 [1 0 0 0 0] 4 [8 -1 0 0 0] 2 [7 -2 3 -4 0] 5 [9 -2 0 0 0]\\
-4 2 [5 0 0 0 0] 1 [2 0 0 0 0] 3 [8 1 0 0 0] 5 [10 1 -4 0 0]\\
-5 2 [6 0 0 0 0] 3 [9 2 0 0 0] 1 [3 3 0 0 0] 4 [10 4 -1 0 0]\\

\noindent K5\#24 (non)\\
-1 4 [1 -4 0 0 0] 2 [2 -1 0 0 0] 5 [3 0 0 0 0] 3 [4 0 0 0 0]\\
-2 5 [5 2 0 0 0] 3 [6 3 0 0 0] 4 [7 0 0 0 0] 1 [2 1 0 0 0]\\
-3 5 [8 0 0 0 0] 4 [9 0 0 0 0] 2 [6 -3 0 0 0] 1 [4 0 0 0 0]\\
-4 2 [7 0 0 0 0] 1 [1 4 0 0 0] 5 [10 4 -3 2 0] 3 [9 0 0 0 0]\\
-5 4 [10 -2 3 -4 0] 2 [5 -2 0 0 0] 3 [8 0 0 0 0] 1 [3 0 0 0 0]\\

\noindent K5\#25 (non)\\
-1 3 [1 -1 2 0 0] 2 [2 -1 0 0 0] 5 [3 0 0 0 0] 4 [4 -4 0 0 0]\\
-2 5 [5 2 0 0 0] 3 [6 3 0 0 0] 4 [7 0 0 0 0] 1 [2 1 0 0 0]\\
-3 1 [1 -2 1 0 0] 4 [8 0 0 0 0] 2 [6 -3 0 0 0] 5 [9 0 0 0 0]\\
-4 2 [7 0 0 0 0] 1 [4 4 0 0 0] 3 [8 0 0 0 0] 5 [10 1 -2 3 0]\\
-5 1 [3 0 0 0 0] 2 [5 -2 0 0 0] 3 [9 0 0 0 0] 4 [10 -3 2 -1 0]\\

\noindent K5\#26 (non)\\
-1 2 [1 -1 0 0 0] 3 [2 -2 3 0 0] 5 [3 0 0 0 0] 4 [4 -4 0 0 0]\\
-2 5 [5 2 0 0 0] 3 [6 3 0 0 0] 4 [7 0 0 0 0] 1 [1 1 0 0 0]\\
-3 5 [8 0 0 0 0] 4 [9 0 0 0 0] 2 [6 -3 0 0 0] 1 [2 -3 2 0 0]\\
-4 2 [7 0 0 0 0] 1 [4 4 0 0 0] 5 [10 4 0 0 0] 3 [9 0 0 0 0]\\
-5 1 [3 0 0 0 0] 2 [5 -2 0 0 0] 3 [8 0 0 0 0] 4 [10 -4 0 0 0]\\

\noindent K5\#27 (non)\\
-1 2 [1 -1 0 0 0] 3 [2 -2 3 0 0] 5 [3 0 0 0 0] 4 [4 -4 0 0 0]\\
-2 5 [5 2 0 0 0] 3 [6 3 0 0 0] 4 [7 0 0 0 0] 1 [1 1 0 0 0]\\
-3 5 [8 0 0 0 0] 4 [9 0 0 0 0] 2 [6 -3 0 0 0] 1 [2 -3 2 0 0]\\
-4 5 [10 3 4 0 0] 1 [4 4 0 0 0] 3 [9 0 0 0 0] 2 [7 0 0 0 0]\\
-5 1 [3 0 0 0 0] 2 [5 -2 0 0 0] 3 [8 0 0 0 0] 4 [10 -4 -3 0 0]\\

\noindent K5\#28 (non)\\
-1 5 [1 0 0 0 0] 2 [2 -3 0 0 0] 3 [3 0 0 0 0] 4 [4 -2 0 0 0]\\
-2 3 [5 4 0 0 0] 5 [6 0 0 0 0] 1 [2 3 0 0 0] 4 [7 4 -1 0 0]\\
-3 2 [5 -4 0 0 0] 4 [8 -1 0 0 0] 5 [9 -2 3 -4 0] 1 [3 0 0 0 0]\\
-4 2 [7 1 -4 0 0] 1 [4 2 0 0 0] 5 [10 0 0 0 0] 3 [8 1 0 0 0]\\
-5 4 [10 0 0 0 0] 2 [6 0 0 0 0] 3 [9 4 -3 2 0] 1 [1 0 0 0 0]\\

\noindent K5\#29 (non)\\
-1 2 [1 -1 0 0 0] 4 [2 -2 0 0 0] 5 [3 0 0 0 0] 3 [4 -4 3 0 0]\\
-2 4 [5 0 0 0 0] 5 [6 4 0 0 0] 3 [7 0 0 0 0] 1 [1 1 0 0 0]\\
-3 2 [7 0 0 0 0] 1 [4 -3 4 0 0] 4 [8 -3 0 0 0] 5 [9 0 0 0 0]\\
-4 2 [5 0 0 0 0] 1 [2 2 0 0 0] 5 [10 2 -1 0 0] 3 [8 3 0 0 0]\\
-5 4 [10 1 -2 0 0] 3 [9 0 0 0 0] 2 [6 -4 0 0 0] 1 [3 0 0 0 0]\\

\noindent K5\#30 (non)\\
-1 2 [1 -1 0 0 0] 4 [2 -2 0 0 0] 5 [3 0 0 0 0] 3 [4 -1 2 0 0]\\
-2 4 [5 0 0 0 0] 5 [6 4 0 0 0] 3 [7 0 0 0 0] 1 [1 1 0 0 0]\\
-3 1 [4 -2 1 0 0] 2 [7 0 0 0 0] 4 [8 -3 0 0 0] 5 [9 0 0 0 0]\\
-4 2 [5 0 0 0 0] 1 [2 2 0 0 0] 5 [10 2 -1 -2 1] 3 [8 3 0 0 0]\\
-5 1 [3 0 0 0 0] 3 [9 0 0 0 0] 2 [6 -4 0 0 0] 4 [10 -1 2 1 -2]\\

\noindent K5\#31 (non)\\
-1 3 [1 -1 0 0 0] 4 [2 -2 0 0 0] 5 [3 0 0 0 0] 2 [4 -1 0 0 0]\\
-2 4 [5 0 0 0 0] 5 [6 4 0 0 0] 3 [7 0 0 0 0] 1 [4 1 0 0 0]\\
-3 1 [1 1 0 0 0] 2 [7 0 0 0 0] 4 [8 -3 0 0 0] 5 [9 0 0 0 0]\\
-4 2 [5 0 0 0 0] 1 [2 2 0 0 0] 5 [10 2 -1 -2 1] 3 [8 3 0 0 0]\\
-5 1 [3 0 0 0 0] 3 [9 0 0 0 0] 2 [6 -4 0 0 0] 4 [10 -1 2 1 -2]\\

}

%%%%
\bigskip
\noindent\textbf{B}: 
Rotation systems for the $13$ inequivalent $2$-cell embeddings of $K_5$ on the triple torus, $11$ orientable and $2$ non-orientable. Each embedding has exactly one face, with each vertex appearing four times and each edge appearing twice on its boundary. These were found by constructing all possible rotation systems 
for $K_5$ for any orientable genus, then selecting those with exactly one face, so that it is a rotation
system for the triple torus.   This method does not construct drawings of the embeddings.
%\vspace{-3mm}
%%%%%%%%%%

{\scriptsize

\begin{multicols}{5}

\noindent K5\#a (or)\\
% #1
-1 2 5 3 4 \\
-2  1 5 3 4 \\
-3  1 2 5 4 \\
-4  1 2 3 5 \\
-5 1 2 3 4 \\

\noindent K5\#b (or)\\
% #10
-1 2 4 3 5 \\
-2  1 3 4 5 \\
-3  1 2 5 4 \\
-4  1 2 3 5 \\
-5 1 2 3 4 \\

\noindent K5\#c (or)\\
% #101
-1 2 4 3 5 \\
-2  1 3 4 5 \\
-3  1 5 2 4 \\
-4  1 2 5 3 \\
-5 1 2 3 4 \\

%%%
\columnbreak

\noindent K5\#d (or)\\
% #102
-1 2 4 5 3 \\
-2  1 3 4 5 \\
-3  1 5 2 4 \\
-4  1 2 5 3 \\
-5 1 2 3 4 \\

\noindent K5\#e (or)\\
% #103
-1 2 5 3 4 \\
-2  1 4 5 3 \\
-3  1 5 2 4 \\
-4  1 2 5 3 \\
-5 1 2 3 4 \\

\noindent K5\#f (or)\\
% #105
-1 2 5 4 3 \\
-2  1 4 5 3 \\
-3  1 5 2 4 \\
-4  1 2 5 3 \\
-5 1 2 3 4 \\

%%%
\columnbreak

\noindent K5\#g (or)\\
% #107
-1 2 5 4 3 \\
-2  1 5 3 4 \\
-3  1 5 4 2 \\
-4  1 2 5 3 \\
-5 1 2 3 4 \\

\noindent K5\#h (or)\\
% #114
-1 2 3 4 5 \\
-2  1 3 4 5 \\
-3  1 5 4 2 \\
-4  1 2 5 3 \\
-5 1 2 3 4 \\

\noindent K5\#i (or)\\
% #115
-1 2 4 3 5 \\
-2  1 3 4 5 \\
-3  1 5 4 2 \\
-4  1 2 5 3 \\
-5 1 2 3 4 \\

%%%%%
\columnbreak

\noindent K5\#j (or)\\
% #139
-1 2 5 4 3 \\
-2  1 5 3 4 \\
-3  1 2 5 4 \\
-4  1 3 2 5 \\
-5 1 2 3 4 \\

\noindent K5\#k (or)\\
% #172
-1 2 3 4 5 \\
-2  1 5 4 3 \\
-3  1 5 4 2 \\
-4  1 3 2 5 \\
-5 1 2 3 4 \\

%%%%%%
%\columnbreak

\noindent K5\#l (non)\\
% #14
-1 2 3 4 5 \\
-2  1 4 3 5 \\
-3  1 2 5 4 \\
-4  1 2 3 5 \\
-5 1 2 3 4 \\

%%%
\columnbreak

\noindent K5\#m (non)\\
% #199
-1 2 3 4 5 \\
-2  1 5 4 3 \\
-3  1 4 5 2 \\
-4  1 3 2 5 \\
-5 1 2 3 4 \\

\end{multicols}

}

%%%%%%%%
\bigskip
\noindent\textbf{C}: Erratum to \cite{GagarinKocayNeilsen}.
Tables of torus embeddings of many small graphs appear in~\cite{GagarinKocayNeilsen}.
The software that generated the original embeddings missed several embeddings. The condition 
that controls the loop was incorrect, causing the loop to sometimes stop too soon. The 
corrected numbers of embeddings are shown in the following table. The graphs with corrected 
entries are marked with [*] in the last column.
These are taken from \emph{http://www.combinatorialmath.ca/G\&G/}.

\noindent (And in the toroidal embedding $K_5\#1$ of Figure~7 
and Figure~1 of~\cite{GagarinKocayNeilsen}, the embedding
is incorrectly marked as non-orientable.)

%%%%
\begin{table}[htbp]
  \centering
  %\small
  %\footnotesize
  \scriptsize
  \begin{tabular}{@{} lllllllll @{}}
{\bf graph}&$n$&$\varepsilon$&$f$&{\bf \#emb.}&{\bf or.}&{\bf non.}&{\bf groups}\\
$K_4$&4&6&2&2&0&2&[24] $4^1,3^1$\\
$K_5$&5&10&5&6&3&3&[120] $20^1,4^1,2^3,1^1$ [*]\\
$K_{3,3}$&6&9&3&2&0&2&[72] $18^1,2^1$\\
3-Prism&6&9&3&5&0&5&[12] $6^1,2^2,1^2$\\
Octahedron&6&12&6&17&4&13&[48] $12^1,6^1,4^3,3^1,2^6,1^5$\\
$K_6$&6&15&9&4&2&2&[720] $6^2,2^1,1^1$\\
$K_{3,4}$&7&12&5&3&0&3&[144] $4^1,3^1,2^1$\\
${\tiny\sim} C_7=C_7(2)$&7&14&7&28&23&5&[14] $14^1,2^{14},1^{13}$\\
$K_7$ [$\Delta$]&7&21&14&1&1&0&[5040] $42^1$\\
$K_{3,5}$&8&15&7&1&0&1&[720] $3^1$\\
Cube$=C_4\times K_2$&8&12&4&5&0&5&[48] $24^1,8^2,3^1,2^1$\\
$C_8^+$&8&12&4&5&1&4&[16] $2^4,1^1$[*]\\
$K_{4,4}$&8&16&8&2&0&2&[1152] $32^1,16^1$\\
${\tiny\sim}C_8+=C_8(2)$&8&16&8&37&20&17&[16] $16^1,4^4,2^{13},1^{19}$ [*]\\
${\tiny\sim}$Cube&8&16&8&8&4&4&[48] $4^2,2^5,1^1$\\
${\tiny\sim}C_8$&8&20&12&10&10&0&[16] $2^7,1^3$ [*]\\
${\tiny\sim}(2C_4)$&8&20&12&6&2&4&[128] $8^2,4^2,2^2$ [*]\\
${\tiny\sim}(4K_2)$ [$\Delta$]&8&24&16&1&0&1&[384] $16^1$\\
$K_{3,6}$&9&18&9&1&0&1&[4320] $18^1$\\
$C_9(2)$&9&18&9&37&34&3&[18] $18^1,6^1,2^{19},1^{16}$ [*]\\
$C_9(3)$&9&18&9&6&4&2&[18] $18^1,3^1, 2^3,1^1$ [*]\\
$K_3\times K_3$=Paley&9&18&9&7&3&4&[72] $36^1,18^1,4^1,2^3,1^{1}$\\
${\tiny\sim}(3K_3)$ [$\Delta$]&9&27&18&1&0&1&[1296] $54^1$\\
${\tiny\sim}C_9$ [$\Delta$]&9&27&18&1&1&0&[18] $18^1$\\
Petersen&10&15&5&1&0&1&[120] $3^1$\\
$C_{10}^+$&10&15&5&6&1&5&[20] $10^1, 2^4, 1^1$\\
$C_5\times K_2$=5-Prism&10&15&5&5&0&5&[20] $2^3, 1^2$\\
$C_{10}(2)$&10&20&10&60&42&18&[20] $20^1, 4^3, 2^{23}, 1^{33}$\\
$C_{10}(4)$&10&20&10&1&1&0&[320] $20^1$\\
${\tiny\sim}(K_5\times K_2)$&10&20&10&1&1&0&[240] $40^1$\\
${\tiny\sim}C_{10}(2)$&10&25&15&1&0&1&[20] $10^1$\\
${\tiny\sim}C_{10}(4)$&10&25&15&4&4&0&[320] $10^1,2^3$\\
${\tiny\sim}(C_5\times K_2)$ [$\Delta$]&10&30&20&1&1&0&[20] $20^1$\\
$C_{11}(2)$&11&22&11&77&74&3&[22] $22^1, 2^{36}, 1^{40}$ [*]\\
$C_{11}(3)$&11&22&11&1&1&0&[22] $22^1$\\
${\tiny\sim}C_{11}(3)$ [$\Delta$]&11&33&22&1&1&0&[22] $22^1$\\
$C_{12}(5^+)$&12&18&6&3&1&2&[48] $12^1, 6^1, 2^1$[*]\\
$C_6\times K_2$=6-Prism&12&18&6&9&0&9&[24] $12^1,4^2,2^4,1^2$\\
$C_{12}^+$&12&18&6&7&1&6&[24] $6^1,2^4,1^{2}$\\
trunc($K_4$)&12&18&6&9&0&9&[24] $4^1,3^1,2^2,1^5$\\
$C_{12}(3^+,6)$&12&24&12&1&0&1&[48] $24^1$\\
$C_{12}(2)$&12&24&12&138&110&28&[24] $24^1, 8^1, 6^2, 4^4, 2^{45}, 1^{85}$ [*]\\
$C_{12}(3)$&12&24&12&1&1&0&[24] $24^1$\\
$C_{12}(4)$&12&24&12&2&1&1&[24] $24^1, 4^1$\\
$C_{12}(5)$&12&24&12&2&0&2&[768] $24^2$ [*]\\
$C_{12}(5^+,6)$&12&24&12&10&10&0&[48] $6^2, 2^6,1^2$\\
L(Cube)&12&24&12&16&1&15&[48] $24^1, 8^2, 4^1, 3^1, 2^3, 1^8$ [*]\\
$C_4\times C_3$&12&24&12&3&0&3&[48] $24^1, 4^1, 2^1$ [*]\\
antip(trunc($K_4$))&12&24&12&1&0&1&[24] $3^1$\\
Icosahedron&12&30&18&12&5&7&[120] $3^1,2^4, 1^7$\\
Octahedron $\times K_2$&12&30&18&1&0&1&[96] $12^1$\\
$C_{12}(5,6)$&12&30&18&8&5&3&[768] $12^1, 6^1, 4^1, 2^4, 1^1$\\
$C_{12}(2,5+)$&12&30&18&1&1&0&[12] $12^1$\\
$C_{12}(4,5^+)$&12&30&18&1&1&0&[12] $12^1$\\
$C_{12}(2,3)$ [$\Delta$]&12&36&24&1&1&0&[24] $24^1$\\
$C_{12}(2,5)$ [$\Delta$]&12&36&24&1&0&1&[144] $72^1$\\
$C_{12}(3,4)$ [$\Delta$]&12&36&24&1&1&0&[24] $24^1$\\
$C_{12}(4,5)$ [$\Delta$]&12&36&24&1&0&1&[48] $24^1$\\
$Q_4=C_4\times C_4$&16&32&16&1&0&1&[384] $64^1$\\
 \end{tabular}
%  \caption{TableCaption}
%  \label{tab:label}
\end{table}

\end{document}